\documentclass[11pt,reqno]{amsart}
\usepackage[hypertex]{hyperref}

\usepackage{xcolor}
\usepackage{amsmath,amsthm}
\usepackage{amssymb}
\usepackage{euscript}
\usepackage[frame,ps,matrix,arrow,curve,rotate,all,2cell,tips]{xy}
\usepackage{epic,eepic}

\numberwithin{equation}{subsection}

\newcommand{\sqsp}{\renewcommand{\baselinestretch}{1.25}\tiny\normalsize}

\newtheorem{theorem}[subsubsection]{Theorem}
\newtheorem{lemma}[subsubsection]{Lemma}

\newtheorem{proposition}[subsubsection]{Proposition}
\newtheorem{corollary}[subsubsection]{Corollary}

\theoremstyle{definition}
\newtheorem{definition}[subsubsection]{Definition}
\newtheorem{example}[subsubsection]{Example}
\newtheorem{remark}[subsubsection]{Remark}

\newcommand{\frakC}{\mathfrak{C}}

\newcommand{\sfE}{\mathsf{E}}
\newcommand{\sfI}{\mathsf{I}}

\newcommand{\sfO}{\mathsf{O}}
\newcommand{\sfP}{\mathsf{P}}
\newcommand{\pbar}{\overline{\sfP}}
\newcommand{\sfQ}{\mathsf{Q}}

\newcommand{\Se}{\mathsf{Se}}
\newcommand{\sfT}{\mathsf{T}}

\newcommand{\bSe}{\mathbf{Se}}
\newcommand{\set}{\mathbf{Set}}
\newcommand{\bk}{\mathbf{k}}
\newcommand{\bP}{\mathbf{P}}
\newcommand{\bone}{\mathbf{1}}
\newcommand{\bM}{\mathbf{M}}

\newcommand{\ttone}{\mathtt{1}}

\newcommand{\catc}{\mathcal{C}}

\newcommand{\cate}{\mathcal{E}}

\newcommand{\catm}{\mathcal{M}}

\newcommand{\catp}{\mathcal{P}}

\newcommand{\catrcf}{\mathcal{RCF}}

\newcommand{\ua}{\underline{a}}
\newcommand{\ub}{\underline{b}}
\newcommand{\uc}{\underline{c}}
\newcommand{\ud}{\underline{d}}

\newcommand{\uy}{\underline{y}}
\newcommand{\ualpha}{\underline{\alpha}}
\newcommand{\ubeta}{\underline{\beta}}
\newcommand{\ugamma}{\underline{\gamma}}
\newcommand{\udelta}{\underline{\delta}}
\newcommand{\uepsilon}{\underline{\varepsilon}}

\newcommand{\xbar}{\overline{x}}
\newcommand{\zbar}{\overline{z}}

\newcommand{\Xtilde}{\widetilde{X}}

\newcommand{\boxv}{\boxtimes_v}
\newcommand{\boxh}{\boxtimes_h}

\newcommand{\sigmaec}{\mathbf{\Sigma}^\frakC_\cate} 
\newcommand{\sigmac}{\mathbf{\Sigma}^\frakC}
\newcommand{\prop}{\mathbf{PROP}} 
\newcommand{\propc}{\mathbf{PROP}^{\frakC}}

\newcommand{\operadc}{\mathbf{Operad}^\frakC}
\newcommand{\mon}{\mathbf{Mon}} 
\newcommand{\alg}{\mathbf{Alg}} 
\newcommand{\algn}{\mathbf{Alg}^n} 
\newcommand{\propset}{\set^{\bP(\sfP)}} 
\newcommand{\propsetQ}{\set^{\bP(\sfQ)}} 
\newcommand{\kmodule}{\mathbf{Mod}(\bk)}  
\newcommand{\propmodule}{\kmodule^{\bP(\sfP)}}  
\newcommand{\stack}{\mathbf{Stack}} 

\DeclareMathOperator{\elt}{elt}

\DeclareMathOperator*{\colim}{colim}

\begin{document}

\title[Higher dimensional algebras via colored PROPs]{Higher dimensional algebras via colored PROPs}
\author{Donald Yau}

\begin{abstract}
Starting from any unital colored PROP $\sfP$, we define a category $\bP(\sfP)$ of shapes called $\sfP$-propertopes.  Presheaves on $\bP(\sfP)$ are called $\sfP$-propertopic sets.  For $0 \leq n \leq \infty$ we define and study $n$-time categorified $\sfP$-algebras as $\sfP$-propertopic sets with some lifting properties.  Taking appropriate PROPs $\sfP$, we obtain higher categorical versions of polycategories, $2$-fold monoidal categories, topological quantum field theories, and so on.
\end{abstract}



\address{Department of Mathematics\\
    The Ohio State University at Newark\\
    1179 University Drive\\
    Newark, OH 43055}
\email{dyau@math.ohio-state.edu}

\date{\today}
\maketitle

\tableofcontents

\section{Introduction}

\sqsp


The purpose of this paper is to study higher dimensional versions of algebraic structures.  By \emph{higher dimensional algebras} we mean higher categorical analogues of algebras.  The process of going from algebras to higher dimensional algebras is called \emph{categorification}.  For example, a set is the simplest kind of algebra, one in which there is no further structure.  A category is a $1$-time categorification of a set.  Likewise, a monoidal category is a $1$-time categorification of a monoid.  Roughly speaking, higher category theory is the study of $n$-time categorified sets, monoids, commutative monoids, etc.  We aim to study $n$-time categorified algebras in general for $0 \leq n \leq \infty$.

We have several specific goals for this paper:
\begin{enumerate}
\item
We develop a concept of higher dimensional algebra for a sufficiently general class of algebras.  In particular, we consider not only algebras with multiple inputs and one output (e.g., a monoid), but also those that have multiple inputs and multiple outputs.  Such algebraic structures include bialgebras, polycategories (Example ~\ref{ex:polycategory}), and the Segal category (Example ~\ref{ex:CFT}).
\item
Our definition of higher dimensional algebra is sufficiently simple and intuitive so that they can be readily used in applications, including topological field theories in mathematical physics and higher category theory itself.
\item
We organize the coherence laws of our higher dimensional algebras in a systematic and trackable way.
\end{enumerate}


To be more precise, our algebras are algebras over an arbitrary colored PROP.  A PROP \cite{maclane1}, short for \textbf{pro}duct and \textbf{p}ermutation category, is a very general algebraic machinery that can describe algebraic structures with multiple inputs and multiple outputs.  PROPs have long been used in algebraic topology to study loop spaces \cite{adams,bv}.  Most familiar types of algebras -- e.g., associative, Lie, commutative, Gerstenhaber, and Hopf -- are algebras over certain PROPs.  Moreover, the Segal PROP \cite{segalA,segalB,segalC}, made up of Riemann surfaces with boundary holes, in topological field theories is another example of a PROP (Example ~\ref{ex:CFT}).

A \emph{colored} PROP is a generalization of a PROP that can encode even more general types of algebraic structures.  For example, diagrams of algebras over a PROP, modules over an algebra, module-algebras, and its variants (module-coalgebras, comodule-algebras, and comodule-coalgebras) are algebras over certain colored PROPs.  The Segal PROP has an obvious colored analogue in which the boundary holes in the Riemann surfaces are allowed to have different circumferences.  Closely related is the colored PROP $\catrcf(g)$ \cite{chataur,cg} in string topology that is built from spaces of reduced metric Sullivan chord diagrams with genus $g$.  Multi-categories (a.k.a.\ colored operads) are to operads what colored PROPs are to PROPs.  In the simplest case, the set of colors $\frakC$ is the one-element set, and $\{*\}$-colored PROPs are just PROPs.

Categorification involves a level-shifting process.  For example, sets and functions are replaced by categories and functors.  Extra structures on sets are replaced by functors on categories.  The equations satisfied by these extra structures are replaced by natural transformations, which satisfy their own coherence laws.  For example, in a monoidal category, the monoidal product $\otimes$ is not associative, but there is an associator natural isomorphism.  The associator is required to satisfy a pentagon axiom \cite[Chapter VII]{maclane2}.  See \cite{baez1,bd2} for an introduction to categorification.


We achieve the level-shifting effect of categorification by the so-called \emph{slice construction}.  This construction was pioneered by Baez and Dolan \cite{bd1}, in which the construction was defined for colored operads.  The slice PROP, which we also call \emph{higher PROP}, $\sfP^+$ of a colored PROP $\sfP$ has the following properties:
\begin{enumerate}
\item
The set of colors in $\sfP^+$ is the set of operations ($=$ elements) in $\sfP$.
\item
The operations in $\sfP^+$ are the reduction laws in $\sfP$.
\item
The reduction laws in $\sfP^+$ are the ways of combining reduction laws in $\sfP$ to obtain other reduction laws.
\end{enumerate}
What we mean by a \emph{reduction law} here is an equation stating that the composite of some elements is equal to some element.  The slice construction can be iterated, giving rise to the higher PROPs $\sfP^{n+} = (\sfP^{(n-1)+})^+$ for $n \geq 1$ with $\sfP^{0+} = \sfP$.


The higher PROPs are used to define higher dimensional $\sfP$-algebras, which we call \emph{weak-$n$ $\sfP$-algebras}, using some ideas from homotopy theory.  The elements in the higher PROP $\sfP^{n+}$ are called \emph{$n$-dimensional $\sfP$-propertopes}.  The name \emph{propertope} is an abbreviation for \textbf{pro}duct, \textbf{per}mutation, and poly\textbf{tope}.  We construct a category $\bP(\sfP)$ consisting of the $\sfP$-propertopes of all dimensions, in which morphisms are generated by certain \emph{face maps}.  These face maps satisfy some consistency conditions that are analogous to the simplicial identities.  By analogy with simplicial sets, we look at the presheaf category $\propset$, whose objects are called \emph{$\sfP$-propertopic sets}.

In a $\sfP$-propertopic set, there are elements called \emph{$k$-cells} with shapes corresponding to $k$-dimensional $\sfP$-propertopes for $k \geq 0$, \emph{horns}, and \emph{boundaries}.  For $n$ in the range $0 \leq n \leq \infty$, we define a \emph{weak-$n$ $\sfP$-algebra} as a $\sfP$-propertopic set in which certain horns and boundaries have (unique) extensions to cells, called \emph{fillings}.  Under this analogy with simplicial sets, weak-$n$ $\sfP$-algebras are the $\sfP$-propertopic analogues of homotopy $n$-types when $n < \infty$.  When $n = \infty$, weak-$\omega$ $\sfP$-algebras are the $\sfP$-propertopic analogues of Kan complexes.



In a weak-$n$ $\sfP$-algebra, the $k$-cells play the roles of $k$-morphisms in higher category theory.  For $0 \leq k < n$, $k$-cells can be composed via $(k+1)$-cells using the horn-filling property of a weak-$n$ $\sfP$-algebra.  These compositions are in general \emph{not} a function, but a multi-valued function.  On the other hand,  composition of the $n$-cells (if $n < \infty$) is an honest operation, which comes from the unique horn-filling and boundary-filling properties in a weak-$n$ $\sfP$-algebra.  In fact, all the higher cells together (i.e., $m$-cells for $m \geq n$) form a $\sfP^{n+}$-algebra.


Since we start with a colored PROP $\sfP$, the (multi-valued) compositions in a weak-$n$ $\sfP$-algebra have multiple inputs and multiple outputs.  By allowing compositions to have multiple inputs and multiple outputs, our theory of higher dimensional $\sfP$-algebras should be particularly suitable for applications in topological field theories, logic, and computer science.  Some such applications in topological field theories are briefly discussed in \S\ref{subsec:hdtft}.


In higher category theory, one major issue is to organize the coherence laws of the higher morphisms.  In our theory of weak-$n$ $\sfP$-algebras, \emph{coherence laws are treated as compositions}.  Coherence laws about the $k$-cells are the relations among the $k$-cells.  Such relations are exactly the ways in which the $k$-cells are composed via the $(k+1)$-cells.  The relations among these $(k+1)$-cells are the ways in which they are composed via the $(k+2)$-cells, and so forth.  This characteristic of our weak-$n$ $\sfP$-algebras is similar to Leinster's definition of a weak $n$-category \cite[Chapter 9]{leinster}, in which coherence laws and compositions are also treated as the same concept called \emph{contraction}.

\subsection{Organization of the paper}

Here is a brief summary of the remaining sections in this paper.  There is a summary at the beginning of each section as well.

In \S\ref{sec:PROP} we discuss some basics of colored PROPs and their relationships with colored operads.  Some examples of colored PROPs are given at the end of that section.  In \S\ref{sec:higherPROP} we construct the slice PROP $\sfP^+$ of a colored PROP $\sfP$ and discuss its properties.  In \S\ref{sec:propertopes} we construct the category of $\sfP$-propertopes and discuss the combinatorics of drawing $\sfP$-propertopes.  Then we discuss $\sfP$-propertopic sets and the concepts of cells, horns, boundaries, and fibrations.  In \S\ref{sec:weakalgebra} we define weak-$n$ $\sfP$-algebras and study their structures.  In \S\ref{sec;somehighercat} we discuss several specific types of weak-$n$ $\sfP$-algebras that are relevant in some applications, including higher category theory, higher topological field theories, and higher algebraic geometry.



The only necessary prerequisite to read this paper is some basic knowledge of category theory.  We do not assume any knowledge of higher (strict or weak) categories, except for motivational discussions.  For the reader who is interested in weak $n$-categories (instead of weak-$n$ $\sfP$-algebras in general), we recommend \cite{leinster01,leinster}, in which various definitions of weak $n$-categories are described, including \cite{bd1,batanin,hmp1,hmp2,hmp3,joyal,may01,penon,street,simpson,tam}.

Although our approach to weak-$n$ $\sfP$-algebras is a generalization of \cite{bd1}, we do not assume knowledge of the Baez-Dolan opetopes and the slice construction for colored operads.  We will describe colored PROPs, slice PROPs, $\sfP$-propertopes, and so forth from scratch.






\section{Colored PROPs and algebras}
\label{sec:PROP}

In \S\ref{subsec:coloredbimodules} we introduce colored $\Sigma$-bimodules and colored PROPs. In \S\ref{subsec:coloredoperad} the adjunction between colored operads and colored PROPs is constructed (Theorem ~\ref{operadpropadj}), and the consequence on algebras is discussed (Corollary ~\ref{Opropalgebra}).  In \S\ref{subsec:PROPex} several examples of colored PROPs and their algebras are discussed.  These examples include polycategories, bicommutative bimonoids, and the Segal PROP.

\subsection{Setting}

We work over the base category $\set$ of sets and functions.  The materials in this section are actually valid in a closed symmetric monoidal category $(\cate, \otimes, \bone)$ with all small limits and colimits \cite[Ch.VII and Ch.XI]{maclane2}.  For example, one can easily adapt the discussion in this section to the categories of $\bk$-modules (where $\bk$ is a field of characteristic $0$), chain complexes of $\bk$-modules, simplicial sets, topological spaces, symmetric spectra \cite{hss}, and $S$-modules \cite{ekmm}.

If $\catc$ is a category and $X$ and $Y$ are objects in $\catc$, then $\catc(X,Y)$ denotes the set of morphisms from $X$ to $Y$ in $\catc$.

\subsection{Colored $\Sigma$-bimodules and colored PROPs}
\label{subsec:coloredbimodules}

Colored PROPs are colored $\Sigma$-bimodules equipped with a horizontal composition and a compatible vertical composition (Definition ~\ref{def:Ccoloredprop}).  We first discuss colored $\Sigma$-bimodules.

Fix a non-empty set $\frakC$ once and for all.  The elements in $\frakC$ are called \textbf{colors}.  Our PROPs have a base set of colors $\frakC$.  The simplest case is when $\frakC = \{*\}$, which gives $1$-colored PROPs.

Let $\mathcal{P}(\frakC)$ denote the category whose objects, called \textbf{profiles} or \textbf{$\frakC$-profiles}, are finite non-empty sequences of colors.  If
\[
\ud = (d_1, \ldots , d_m) \in \catp(\frakC),
\]
then we write $|\ud| = m$.  Our convention is to use a normal alphabet, possibly with a subscript (e.g., $d_1$) to denote a color and to use an underlined alphabet (e.g., $\ud$) to denote an object in $\catp(\frakC)$.

Permutations $\sigma \in \Sigma_{|\ud|}$ act on such a profile $\ud$ from the left by permuting the $|\ud|$ colors.  Given two profiles $\uc = (c_1, \ldots , c_n)$ and $\ud = (d_1, \ldots , d_m)$, a \textbf{morphism}
\[
\sigma \colon \uc \to \ud \in \mathcal{P}(\frakC)
\]
is a permutation $\sigma \in \Sigma_{|\uc|}$ such that
\[
\sigma(\uc) = \ud.
\]
Such a morphism exists if and only if $\ud$ is in the orbit of $\uc$.  Of course, if such a morphism exists, then $|\uc| = |\ud|$.  The \textbf{orbit type} of a $\frakC$-profile $\uc$ is denoted by $[\uc]$.

To emphasize that the permutations act on the profiles from the left, we will also write $\catp(\frakC)$ as $\catp_l(\frakC)$.  If we let the permutations act on the profiles from the right instead, then we get an equivalent category $\catp_r(\frakC)$.

Given profiles as above, we define
\begin{equation}
\label{eq:concat}
(\uc, \ud) = (c_1, \ldots , c_n, d_1, \ldots , d_m) \in \catp(\frakC),
\end{equation}
the concatenation of $\uc$ and $\ud$.

\begin{definition}
The category of \textbf{$\frakC$-colored $\Sigma$-bimodules} over $\set$ is defined to be the diagram category $\set^{\catp_l(\frakC) \times \catp_r(\frakC)}$.  To simplify the typography, we will write $\sigmac$ for $\set^{\catp_l(\frakC) \times \catp_r(\frakC)}$.
\end{definition}

In other words, a $\frakC$-colored $\Sigma$-bimodule is a functor
\[
\sfP \colon \catp_l(\frakC) \times \catp_r(\frakC) \to \set,
\]
and a morphism of $\frakC$-colored $\Sigma$-bimodules is a natural transformation of such functors.  Unpacking this definition, one obtains the following concrete description of a $\frakC$-colored $\Sigma$-bimodule.

\begin{proposition}
\label{prop:sigmabimodule}
A $\frakC$-colored $\Sigma$-bimodule $\sfP$ consists of exactly the following data:
\begin{enumerate}
\item
For any $\frakC$-profiles $\ud \in \catp_l(\frakC)$ and $\uc \in \catp_r(\frakC)$, it has a set
\[
\sfP\binom{\underline{d}}{\underline{c}} = \sfP\binom{d_1, \ldots , d_m}{c_1 , \ldots , c_n}.
\]
\item
For any permutations $\sigma \in \Sigma_{|\ud|}$ and $\tau \in \Sigma_{|\uc|}$, it has a map
\begin{equation}
\label{eq:sigmatauaction}
(\sigma; \tau) \colon \sfP\binom{\ud}{\uc} \to \sfP\binom{\sigma\ud}{\uc \tau} \in \set
\end{equation}
such that:
\begin{enumerate}
\item $(1;1)$ is the identity morphism,
\item $(\sigma' \sigma; \tau\tau') = (\sigma';\tau') \circ (\sigma;\tau)$, and
\item $(1;\tau) \circ (\sigma; 1) = (\sigma; \tau) = (\sigma; 1) \circ (1; \tau)$.
\end{enumerate}
\end{enumerate}
Moreover, a morphism $f \colon \sfP \to \sfQ$ of $\frakC$-colored $\Sigma$-bimodules consists of color-preserving maps
\[
\left\{\sfP\binom{\underline{d}}{\underline{c}} \xrightarrow{f} \sfQ\binom{\underline{d}}{\underline{c}} \colon (\ud; \uc) \in \catp_l(\frakC) \times \catp_r(\frakC) \right\}
\]
such that the square
\[
\SelectTips{cm}{10}
\xymatrix{
\sfP\dbinom{\ud}{\uc} \ar[r]^-f  \ar[d]_-{(\sigma;\tau)} & \sfQ\dbinom{\ud}{\uc} \ar[d]^-{(\sigma;\tau)} \\
\sfP\dbinom{\sigma\ud}{\uc\tau}\ar[r]^-f & \sfQ\dbinom{\sigma\ud}{\uc\tau}
}
\]
is commutative for any permutations $\sigma \in \Sigma_{|\ud|}$ and $\tau \in \Sigma_{|\uc|}$.
\end{proposition}

One should think of the set $\sfP\binom{\underline{d}}{\underline{c}}$ as a space of operations with $|\uc| = n$ inputs and $|\ud| = m$ outputs.   The $n$ inputs have colors $c_1, \ldots , c_n$, and the $m$ outputs have colors $d_1, \ldots , d_m$.

\begin{definition}
Let $\sfP$ be a $\frakC$-colored $\Sigma$-bimodule, and let $m$ and $n$ be positive integers.  Define the set
\begin{equation}
\label{eq:Edecomp}
\sfP(m,n)
= \colim \sfP\binom{d_1, \ldots , d_m}{c_1 , \ldots , c_n}
= \colim \sfP\binom{\underline{d}}{\underline{c}},
\end{equation}
where the colimit is taken over all $\frakC$-profiles $\ud$ and $\uc$ with $|\ud| = m$ and $|\uc| = n$ using the maps \eqref{eq:sigmatauaction}.  The object $\sfP(m,n)$ is said to have \textbf{biarity $(m,n)$}, and $\sfP\binom{\underline{d}}{\underline{c}}$ is called a \textbf{component} of $\sfP(m,n)$.
\end{definition}

The following result is an immediate consequence of Proposition ~\ref{prop:sigmabimodule}.

\begin{corollary}
Let $\sfP$ be a $\frakC$-colored $\Sigma$-bimodule, and let $m$ and $n$ be positive integers.  Then the set $\sfP(m,n)$ admits a left $\Sigma_m$-action and a right $\Sigma_n$-action such that the two actions commute.
\end{corollary}

The category $\sigmac$ of $\frakC$-colored $\Sigma$-bimodules can be decomposed into smaller pieces according to the orbit types of $\frakC$-profiles.  To describe this decomposition, we need the following smaller indexing categories.

\begin{definition}
\label{sigmab}
Let $\ub = (b_1, \ldots , b_k)$ be a $\frakC$-profile.  Define the category $\Sigma_{\ub}$ whose objects are the $\frakC$-profiles
\[
\tau\ub = \left(b_{\tau(1)}, \ldots , b_{\tau(k)}\right) \in \catp(\frakC)
\]
obtained from $\ub$ by permutations $\tau \in \Sigma_k$.  Given two (possibly equal) objects $\tau \ub$ and $\tau'\ub$ in $\Sigma_{\ub}$, a morphism
\[
\tau'' \colon \tau \ub \to \tau' \ub
\]
is a permutation in $\Sigma_k$ such that
\[
\tau''\tau \ub = \tau' \ub
\]
as $\frakC$-profiles.
\end{definition}

Notice that when we write $\tau\ub$ as an object in $\Sigma_{\ub}$, the permutation $\tau$ is not necessarily unique.  Indeed, $\tau'\ub$ is the same object as $\tau\ub$ if and only if they are equal as $\frakC$-profiles.

The category $\Sigma_{\ub}$ is a groupoid, i.e., every morphism in it is invertible.  Moreover, this groupoid is connected.  In other words, given any two objects $\tau \ub$ and $\tau'\ub$ in $\Sigma_{\ub}$, there is at least one morphism
\[
\tau'\tau^{-1} \colon \tau \ub \to \tau' \ub.
\]
There are other morphisms $\tau \ub \to \tau' \ub$ if and only if $\ub$ has repeated colors.  The set of objects in $\Sigma_{\ub}$ is exactly what constitutes the orbit type of $\ub$.  A morphism in $\Sigma_{\ub}$ is a way to permute from one representative in the orbit type of $\ub$ to another representative.  It is easy to see that there is an isomorphism
\[
\Sigma_{\ub} \cong \Sigma_{\tau\ub}
\]
of groupoids for any $\tau \in \Sigma_{|\ub|}$.

\begin{example}
\label{ex:Sigmab}
In the one-colored case, i.e., $\frakC = \{*\}$, a $\frakC$-profile $\ub$ is uniquely determined by its length $|\ub| = k$.  In this case, there is precisely one object
\[
\ub = \underbrace{(*, \ldots , *)}_{k \, *'s}
\]
in the category $\Sigma_{\ub}$, since $\ub$ is unchanged by any permutation in $\Sigma_k$.  For the same reason, the set of morphisms $\ub \to \ub$ is exactly $\Sigma_k$.  In other words, in the one-colored case, $\Sigma_{\ub}$ is the permutation group $\Sigma_{|\ub|}$, regarded as a category with one object.\qed
\end{example}

\begin{example}
\label{ex:Sigmab'}
In the other extreme, suppose that $\ub = \left(b_1, \ldots , b_k\right)$ consists of distinct colors, i.e., $b_i \not= b_j$ if $i \not= j$.  There are now $k!$ different permutations of $\ub$, one for each $\tau \in \Sigma_k$.  So there are $k!$ objects in $\Sigma_{\ub}$.  Given two objects $\tau\ub$ and $\tau'\ub$ in $\Sigma_{\ub}$, there is a unique way to permute $\tau\ub$ to get $\tau'\ub$, namely,
\[
(\tau'\tau^{-1})\tau\ub = \tau'\ub.
\]
In other words, given any two objects $\tau\ub$ and $\tau'\ub$ in $\Sigma_{\ub}$, there is a unique morphism $\tau'\tau^{-1} \colon \tau\ub \to \tau'\ub$.\qed
\end{example}

To decompose $\frakC$-colored $\Sigma$-bimodules, we actually need a pair of $\frakC$-profiles at a time.  So we introduce the following groupoid.

\begin{definition}
Given any pair of $\frakC$-profiles $\ud$ and $\uc$, define the groupoid
\[
\Sigma_{\ud; \uc} = \Sigma_{\ud} \times \Sigma_{\uc}^{op},
\]
where $\Sigma_{\ud}$ and $\Sigma_{\uc}$ are the groupoids defined in Definition ~\ref{sigmab}.
\end{definition}

If $\ud = (d_1, \ldots , d_m)$ and $\uc = (c_1, \ldots , c_n)$, then we write the objects in $\Sigma_{\ud; \uc}$ as pairs
\[
\binom{\sigma\ud}{\uc\tau} = \binom{d_{\sigma(1)}, \ldots , d_{\sigma(m)}}{c_{\tau^{-1}(1)}, \ldots , c_{\tau^{-1}(n)}}
\]
for $\sigma \in \Sigma_m$ and $\tau \in \Sigma_n$.

\begin{example}
Continuing Example ~\ref{ex:Sigmab}, if $\frakC = \{*\}$, then $\Sigma_{\ud; \uc}$ is the product group $\Sigma_{|\ud|} \times \Sigma_{|\uc|}^{op}$, considered as a category with one object.\qed
\end{example}

\begin{example}
On the other hand, suppose that each of $\ud$ and $\uc$ consists of distinct colors, as in Example ~\ref{ex:Sigmab'}.  Then there are $|\ud|! |\uc|!$ objects in $\Sigma_{\ud; \uc}$.  There is a unique morphism from any object in $\Sigma_{\ud; \uc}$ to any other object.\qed
\end{example}

Given any $\frakC$-profile $\ud$, recall that we denote by $[\ud]$ the orbit type of $\ud$ under permutations in $\Sigma_{|\ud|}$.  The following result is the decomposition of $\frakC$-colored $\Sigma$-bimodules that we have been referring to.

\begin{proposition}
\label{decomposition}
There is a canonical isomorphism
\[
\sigmac \cong \prod_{[\ud], [\uc]} \set^{\Sigma_{\ud; \uc}},
\quad \sfP \mapsto \left\{\sfP\binom{[\ud]}{[\uc]}\right\}
\]
of categories, in which the product runs over all the pairs of orbit types of $\frakC$-profiles.
\end{proposition}

\begin{proof}
First we should clarify the meaning of $\set^{\Sigma_{\ud; \uc}}$.  For each orbit type $[\ud]$, we choose a representative $\ud$.  Such choices of representatives are then used to form the groupoids $\Sigma_{\ud}$ and the diagram categories $\set^{\Sigma_{\ud; \uc}} = \set^{\Sigma_{\ud} \times \Sigma_{\uc}^{op}}$.

Now given a $\frakC$-colored $\Sigma$-bimodule $\sfP$ over $\set$, we restrict to a pair $[\ud]$ and $[\uc]$ of orbit types of $\frakC$-profiles.  The restricted diagram $\sfP\binom{[\ud]}{[\uc]}$ has objects $\sfP\binom{\sigma\ud}{\uc\tau}$ for $\sigma \in \Sigma_{|\ud|}$ and $\tau \in \Sigma_{|\uc|}$.  Each map
\[
\sfP\binom{\sigma\ud}{\uc\tau} \to \sfP\binom{\sigma'\ud}{\uc\tau'}
\]
in $\sfP$ corresponds to a unique morphism
\[
\binom{\sigma\ud}{\uc\tau} \to \binom{\sigma'\ud}{\uc\tau'}
\]
in $\Sigma_{\ud;\uc}$.  So the restricted diagram is actually an object in the diagram category $\set^{\Sigma_{\ud; \uc}}$.  Since $\sfP$ is uniquely determined by such restricted diagrams, the result follows.
\end{proof}

\begin{example}
If $\frakC = \{*\}$, then the decomposition in Proposition ~\ref{decomposition} becomes
\[
\sigmac \cong \prod_{m, n \geq 1} \set^{\Sigma_m \times \Sigma_n^{op}}.
\]
A object in the diagram category $\set^{\Sigma_m \times \Sigma_n^{op}}$ is simply a set $\sfP(m,n)$ with a left $\Sigma_m$-action and a right $\Sigma_n$-action that commute with each other.\qed
\end{example}

We now define $\frakC$-colored PROPs.

\begin{definition}
\label{def:Ccoloredprop}
A \textbf{unital $\frakC$-colored PROP} $\sfP$ consists of a $\frakC$-colored $\Sigma$-bimodule $\sfP$ with the following additional structures:
\begin{enumerate}
\item
For any $\frakC$-profiles $\ub$, $\uc$, and $\ud$, it has a \textbf{vertical composition}
\begin{equation}
\label{verticalcomp}
\sfP\binom{\ud}{\ub} \times \sfP\binom{\ub}{\uc} \xrightarrow{\circ} \sfP\binom{\ud}{\uc}
\end{equation}
that is associative and bi-equivariant.  The bi-equivariance of $\circ$ means that the diagram
\begin{equation}
\label{verticalequivariance}
\SelectTips{cm}{10}
\xymatrix{
\sfP\dbinom{\ud}{\ub\tau^{-1}} \times \sfP\dbinom{\tau\ub}{\uc} \ar@{=}[r] & \sfP\dbinom{\ud}{\ub\tau^{-1}} \times \sfP\dbinom{\tau\ub}{\uc} \ar[d]^-{\circ} \\
\sfP\dbinom{\ud}{\ub} \times \sfP\dbinom{\ub}{\uc} \ar[u]^-{(1;\tau^{-1}) \times (\tau;1)} \ar[r]^-{\circ} \ar[d]_-{(\sigma;1) \times (1;\mu)} & \sfP\dbinom{\ud}{\uc} \ar[d]^-{(\sigma;\mu)} \\
\sfP\dbinom{\sigma\ud}{\ub} \times \sfP\dbinom{\ub}{\uc\mu} \ar[r]^-{\circ} & \sfP\dbinom{\sigma\ud}{\uc\mu}
}
\end{equation}
is commutative for any permutations $\sigma \in \Sigma_{|\ud|}$, $\mu \in \Sigma_{|\uc|}$, and $\tau \in \Sigma_{|\ub|}$.
\item
For any $\frakC$-profiles $\ud_1$, $\ud_2$, $\uc_1$, and $\uc_2$, it has a \textbf{horizontal composition}
\begin{equation}
\label{horizontalcomposition2}
\sfP\binom{\ud_1}{\uc_1} \times \sfP\binom{\ud_2}{\uc_2} \xrightarrow{\otimes} \sfP\binom{\ud_1,\ud_2}{\uc_1,\uc_2}
\end{equation}
that is associative and bi-equivariant.  The bi-equivariance of $\otimes$ means that the square
\begin{equation}
\label{horizontalequivariance}
\SelectTips{cm}{10}
\xymatrix{
\sfP\dbinom{\ud_1}{\uc_1} \times \sfP\dbinom{\ud_2}{\uc_2} \ar[r]^-{\otimes} \ar[d]_-{(\sigma_1;\tau_1) \times (\sigma_2;\tau_2)} & \sfP\dbinom{\ud_1,\ud_2}{\uc_1,\uc_2} \ar[d]^-{(\sigma_1 \times \sigma_2; \tau_1 \times \tau_2)}\\
\sfP\dbinom{\sigma_1\ud_1}{\uc_1\tau_1} \times \sfP\dbinom{\sigma_2\ud_2}{\uc_2\tau_2} \ar[r]^-{\otimes} & \sfP\dbinom{\sigma_1\ud_1,\sigma_2\ud_2}{\uc_1\tau_1,\uc_2\tau_2}
}
\end{equation}
is commutative for any permutations $\sigma_i \in \Sigma_{|\ud_i|}$ and $\tau_i \in \Sigma_{|\uc_i|}$.
\item
For each color $c \in \frakC$, it has a \textbf{$c$-colored unit}
\[
\mathtt{1}_c \in \sfP\binom{c}{c}
\]
such that for $\uc = (c_1, \ldots , c_n)$, the horizontal composite
\[
\mathtt{1}_{c_1} \otimes \cdots \otimes \mathtt{1}_{c_n} \in \sfP\binom{\uc}{\uc}
\]
acts as the two-sided unit for the vertical composition.
\end{enumerate}
Moreover, the vertical and horizontal compositions are required to satisfy the \textbf{interchange rule}, which says that the diagram
\begin{equation}
\label{interchangerule}
\SelectTips{cm}{10}
\xymatrix{
\left[\sfP\binom{\ud_1}{\ub_1} \times \sfP\binom{\ud_2}{\ub_2}\right] \times \left[\sfP\binom{\ub_1}{\uc_1} \times \sfP\binom{\ub_2}{\uc_2}\right] \ar[r]^-{\text{switch}}_-{\cong} \ar[d]_-{(\otimes, \otimes)} &
\left[\sfP\binom{\ud_1}{\ub_1} \times \sfP\binom{\ub_1}{\uc_1}\right] \times \left[\sfP\binom{\ud_2}{\ub_2} \times \sfP\binom{\ub_2}{\uc_2}\right] \ar[d]^-{(\circ, \circ)}\\
\sfP\binom{\ud_1,\ud_2}{\ub_1,\ub_2} \times \sfP\binom{\ub_1,\ub_2}{\uc_1,\uc_2} \ar[d]_-{\circ} & \sfP\binom{\ud_1}{\uc_1} \times \sfP\binom{\ud_2}{\uc_2} \ar[d]^-{\otimes}\\
\sfP\binom{\ud_1,\ud_2}{\uc_1,\uc_2} \ar@{=}[r] & \sfP\binom{\ud_1,\ud_2}{\uc_1,\uc_2}.
}
\end{equation}
is commutative for any $\frakC$-profiles $\ub_i$, $\uc_i$, and $\ud_i$ $(i = 1, 2)$.

A \textbf{morphism} of unital $\frakC$-colored PROPs is a morphism of the underlying $\frakC$-colored $\Sigma$-bimodules that commutes with the horizontal and the vertical compositions and preserves the $c$-color unit for each $c \in \frakC$.
\end{definition}

One obtains the notion of a \textbf{non-unital} $\frakC$-colored PROP by omitting the requirements about the $c$-colored units.  The category of non-unital $\frakC$-colored PROPs is denoted by $\propc$.  The category of unital $\frakC$-colored PROPs is a subcategory of $\propc$.

If $\frakC = \{*\}$ is the one-element set, then we say \textbf{$1$-colored PROPs} or just \textbf{PROPs} for $\{*\}$-colored PROPs.

\begin{remark}
A unital $\frakC$-colored PROP in a symmetric monoidal category $\cate$ can also be defined as a strict monoidal category $(\sfP, \odot)$ enriched over $\cate$.  The objects in $(\sfP, \odot)$ are the $\frakC$-profiles, and the monoidal product $\odot$ is concatenation of $\frakC$-profiles.  The morphism object $\sfP(\ud,\uc)$ is what we write as $\sfP\binom{\ud}{\uc}$ above.  Moreover, given any permutations $\sigma \in \Sigma_{|\ud|}$ and $\tau \in \Sigma_{|\uc|}$, it is required that there be an associated map
\[
(\sigma;\tau) \colon \sfP(\ud,\uc) \to \sfP(\sigma\ud,\uc\tau) \in \cate
\]
on the morphism objects such that some obvious axioms are satisfied.  In this formulation, the horizontal composition is induced on the morphism objects by the monoidal product $\odot$ and the enrichment over $\cate$.  The vertical composition is the categorical composition in the category $\sfP$.  This generalizes what is known in the $1$-colored case \cite{maclane1,markl06}.
\end{remark}

\begin{remark}
There is another conceptual description of (non-unital) $\frakC$-colored PROPs in a symmetric monoidal category $\cate$ with a zero object.  In this setting, non-unital $\frakC$-colored PROPs are \emph{$\boxv$-monoidal $\boxh$-monoids}, where $\boxv$ is a monoidal product on the category $\sigmaec$ of $\frakC$-colored $\Sigma$-bimodules in $\cate$.  A monoid in $(\sigmaec,\boxv)$ is a $\frakC$-colored $\Sigma$-bimodule equipped with a vertical composition.  There is a monoidal product $\boxh$ on the category $\mon(\sigmaec,\boxv)$ of monoids in $(\sigmaec,\boxv)$.  The monoids in $\left(\mon(\sigmaec,\boxv),\boxh\right)$ are exactly the non-unital $\frakC$-colored PROPs.  This description of $\frakC$-colored PROPs as $\boxv$-monoidal $\boxh$-monoids is analogous to the description of operads as monoids in the category of $\Sigma$-objects. The reader is referred to \cite{jy} for detailed discussion of colored PROPs from this view point.
\end{remark}

Before we talk about algebras over a $\frakC$-colored PROP $\sfP$, let us first spell out the colored endomorphism PROP through which $\sfP$-algebras are defined.  If $X$ and $Y$ are sets, then we write $Y^X$ for the set $\set(X,Y)$ of functions from $X$ to $Y$.

\begin{definition}
\label{def:endoPROP}
A  $\frakC$-colored \textbf{endomorphism PROP} $E_X$ consists of a $\frakC$-graded set $X = \{X_c\}_{c\in \frakC}$.  Given $m, n \geq 1$ and colors $c_1, \ldots , c_n$, $d_1, \ldots , d_m$, it has the component
\[
E_X\binom{\underline{d}}{\underline{c}}
= \left(X_{d_1} \times \cdots \times X_{d_m}\right)^{\left(X_{c_1} \times \cdots \times X_{c_n}\right)} = X_{\underline{d}}^{X_{\underline{c}}}.
\]
The $\Sigma_m$-$\Sigma_n$ action is the obvious one, with $\Sigma_m$ permuting the $m$ factors $X_{\underline{d}} = X_{d_1} \times \cdots \times X_{d_m}$ and $\Sigma_n$ permuting the $n$ factors in the exponent.  The horizontal composition in $E_X$ is given by Cartesian products of functions.  The vertical composition in $E_X$ is composition of functions with matching colors.
\end{definition}

Note that the endomorphism PROP $E_X$ is a \emph{unital} $\frakC$-colored PROP.
Indeed, for a color $c \in \frakC$, the $c$-colored unit
\[
\mathtt{1}_c \in E_X\binom{c}{c} = X_c^{X_c}
\]
is the identity map of $X_c$.

\begin{definition}
\label{def:coloredPROPalgebra}
For a unital (resp. non-unital) $\frakC$-colored PROP $\sfP$, a \textbf{$\sfP$-algebra} is a morphism
\[
\lambda \colon \sfP \to E_X
\]
of unital (resp. non-unital) $\frakC$-colored PROPs, where $E_X$ is the $\frakC$-colored endomorphism PROP of a $\frakC$-graded set $X$.  We say that $X$ is a $\sfP$-algebra with \textbf{structure map} $\lambda$.  Morphisms of $\sfP$-algebras are defined below.  The category of $\sfP$-algebras is denoted by $\alg(\sfP)$.
\end{definition}

Suppose that $\sfP$ is a unital $\frakC$-colored PROP.  If we want to emphasize that we are considering $\sfP$-algebras with $\sfP$ being unital, we will call them \textbf{unital $\sfP$-algebras}.

As usual one can unpack Definition ~\ref{def:coloredPROPalgebra} and, using adjunction, express the structure map as a collection of maps
\begin{equation}
\label{propstructuremap}
\lambda \colon \sfP\binom{\underline{d}}{\underline{c}} \times X_{\underline{c}} \to X_{\underline{d}},
\end{equation}
one for each pair $(\underline{d}; \underline{c})$ of $\frakC$-profiles.  These maps are associative (with respect to both the horizontal and the vertical compositions) and bi-equivariant.  They also respect the $c$-colored units in the unital case.

A \textbf{morphism} $f \colon X \to Y$ of $\sfP$-algebras is a collection of maps
\[
f = \{f_c \colon X_c \to Y_c\}_{c\in \frakC}
\]
such that the diagram
\begin{equation}
\label{algebramap}
\SelectTips{cm}{10}
\xymatrix{
\sfP\dbinom{\underline{d}}{\underline{c}} \times X_{\underline{c}} \ar[r]^-{\lambda_X} \ar[d]_{Id \times f_{\underline{c}}} & X_{\underline{d}} \ar[d]^{f_{\underline{d}}} \\
\sfP\dbinom{\underline{d}}{\underline{c}} \times Y_{\underline{c}} \ar[r]^-{\lambda_Y} & Y_{\underline{d}}
}
\end{equation}
commutes for all $m, n \geq 1$ and colors $c_1, \ldots , c_n$ and $d_1, \ldots , d_m$.  Here we used the shorthand
\[
f_{\underline{c}} = f_{c_1} \times \cdots \times f_{c_n},
\]
and similarly for $f_{\ud}$.

\subsection{Colored operads and colored PROPs}
\label{subsec:coloredoperad}

A \textbf{$\frakC$-colored operad} $\sfO$ over $\set$ consists of sets
\[
\sfO\binom{d}{\uc} = \sfO\binom{d}{c_1, \ldots , c_n}
\]
for any colors $d, c_1, \ldots , c_n \in \frakC$.  There is a structure map
\[
\rho \colon \sfP\dbinom{d}{\uc} \times \sfP\dbinom{c_1}{\ub^1} \times \cdots \times \sfP\dbinom{c_n}{\ub^n} \to \sfP\dbinom{d}{\ub^1, \ldots , \ub^n}
\]
that is associative, right equivariant, and unital (in the unital case) in a suitable sense.  We refer the reader to \cite{may97} or \cite{markl06} for the precise formulations of these well-known axioms for (unital) operads.  The definitions in the colored case can be found in, e.g., \cite[Section 2]{markl04} or \cite[Section 2]{bd1}.  The category of $\frakC$-colored operads over $\set$ is denoted by $\operadc$.

The \textbf{$\frakC$-colored endomorphism operad} $End_X$ of a $\frakC$-graded set $X = \{X_c\}$ has components
\[
End_X\binom{d}{c_1, \ldots , c_n} = X_d^{X_{c_1} \times \cdots \times X_{c_n}}.
\]
The structure map $\rho$ is given by Cartesian product of functions and composition.  The right equivariance comes from permutations of the factors in $X_{c_1} \times \cdots \times X_{c_n}$.  For a $\frakC$-colored operad $\sfO$, an \textbf{$\sfO$-algebra} is a map
\[
\lambda \colon \sfO \to End_X
\]
of $\frakC$-colored operads.

Comparing the definitions of the $\frakC$-colored endomorphism PROP $E_X$ and endomorphism operad $End_X$, one can see that a colored operad is ``small" than a colored PROP.  In fact, it is straightforward to see that the colored endomorphism operad $End_X$ is obtained from the colored endomorphism PROP $E_X$ by forgetting structures.  This is, of course, not an accident.  In fact, there is a free-forgetful adjoint pair between colored operads and colored PROPs.

To construct the free colored PROP of a colored operad, we need a functor
\[
\boxdot \colon \set^{\Sigma_{\ud;\uc}} \times \set^{\Sigma_{\ub;\ua}} \to \set^{\Sigma_{(\ud,\ub);(\uc,\ua)}}.
\]
The functor $\boxdot$ is constructed as an inclusion functor followed by a left Kan extension.  Indeed, there is a functor
\[
\iota \colon \set^{\Sigma_{\ud;\uc}} \times \set^{\Sigma_{\ub;\ua}} \to \set^{\Sigma_{\ud} \times \Sigma_{\ub} \times \Sigma_{\uc}^{op} \times \Sigma_{\ua}^{op}}, \quad (X,Y) \mapsto X \times Y
\]
that sends $(X,Y) \in \set^{\Sigma_{\ud;\uc}} \times \set^{\Sigma_{\ub;\ua}}$ to the diagram $X \times Y$ with
\begin{equation}
\label{XtensorY}
(X \times Y)\left(\sigma\ud; \mu\ub; \uc\tau^{-1}; \ua\nu^{-1}\right) = X\binom{\sigma\ud}{\uc\tau^{-1}} \times Y\binom{\mu\ub}{\ua\nu^{-1}},
\end{equation}
and similarly for maps in $\Sigma_{\ud} \times \Sigma_{\ub} \times \Sigma_{\uc}^{op} \times \Sigma_{\ua}^{op}$.  On the other hand, the subcategory inclusion
\[
\left(\Sigma_{\ud} \times \Sigma_{\ub}\right) \times \left(\Sigma_{\uc}^{op} \times \Sigma_{\ua}^{op}\right) \xrightarrow{i} \Sigma_{(\ud,\ub);(\uc,\ua)} = \Sigma_{(\ud;\ub)} \times \Sigma_{(\uc,\ua)}^{op}
\]
induces a functor on the diagram categories
\begin{equation}
\label{ei}
\set^{i} \colon \set^{\Sigma_{(\ud,\ub);(\uc,\ua)}} \to \set^{\Sigma_{\ud} \times \Sigma_{\ub} \times \Sigma_{\uc}^{op} \times \Sigma_{\ua}^{op}}.
\end{equation}
This last functor has a left adjoint
\begin{equation}
\label{leftKanext}
K \colon \set^{\Sigma_{\ud} \times \Sigma_{\ub} \times \Sigma_{\uc}^{op} \times \Sigma_{\ua}^{op}} \to \set^{\Sigma_{(\ud,\ub);(\uc,\ua)}},
\end{equation}
which sends a functor $Z \in \set^{\Sigma_{\ud} \times \Sigma_{\ub} \times \Sigma_{\uc}^{op} \times \Sigma_{\ua}^{op}}$ to the left Kan extension of $Z$ 
along $i$ \cite[pp.236-240]{maclane2}.  This left Kan extension exists because $\Sigma_{\ud} \times \Sigma_{\ub} \times \Sigma_{\uc}^{op} \times \Sigma_{\ua}^{op}$ is a small category, and $\set$ has small colimits.

\begin{lemma}
\label{boxdotassociative}
The functor
\[
\boxdot = K\iota \colon \set^{\Sigma_{\ud;\uc}} \times \set^{\Sigma_{\ub;\ua}} \to \set^{\Sigma_{(\ud,\ub);(\uc,\ua)}}
\]
is associative in the obvious sense.
\end{lemma}

\begin{proof}
The associativity of $\boxdot$ is a consequence of the associativity of $\times$ in $\set$ \eqref{XtensorY} and the universal properties of left Kan extensions.
\end{proof}

\begin{theorem}
\label{operadpropadj}
There is a pair of adjoint functors
\begin{equation}
\label{operadprop}
(-)_{prop} \colon \operadc \rightleftarrows \propc \colon U
\end{equation}
between the categories of non-unital $\frakC$-colored PROPs and non-unital $\frakC$-colored operads, with $U$ being the right adjoint.  Moreover, these functors restrict to the subcategories of unital $\frakC$-colored operads and unital $\frakC$-colored PROPs.
\end{theorem}


\begin{proof}
First we construct the forgetful functor $U$.  Suppose that $d, c_i, b^i_j \in \frakC$ are colors, where $1 \leq i \leq n$ and, for each $i$, $1 \leq j \leq k_i$.  Write
\[
\begin{split}
\uc &= (c_1, \ldots , c_n),\\
\ub^i &= (b^i_1, \ldots , b^i_{k_i}),\\
\ub &= (\ub^1, \ldots , \ub^n).
\end{split}
\]
If $\sfP$ is a $\frakC$-colored PROP, then the components in the $\frakC$-colored operad $U\sfP$ are
\begin{equation}
\label{Uprop}
(U\sfP)\binom{d}{c_1, \ldots , c_n} = \sfP\binom{d}{c_1, \ldots , c_n} = \sfP\binom{d}{\uc}.
\end{equation}
The structure map $\rho$ of the $\frakC$-colored operad $U\sfP$ is the composition
\begin{equation}
\label{Upropstructuremap}
\SelectTips{cm}{10}
\xymatrix{
\sfP\dbinom{d}{\uc} \times \sfP\dbinom{c_1}{\ub^1} \times \cdots \times \sfP\dbinom{c_n}{\ub^n} \ar[dr]^-{\rho} \ar[d]_-{Id \times \text{(horizontal)}} & \\
\sfP\dbinom{d}{\uc} \times \sfP\dbinom{\uc}{\ub} \ar[r]^-{\circ} & \sfP\dbinom{d}{\ub}.
}
\end{equation}
The associativity of the horizontal and the vertical compositions in $\sfP$ together with the interchange rule \eqref{interchangerule} imply that $\rho$ is associative.  The right equivariance of $\rho$ follows from those of $\otimes$ and $\circ$.  If $\sfP$ is unital, it is easy to see that $U\sfP$ is unital as well.

Now we construct the unique colored PROP $\sfO_{prop}$ generated by a colored operad $\sfO$.  Let $\sfO$ be a $\frakC$-colored operad with components
\[
\sfO\binom{d}{c_1, \ldots , c_n} = \sfO\binom{d}{\uc}
\]
for $d, c_i \in \frakC$.   First we define the underlying $\frakC$-colored $\Sigma$-bimodule of $\sfO_{prop}$.  Using the decomposition of $\sigmac$ (Proposition ~\ref{decomposition}), we have to specify the diagrams
\[
\sfO_{prop}\binom{[\ud]}{[\uc]} \in \set^{\Sigma_{\ud;\uc}} = \set^{\Sigma_{\ud} \times \Sigma_{\uc}^{op}},
\]
where $\ud = (d_1, \ldots , d_m)$ and $\uc = (c_1, \ldots , c_n)$ are $\frakC$-profiles.  To each partition
\[
r_1 + \cdots + r_m = n
\]
of $n$ with each $r_i \geq 1$, we can associate to the $\frakC$-colored operad $\sfO$ the diagrams
\[
\sfO\binom{[d_i]}{[\uc_i]} \in \set^{\Sigma_{d_i; \uc_i}} = \set^{\Sigma_{d_i} \times \Sigma_{\uc_i}^{op}} = \set^{\{*\} \times \Sigma_{\uc_i}^{op}}
\]
for $1 \leq i \leq m$, where
\[
\uc_i = \left(c_{r_1 + \cdots + r_{i-1} + 1}, \ldots , c_{r_1 + \cdots + r_i}\right).
\]
Using the associativity of $\boxdot$ (Lemma ~\ref{boxdotassociative}), we define the object
\begin{equation}
\label{Oprop}
\sfO_{prop}\binom{[\ud]}{[\uc]} = \coprod_{r_1 + \cdots + r_m = n} \sfO\binom{[d_1]}{[\uc_1]} \boxdot \cdots \boxdot \sfO\binom{[d_m]}{[\uc_m]} \in \set^{\Sigma_{\ud;\uc}},
\end{equation}
where the coproduct is taken over all the partitions $r_1 + \cdots + r_m = n$ with each $r_i \geq 1$.  (Of course, if $m > n$, then no such partition exists, in which case $\sfO_{prop}\binom{[\ud]}{[\uc]}$ is the empty diagram.)  By Proposition ~\ref{decomposition}, this defines $\sfO_{prop}$ as an object in $\sigmac$.

The horizontal composition in $\sfO_{prop}$ is given by concatenation of $\boxdot$ products and inclusion of summands.  Using the universal properties of left Kan extensions, the vertical composition in $\sfO_{prop}$ is uniquely determined by the operad composition in $\sfO$.  One can check that $(-)_{prop}$ is left adjoint to the forgetful functor $U$.
\end{proof}

Note that the left adjoint $(-)_{prop}$ is an embedding.  In fact, for a $\frakC$-color operad $\sfO$, it follows from the definitions of $(-)_{prop}$ and $U$ that
\[
\sfO = U (\sfO_{prop}).
\]
Using the above adjunction, we now observe that passing from a colored operad $\sfO$ to the colored PROP $\sfO_{prop}$ does not alter the category of algebras.

\begin{corollary}
\label{Opropalgebra}
Let $\sfO$ be a $\frakC$-colored operad.  Then there are functors
\[
\Phi \colon \alg(\sfO) \rightleftarrows \alg(\sfO_{prop}) \colon \Psi
\]
that give an equivalence between the categories $\alg(\sfO)$ of $\sfO$-algebras and $\alg(\sfO_{prop})$ of $\sfO_{prop}$-algebras.  Moreover, if $\sfO$ is unital, then these functors give an equivalence between the categories of unital $\sfO$-algebras and unital $\sfO_{prop}$-algebras.
\end{corollary}

\begin{proof}
First observe that in each of the two categories, an algebra has an underlying $\frakC$-graded set $\{A_c\}$.  Given an $\sfO$-algebra $A$, the formula \eqref{Oprop} for $\sfO_{prop}$ together with the universal properties of left Kan extensions extend $A$ to an $\sfO_{prop}$-algebra.  This is the functor $\Phi$.

Conversely, an $\sfO_{prop}$-algebra is a map
\[
\lambda \colon \sfO_{prop} \to E_X
\]
of $\frakC$-colored PROPs, where $E_X$ is the $\frakC$-colored endomorphism PROP of a $\frakC$-graded set $X = \{X_c\}$.  Using the free-forgetful adjunction from Theorem  ~\ref{operadpropadj}, this $\sfO_{prop}$-algebra is equivalent to a map
\[
\lambda' \colon \sfO \to U(E_X)
\]
of $\frakC$-colored operads.  From the definition (\eqref{Uprop} and \eqref{Upropstructuremap}) of the forgetful functor $U$, one observes that $U(E_X)$ is the $\frakC$-colored endomorphism operad of $X$.  Therefore, the map $\lambda'$ is actually giving an $\sfO$-algebra structure on $X$.  This is the functor $\Psi$.  One can check that the functor $\Phi$ and $\Psi$ give an equivalence of categories.  The unital assertion is immediate from the definitions of $\Phi$ and $\Psi$
\end{proof}

\subsection{Examples}
\label{subsec:PROPex}

\begin{example}[\textbf{Polycategories as colored PROPs}]
\label{ex:polycategory}
Lambek's multicategory ($=$ colored operad) \cite{lambek1} generalizes a small category by allowing the source of a morphism to be a finite sequence of objects.  So a morphism in a multicategory takes the form
\[
f \colon (x_1, \ldots , x_n) \to y,
\]
where the $x_i$ and $y$ are objects in the multicategory.  A \textbf{polycategory} \cite{koslowski,szabo} generalizes a multicategory by allowing both the source and the target of a morphism to be finite sequences of objects.  So a morphism in a polycategory, called \textbf{polymorphism}, takes the form
\[
f \colon (x_1, \ldots , x_n) \to (y_1, \ldots , y_m).
\]
Polycategories and their close variants are important tools in proof theory and theoretical computer science \cite{bhru,hyland,hs}.  The point is that compositions of polymorphisms allow one to perform \emph{cuts} to \emph{sequents}; see, e.g., \cite{bs,glt} for the definitions of these terms from proof theory.  Polycategories are even used in linguistics \cite{lambek2}.

As pointed out in \cite{markl06}, a polycategory is essentially a colored dioperad \cite{gan}.  The set of colors is the set of objects in the polycategory.  Just as a colored operad generates a unique colored PROP (Theorem ~\ref{operadpropadj}), so does a colored dioperad.  In fact, the pasting scheme that defines dioperads are the connected simply-connected graphs, which form a subset of the graphs constituting the pasting scheme of PROPs in general.  In particular, for a polycategory $C$, one can associate to it an $Ob(C)$-colored PROP, which uniquely determines the polycategory $C$.  Weak $n$ versions of polycategories will be discussed in \S\ref{subsec:hdcat}.
\qed
\end{example}

\begin{example}[\textbf{Sets as algebras over the initial PROP}]
\label{initialprop}
Let $\sfI$ be the initial $1$-colored unital PROP in $\set$, i.e, the initial object among the $1$-colored unital PROPs.  The components of $\sfI$ are
\[
\sfI(m,n) = \begin{cases} \varnothing & \text{if } m \not= n,\\
\{*\} & \text{if } m = n.\end{cases}
\]
The PROP structure on $\sfI$ is the obvious ones.  A unital $\sfI$-algebra consists of a set $A$ together with maps
\[
A^{\times n} \cong \sfI(n,n) \times A^{\times n} \to A^{\times n},
\]
which all act as the identity maps.  Thus, the categories of unital $\sfI$-algebras and $\set$ are isomorphic.  We say that $\sfI$ is the PROP for sets.

Note that $\sfI$ is the unique $1$-colored PROP generated by the $1$-colored unital operad $I$, whose only non-empty component is $I(1) = \{*\}$.  In particular, the categories of $\sfI$-algebras and $I$-algebras are isomorphic by Corollary ~\ref{Opropalgebra}.  Since it is well-known that $I$-algebras are sets (see, e.g., \cite[Example 16]{bd1}), this also confirms that $\sfI$ is the PROP for sets.  We will use $\sfI$ to define weak $n$-categories in Definition ~\ref{def:weakncat}.
\qed
\end{example}

\begin{example}[\textbf{Bicommutative bimonoids as algebras over the terminal PROP}]
\label{terminalprop}
Let $\sfT$ be the $1$-colored unital PROP given by
\[
\sfT(m,n) = \{*\}
\]
for all $m,n \geq 1$.  The PROP structure on $\sfT$ is the obvious ones.  Then $\sfT$ is the terminal object among all the $1$-colored PROPs in $\set$ (not just the unital PROPs).  A unital $\sfT$-algebra consists of a set $B$ together with bi-equivariant maps
\[
B^{\times n} \cong \sfT(m,n) \times B^{\times n} \xrightarrow{\mu(m,n)} B^{\times m}
\]
for $m, n \geq 1$ such that
\begin{equation}
\label{mumn}
\mu(m,n) = \mu(m,k) \circ \mu(k,n)
\end{equation}
for all $n, k, m \geq 1$,
\begin{equation}
\label{musum}
\mu(m_1 + m_2, n_1 + n_2) = \mu(m_1, n_1) \times \mu(m_2, n_2)
\end{equation}
for all $m_i, n_i \geq 1$, and $\mu(n,n)$ is the identity map for each $n$.  In particular, the maps
\[
B^{\times 2} \xrightarrow{\mu = \mu(1,2)} B \quad \text{and} \quad
B \xrightarrow{\Delta = \mu(2,1)} B^{\times 2}
\]
give $B$ the structures of an associative commutative monoid and of a coassociative cocommutative comonoid, respectively.  Moreover, we have
\[
Id = \mu(1,1) = \mu(1,2) \circ \mu(2,1) = \mu \circ \Delta \colon B \to B,
\]
\[
\mu^{n-1} = \mu(1,n) \colon B^{\times n} \to B, \quad \text{and} \quad
\Delta^{m-1} = \mu(m,1) \colon B \to B^{\times m}.
\]
We call $(B, \mu, \Delta)$ (the object $B$ with the two (co)associative (co)commutative operations $\mu$ and $\Delta$ such that $Id = \mu \circ \Delta$) a \textbf{bicommutative bimonoid}.

We claim that the bicommutative bimonoid structure of $(B, \mu, \Delta)$ uniquely determines $B$ as a unital $\sfT$-algebra.  In fact, we have
\[
\begin{split}
\mu(m,n) &= \mu(m,k) \circ \mu(k,n)\\
&= \mu(m,1) \circ \mu(1,k) \circ \mu(k,1) \circ \mu(1,n)\\
&= \mu(m,1) \circ Id \circ \mu(1,n) \\
&= \Delta^{m-1} \circ \mu^{n-1}.
\end{split}
\]
This shows that every $\mu(m,n)$ is uniquely determined by $\mu$ and $\Delta$.  It follows that the categories of unital $\sfT$-algebras and of bicommutative bimonoids are canonically isomorphic.  Following the tradition of homotopy theory, we might also call $\sfT$ an \textbf{$E_\infty$-PROP}.

Note that $\sfT$ is \emph{not} the unique PROP generated by the terminal $1$-colored operad $T$, which has $T(n) = \{*\}$ for each $n \geq 1$.  One can see this by considering the $(m,n)$ component (with $m > n$) of the PROP $T_{prop}$ generated by $T$.  For example, one can check that
\[
T_{prop}(m,n) = \varnothing \quad \text{when} \quad m > n.
\]
So clearly
\[
\sfT \not= T_{prop}.
\]
However, since $\sfT$ is the terminal $1$-colored PROP, there is a unique map
\[
i \colon T_{prop} \to \sfT.
\]
Thus, each unital $\sfT$-algebra ($=$ bicommutative bimonoid) $B = (B, \mu, \Delta)$ also has a unital $T_{prop}$-algebra ($=$ $T$-algebra) structure $i^*B$. It is known that unital $T$-algebras are commutative monoids (see, e.g, \cite{bd1}).  It is not hard to check that, in fact, $i^*B$ is the commutative monoid $(B, \mu)$ obtained from the bicommutative bimonoid $(B, \mu, \Delta)$ by forgetting about the comultiplication $\Delta$.

Weak $n$ versions of bicommutative bimonoids will be defined in Definition ~\ref{def:bicomweakncat}.
\qed
\end{example}

\begin{example}[\textbf{Topological Field Theories and the Segal PROP}]
\label{ex:CFT}
Our discussion of the Segal $\Se$ PROP follows \cite{cv,segalA,segalB,segalC}.  The $1$-colored Segal PROP $\Se$ comes from moduli spaces of Riemann surfaces with boundary holes.  It is of great importance in mathematical physics because several topological field theories are algebras over various versions of the Segal PROP.  Among those topological field theories are:
\begin{enumerate}
\item
Conformal Field Theory (CFT);
\item
Topological Conformal Field Theory (TCFT), also known as a string background;
\item
Cohomological Field Theory-$I$ (CohFT-$I$);
\item
Topological Quantum Field Theory (TQFT).
\end{enumerate}

Although the discussion below focuses on the $1$-colored version of $\Se$, we should point out that it is easy to generalize the Segal PROP $\Se$ to allow boundary holes with different sizes, in which case $\Se$ is a colored PROP.  In fact, in the Riemann surfaces under consideration, we can allow the boundary holes to have different circumferences.  In other words, we can allow not only the unit disk but all disks with, say, non-zero circumferences.  In this case, the generalized Segal PROP is a colored PROP, where the set of colors is the set of allowable circumferences of the boundary holes.  The vertical composition in this generalized, colored Segal PROP is then performed only to the Riemann surfaces whose boundary holes have matching circumferences.  With this in mind, our discussion of the various topological field theories can be easily extended to this colored setting as well.

Considering varying circumferences in the boundary holes is not unprecedented.  For example, in the setting of string topology, there is a combinatorially defined colored PROP $\catrcf(g)$ \cite{chataur,cg} that is built from spaces of reduced metric Sullivan chord diagrams with genus $g$.  Such a Sullivan chord diagram is a marked \emph{fat graph} (also known as ribbon graph) that represents a surface with genus $g$ that has a certain number of input and output circles in its boundary.  These boundary circles are allowed to have different circumferences.  The set of such circumferences is the set of colors for the colored PROP $\catrcf(g)$.

For integers $m, n \geq 1$, let $\Se(m,n)$ be the moduli space of (isomorphism classes of) complex Riemann surfaces whose boundaries consist of $m + n$ labeled holomorphic holes that are mutually non-overlapping.  In the literature, $\Se(m,n)$ is sometimes denoted by $\widehat{\catm}(m,n)$.  The holomorphic holes are actually bi-holomorphic maps from $m + n$ copies of the closed unit disk to the Riemann surface.  The first $m$ labeled holomorphic holes are called the \emph{outputs} and the last $n$ are called the \emph{inputs}.  Note that these Riemann surfaces $M$ can have arbitrary genera and are \emph{not} required to be connected.

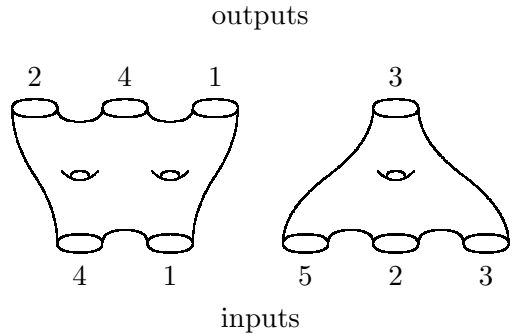
\[
\setlength{\unitlength}{.6mm}
\begin{picture}(120,80)(10,-10)
\qbezier(20,15)(20,17)(25,17)
\qbezier(25,17)(30,17)(30,15)
\qbezier(20,15)(20,13)(25,13)
\qbezier(25,13)(30,13)(30,15)
\qbezier(30,15)(30,18)(35,18) 
\qbezier(35,18)(40,18)(40,15) 
\qbezier(40,15)(40,17)(45,17)
\qbezier(45,17)(50,17)(50,15)
\qbezier(40,15)(40,13)(45,13)
\qbezier(45,13)(50,13)(50,15)
\qbezier(20,15)(20,22.5)(15,30)
\qbezier(15,30)(10,37.5)(10,45)
\qbezier(50,15)(50,22.5)(55,30)
\qbezier(55,30)(60,37.5)(60,45)
\qbezier(10,45)(10,43)(15,43)
\qbezier(15,43)(20,43)(20,45)
\qbezier(10,45)(10,47)(15,47)
\qbezier(15,47)(20,47)(20,45)
\qbezier(20,45)(20,42)(25,42) 
\qbezier(25,42)(30,42)(30,45) 
\qbezier(30,45)(30,43)(35,43)
\qbezier(35,43)(40,43)(40,45)
\qbezier(30,45)(30,47)(35,47)
\qbezier(35,47)(40,47)(40,45)
\qbezier(40,45)(40,42)(45,42) 
\qbezier(45,42)(50,42)(50,45) 
\qbezier(50,45)(50,43)(55,43)
\qbezier(55,43)(60,43)(60,45)
\qbezier(50,45)(50,47)(55,47)
\qbezier(55,47)(60,47)(60,45)
\qbezier(21,31)(23,29)(25,29) 
\qbezier(25,29)(27,29)(29,31)
\qbezier(23,30)(23,31)(25,31)
\qbezier(25,31)(27,31)(27,30)
\qbezier(41,31)(43,29)(45,29) 
\qbezier(45,29)(47,29)(49,31)
\qbezier(43,30)(43,31)(45,31)
\qbezier(45,31)(47,31)(47,30)
\qbezier(70,15)(70,17)(75,17)
\qbezier(75,17)(80,17)(80,15)
\qbezier(70,15)(70,13)(75,13)
\qbezier(75,13)(80,13)(80,15)
\qbezier(80,15)(80,18)(85,18) 
\qbezier(85,18)(90,18)(90,15) 
\qbezier(90,15)(90,17)(95,17)
\qbezier(95,17)(100,17)(100,15)
\qbezier(90,15)(90,13)(95,13)
\qbezier(95,13)(100,13)(100,15)
\qbezier(100,15)(100,18)(105,18) 
\qbezier(105,18)(110,18)(110,15) 
\qbezier(110,15)(110,17)(115,17)
\qbezier(115,17)(120,17)(120,15)
\qbezier(110,15)(110,13)(115,13)
\qbezier(115,13)(120,13)(120,15)
\qbezier(90,45)(90,47)(95,47)
\qbezier(95,47)(100,47)(100,45)
\qbezier(90,45)(90,43)(95,43)
\qbezier(95,43)(100,43)(100,45)
\qbezier(70,15)(70,22.5)(80,30)
\qbezier(80,30)(90,37.5)(90,45)
\qbezier(120,15)(120,22.5)(110,30)
\qbezier(110,30)(100,37.5)(100,45)
\qbezier(91,31)(93,29)(95,29)
\qbezier(95,29)(97,29)(99,31)
\qbezier(93,30)(93,31)(95,31)
\qbezier(95,31)(97,31)(97,30)
\put(65,65){\makebox(0,0){outputs}}
\put(65,-2){\makebox(0,0){inputs}}
\put(25,8){\makebox(0,0){$4$}}
\put(45,8){\makebox(0,0){$1$}}
\put(75,8){\makebox(0,0){$5$}}
\put(95,8){\makebox(0,0){$2$}}
\put(115,8){\makebox(0,0){$3$}}
\put(15,52){\makebox(0,0){$2$}}
\put(35,52){\makebox(0,0){$4$}}
\put(55,52){\makebox(0,0){$1$}}
\put(95,52){\makebox(0,0){$3$}}
\end{picture}
\]

One can visualize a Riemann surface $M \in \Se(m,n)$ as a pair of \emph{alien pants} in which there are $n$ legs (the inputs) and $m$ waists (the outputs).  See the picture above for an element of $\Se(4,5)$ with two connected components.  With this picture in mind, such a Riemann surface is also known as a \emph{worldsheet} in the physics literature.  In this interpretation, a worldsheet is an embedding of closed strings in space-time.  We think of such a Riemann surface $M$ as a machine that provides an operation with $n$ inputs and $m$ outputs.

The collection of moduli spaces
\[
\Se = \{\Se(m,n) \colon m, n \geq 1\}
\]
forms a $1$-colored topological PROP, called the \textbf{Segal PROP}, also known as the \textbf{Segal category}.  Its horizontal composition
\[
\Se(m_1,n_1) \times \Se(m_2,n_2) \xrightarrow{\otimes \,=\, \sqcup} \Se(m_1 + m_2,n_1 + n_2)
\]
is given by disjoint union $M_1 \sqcup M_2$.  In other words, put two pairs of alien pants side-by-side.  Its vertical composition
\[
\Se(m,n) \times \Se(n,k) \xrightarrow{\circ} \Se(m,k), \quad (M,N) \mapsto M \circ N
\]
is given by holomorphically sewing the $n$ output holes (the waists) of $N$ with the $n$ input holes (the legs) of $M$.  The $\Sigma_m$-$\Sigma_n$ action on $\Se(m,n)$ is given by permuting the labels of the $m$ output and the $n$ input holomorphic holes.

Let $\bk$ be a field of characteristic $0$, and let $C_*$ denote the singular chain functor with coefficients in $\bk$.  Applying this singular chain functor to the Segal PROP $\Se$, we obtain
\[
\bSe = C_*(\Se),
\]
which is a $1$-colored PROP over chain complexes of $\bk$-modules.  An algebra over the $\bk$-linear chain PROP $\bSe$ is by definition a \textbf{Topological Conformal Field Theory}.

Passing to homology first, we obtain
\begin{equation}
\label{H*Se}
H_*(\Se),
\end{equation}
which is a $1$-colored PROP over graded $\bk$-modules.  An algebra over the graded $\bk$-linear PROP $H_*(\Se)$ is by definition a \textbf{Cohomological Field Theory-$I$}.

If we take only the $0$th homology module, then we obtain
\begin{equation}
\label{H0Se}
H_0(\Se),
\end{equation}
which is a $1$-colored PROP over $\bk$-modules.  An algebra over the $\bk$-linear PROP $H_0(\Se)$ is by definition a \textbf{Topological Quantum Field Theory}.  Weak $n$ versions of cohomological field theory-$I$ and topological quantum field theory will be defined in \S\ref{subsec:hdtft}.
\qed
\end{example}

\section{Higher PROPs}
\label{sec:higherPROP}

Throughout this section, our underlying category is $\set$.  This assumption can be relaxed a little bit.  What we actually need is that the underlying category $\cate$ be set-based, i.e., there is a suitable forgetful functor $\cate \to \set$.  In this setting, it makes sense to talk about the \emph{underlying set of elements} of an object and the \emph{underlying function} of a morphism in $\cate$.  For example, one can easily adapt the discussion in this section to the case $\cate =$ the category of $\bk$-modules, where $\bk$ is a field of characteristic $0$.

It has long been known that there is a colored operad whose algebras are operads, i.e., the operad for operads.  In fact, given any unital colored operad $\sfO$, it is shown in \cite{bd1} that there exists a unital $\elt(\sfO)$-colored operad $\sfO^+$ whose algebras are exactly the colored operads over $\sfO$.  Here $\elt(\sfO)$ is the set of elements, also called \emph{operations}, in $\sfO$.  For example, starting with the terminal $1$-colored operad $T$, one obtains $T^+$, which is the colored operad for $1$-colored operads.

The so-called \textbf{slice construction} $\sfO^+$ lies at the very heart of the higher category theory of Baez and Dolan \cite{bd1}.  One considers $\sfO^+$ as a \emph{higher operad}, in the sense that the operations in $\sfO$ are now the colors in $\sfO^+$.  From its construction, the operations in $\sfO^+$ are the \emph{reduction laws} in $\sfO$, which are equations stating that the composite of certain operations is equal to some operation.  Moreover, the reduction laws in $\sfO^+$ are the ways of combining reduction laws in $\sfO$ to obtain other reduction laws.  The upgrading process described in the last two sentences, repeated multiple (or infinitely many) times, is essentially how categorification is achieved in the Baez-Dolan setting \cite{bd1}.

The main purpose of this section is to show that there is an analogous slice construction for colored PROPs, giving rise to \emph{higher PROPs}.  Its purpose is the same as in the operad case.  In other words, given a $\frakC$-colored PROP $\sfP$, we will construct a unital $\elt(\sfP)$-colored PROP $\sfP^+$ whose algebras are exactly the $\frakC$-colored PROPs over $\sfP$.  Restricting to the terminal unital PROP $\sfT$ (Example ~\ref{terminalprop}), it follows that $\sfT^+$ is the colored PROP whose algebras are PROPs, i.e., $\sfT^+$ is the colored PROP for $1$-colored PROPs.  Starting with a $\frakC$-colored version $\sfT_\frakC$, one obtains the colored PROP $\sfT_\frakC^+$ for $\frakC$-colored PROPs (Example ~\ref{T+}).

There is another interesting example if we start with the initial $1$-colored PROP $\sfI$ (Example ~\ref{initialprop}).  As we will see in Example ~\ref{I+}, unital $\sfI^+$-algebras are \emph{bi-equivariant graded monoidal monoids}.  Disregarding the bi-equivariance and the grading, these monoidal monoids can be regarded as de-categorified versions of the \emph{$2$-fold monoidal categories} of \cite{bfsv}.

Following the Baez-Dolan approach \cite{bd1} and using our higher PROP construction, we will define the category of \textbf{$\sfP$-propertopes} in Section ~\ref{sec:propertopes}.  These propertopes -- as opposed to shapes such as globes, cubes, simplices, or opetopes -- are the shapes of higher cells in our setting.  Higher dimensional $\sfP$-algebras (i.e., $n$-time categorified $\sfP$-algebras for $0 \leq n \leq \infty$) are certain presheaves on the category of $\sfP$-propertopes.
Since our setting is based on colored PROPs, which model algebraic structures with multiple inputs and multiple outputs, the cells in our higher dimensional algebras also have multiple inputs and multiple outputs, as in a polycategory (Example ~\ref{ex:polycategory}).

In \S\ref{subsec:slice} we state the main result regarding the existence of the slice PROP $\sfP^+$ (Theorem ~\ref{thm:P+}) and discuss several examples.  
The rest of this section, \S\ref{subsec:graphs} and \S\ref{subsec:sliceconstruction}, is devoted to proving Theorem ~\ref{thm:P+}.

\subsection{Slice PROPs}
\label{subsec:slice}

To state the main result of this section, we use the following notations.

\begin{definition}
Given any unital $\frakC$-colored PROP $\sfP$, define the set 
\[
\elt(\sfP) = \coprod_{(\ud;\uc)} \sfP\binom{\ud}{\uc},
\]
where the disjoint union is taken over all the pairs $(\ud;\uc)$ of $\frakC$-profiles.  In other words, $\elt(\sfP)$ is the set of elements in $\sfP$.
\end{definition}

For a category $\catc$ and an object $A$ in $\catc$, the \textbf{over category} $\catc/A$ has as objects the morphisms
\[
f \colon B \to A \in \catc.
\]
A morphism in $\catc/A$ is a commutative triangle in $\catc$:
\[
\SelectTips{cm}{10}
\xymatrix{
B \ar[r]^-{f} \ar[d] & A\\
D \ar[ur]_-{g}.
}
\]
Recall that, given a unital $\frakC$-colored PROP $\sfQ$, the category of unital $\sfQ$-algebras is denoted by $\alg(\sfQ)$.  Also, the category of (non-unital) $\frakC$-colored PROPs is denoted by $\propc$.

\begin{theorem}
\label{thm:P+}
Let $\sfP$ be a unital $\frakC$-colored PROP over $\set$.  Then there exist a unital $\elt(\sfP)$-colored PROP $\sfP^+$ and a canonical isomorphism of categories:
\begin{equation}
\label{p+algebras}
\propc/\sfP \cong \alg(\sfP^+).
\end{equation}
\end{theorem}

The proof will be given at the end of this section.  We note that Theorem ~\ref{thm:P+} also holds with $\bk$-modules in place of $\set$.  The minor modifications needed to adapt the constructions and proofs below to $\bk$-modules will be discussed in Remarks ~\ref{remark:P+chain} and ~\ref{remark:P+alg}.  In the $\bk$-linear setting, the isomorphism \eqref{p+algebras} is an isomorphism of categories enriched over $\bk$-modules.

Observe that we now have two ``enlarging" constructions associated to any unital $\frakC$-colored operad $\sfO$:
\begin{enumerate}
\item
$\sfO \mapsto \sfO_{prop}$, the free $\frakC$-colored PROP generated by $\sfO$ (Theorem ~\ref{operadpropadj}).
\item
$\sfO \mapsto \sfO^+$, the Baez-Dolan \cite{bd1} slice operad of $\sfO$.
\end{enumerate}
These two constructions do \emph{not} commute with each other.  In fact, $(\sfO^+)_{prop}$ is the free colored PROP generated by $\sfO^+$, which is $\elt(\sfO)$-colored.  On other other hand, $(\sfO_{prop})^+$ is $\elt(\sfO_{prop})$- colored.  From the construction of $(-)_{prop}$, one can see that there are, in general, more elements in $\sfO_{prop}$ than in $\sfO$ itself.  This suggests that $(\sfO_{prop})^+$ is in some sense bigger than $(\sfO^+)_{prop}$.  The following result, which will not be used in what follows, gives one interpretation of this comparison.

\begin{corollary}
Let $\sfO$ be a unital $\frakC$-colored operad.  Then there is an embedding of categories
\[
\iota \colon \alg\left((\sfO^+)_{prop}\right) \hookrightarrow \alg\left((\sfO_{prop})^+\right)
\]
\end{corollary}

\begin{proof}
The desired embedding is defined as the following composition:
\[
\SelectTips{cm}{10}
\xymatrix{
\alg\left((\sfO^+)_{prop}\right) \ar[rrr]^-{\iota} \ar[d]_-{\Psi} & & & \alg\left((\sfO_{prop})^+\right)\\
\alg(\sfO^+) \ar[r]^-{\cong} & \operadc/\sfO \ar[rr]^-{(-)_{prop}} & & \propc/\sfO_{prop} \ar[u]_-{\cong}.
}
\]
The embedding $\Psi$ is part of the equivalence in Corollary ~\ref{Opropalgebra}, applied to the $\elt(\sfO)$-colored operad $\sfO^+$.  The embedding $(-)_{prop}$ is induced on the over categories by the original free colored PROP functor with the same notation (Theorem ~\ref{operadpropadj}).  The other two functors are isomorphisms.  The right vertical isomorphism is from Theorem ~\ref{thm:P+}, and the other isomorphism is the operad version from Proposition 13 and Theorem 14 in \cite{bd1}.
\end{proof}

What follows are a few examples of $\sfP^+$-algebras for various colored PROPs $\sfP$.

\begin{example}[\textbf{Colored PROPs as $\sfT_\frakC^+$-algebras}]
\label{T+}
In Example ~\ref{terminalprop} we considered the terminal $1$-colored unital PROP $\sfT$.  Here we consider the $\frakC$-colored version $\sfT_\frakC$, which is given by
\[
\sfT_\frakC\binom{\ud}{\uc} = \{*\}
\]
for any $\frakC$-profiles $\ud$ and $\uc$.  The $\frakC$-colored PROP structure on $\sfT_\frakC$ is the obvious one.  Then $\sfT_\frakC$ is a unital $\frakC$-colored PROP that is the terminal object in the category of all $\frakC$-colored PROPs in $\set$.  So $\frakC$-colored PROPs over $\sfT_\frakC$ are just $\frakC$-colored PROPs.  Thus, by Theorem ~\ref{thm:P+} we have a canonical isomorphism
\[
\propc \cong \alg(\sfT_\frakC^+)
\]
of categories.  In other words, $\frakC$-colored PROPs are exactly the $\sfT_\frakC^+$-algebras.\qed
\end{example}

\begin{example}[\textbf{PROPic algebra structures on $A$ as $E_A^+$-algebras}]
Let $A = \{A_c\}$ be a $\frakC$-graded set, and let $\sfE = E_A$ be the $\frakC$-colored endomorphism PROP of $A$.  A map
\[
f \colon \sfP \to \sfE
\]
of $\frakC$-colored PROPs is, by definition, a $\sfP$-algebra structure on $A$.  Thus, by Theorem ~\ref{thm:P+}, the category $\alg(\sfE^+)$ is canonically isomorphic to the category $\propc/\sfE$ of PROPic algebra structures on $A$.  So all the possible PROPic algebra structures on $A$ are, in fact, just algebras over a single colored PROP $\sfE^+$.
\qed
\end{example}

\begin{example}[\textbf{$\sfI^+$-algebras as de-categorified $2$-fold monoidal categories}]
\label{I+}
Let $\sfI$ be the initial $1$-colored unital PROP in $\set$ (Example ~\ref{initialprop}), which is given by
\[
\sfI(m,n) = \begin{cases} \varnothing & \text{if } m \not= n,\\
\{*\} & \text{if } m = n.\end{cases}
\]
By Theorem ~\ref{thm:P+}, $\sfI^+$-algebras are exactly the $1$-colored PROPs over $\sfI$.  Suppose that
\[
f \colon \sfQ \to \sfI
\]
is a PROP over $\sfI$.  Since $\sfI(m,n) = \varnothing$ unless $m = n$, it follows that
\[
\sfQ(m,n) = \varnothing \quad \text{if} \quad m \not= n.
\]
So the only possibly non-empty components in $\sfQ$ are the diagonal components $\sfQ_n := \sfQ(n,n)$ for $n \geq 1$.  The map
\[
f \colon \sfQ_n \to I(n,n) = \{*\}
\]
is the unique map to the one-element set, which gives no information about the set $\sfQ_n$.  Thus, $\sfI^+$-algebras are $1$-colored PROPs whose non-diagonal components are empty.  Of course, given any $1$-colored PROP $\sfP$, we can replace its non-diagonal components with $\varnothing$.  The result is an $\sfI^+$-algebra.  We now provide an intrinsic description of an $\sfI^+$-algebra $\sfQ$.

Each set $\sfQ_n$ has commuting left $\Sigma_n$-action and right $\Sigma_n$-action, i.e., $\sfQ_n$ is $\Sigma_n$-bi-equivariant.  The horizontal composition in $\sfQ$ takes the form
\[
\sfQ_m \times \sfQ_n \xrightarrow{\otimes} \sfQ_{m+n},
\]
which is associative and bi-equivariant.  In other words, $\coprod_{n \geq 1} \sfQ_n$ is a graded bi-equivariant monoid with respect to $\otimes$.  The vertical composition in $\sfQ$ consists of maps
\[
\sfQ_n \times \sfQ_n \xrightarrow{\circ_n} \sfQ_n
\]
that are associative and bi-equivariant.  In other words, each $\sfQ_n$ is a $\Sigma_n$-bi-equivariant monoid with respect to $\circ_n$.  The interchange rule in this case says that
\begin{equation}
\label{interchangexy}
(x_1 \otimes y_1) \circ_{m+n} (x_2 \otimes y_2) = (x_1 \circ_m x_2) \otimes (y_1 \circ_n y_2)
\end{equation}
for $x_1, x_2 \in \sfQ_m$ and $y_1, y_2 \in \sfQ_n$.  In other words, the local monoid structures of the individual $\sfQ_n$ are compatible with the global monoid structure $\otimes$.  We call such an object
\[
\sfQ = \coprod_{n \geq 1} \sfQ_n
\]
with the above bi-equivariant structures and compatible local and global monoid multiplications a \textbf{bi-equivariant graded monoidal monoid}, or simply \textbf{monoidal monoid}.  So $\sfI^+$ is the countably colored PROP for monoidal monoids. The PROP $\sfI^+$ is countably colored because its set of colors is $\elt(\sfI)$, which has one element for each $n \geq 1$.

For example, let $A$ be an associative algebra over a field $\bk$ of characteristic $0$. Then its tensor algebra
\[
TA = \bigoplus_{n \geq 1} A^{\otimes n} = A \oplus A^{\otimes 2} \oplus A^{\otimes 3} \oplus \cdots
\]
gives such a monoidal monoid with $\sfQ_n = A^{\otimes n}$.  Indeed, we can insist that the $\Sigma_n$-bi-equivariant action on $A^{\otimes n}$ be trivial.  Its local monoid structure $\circ_n$ is the induced multiplication structure from $A$.  In other words, we have
\[
(x_1 \otimes \cdots \otimes x_n) \circ_n (y_1 \otimes \cdots \otimes y_n) = x_1y_1 \otimes \cdots \otimes x_ny_n.
\]
The global monoid structure
\[
A^{\otimes m} \times A^{\otimes n} \to A^{\otimes (m+n)}
\]
is concatenation of tensor factors.  The interchange rule \eqref{interchangexy} in this case says that concatenation of tensor factors commutes with the multiplications on the summands $A^{\otimes n}$.

There is a close connection between our monoidal monoids and the $2$-fold monoidal categories of \cite{bfsv}.  Recall from \cite{bfsv} that a \textbf{$2$-fold monoidal category} consists of a category $\catc$, two strictly associative monoidal products
\[
\otimes_i \colon \catc \times \catc \to \catc
\]
for $i = 1$ and $2$, and an \emph{interchange natural transformation}
\[
(A \otimes_2 B) \otimes_1 (C \otimes_2 D) \xrightarrow{\eta} (A \otimes_1 C) \otimes_2 (B \otimes_1 D)
\]
that makes two associativity type squares commute.  (There are also units for $\otimes_1$ and $\otimes_2$ that we have ignored.)  We can thus think of a monoidal monoid $\sfQ$ as a bi-equivariant graded version of a \emph{de-categorified} $2$-fold monoidal category.  The local and global monoid multiplications $\circ$ and $\otimes$ in $\sfQ$ are de-categorifications of the strictly associative monoidal products $\otimes_1$ and $\otimes_2$.  The interchange rule \eqref{interchangexy} is a de-categorification of the interchange natural transformation $\eta$.  In particular, higher dimensional $\sfI^+$-algebras can be thought of as (close cousins of) higher $2$-fold monoidal categories, or $2$-fold monoidal $n$-categories (Definition ~\ref{def:monoidalmonoidal}).
\qed
\end{example}

\subsection{Graphs, decorations, and evaluations}
\label{subsec:graphs}

Before we can prove Theorem ~\ref{thm:P+}, first we define the graphs that serve as the pasting scheme for the slice PROP $\sfP^+$.  Our graphs are slight modifications of those used in \cite{markl06,mv} for the free $1$-colored PROP.

\begin{definition}
\label{def:graph}
For $m, n \geq 1$, an \textbf{$(m,n)$-graph}\label{graph} is a non-empty, not-necessarily connected,  finite directed graph $G$ satisfying the following conditions:
\begin{enumerate}
\item
Each vertex has at least one incoming edge and at least one outgoing edge.
\item
There are no wheels (i.e., directed cycles).
\item
There are exactly $n$ edges, called \textbf{inputs}, that do not have an initial vertex.
\item
There are exactly $m$ edges, called \textbf{outputs}, that do not have a terminal vertex.
\item
The connected components of $G$ are labeled $\{1, 2, \ldots \}$.
\item
Each connected component has at least one vertex (and hence at least one input and one output).
\item
Within each connected component, the sets of vertices, inputs, and outputs are separately labeled $\{1, 2, \ldots \}$.
\end{enumerate}
When $m$ and $n$ are understood from the context, we will simply call an $(m,n)$-graph $G$ a \textbf{graph}.
\end{definition}

If a graph $G$ has $r$ connected components, then we write
\[
G = G^1 \sqcup \cdots \sqcup G^r,
\]
where $G^j$ is the $j$th connected component of $G$.  The $i$th vertex in $G^j$ is denoted by $v_i^j$.  The sets of vertices and edges in a graph $G$ are denoted by $v(G)$ and $e(G)$, respectively.

The $(m,n)$-graphs are the objects of a groupoid.  An \textbf{isomorphism} between two $(m,n)$-graphs consists of a bijection between the sets of vertices and a bijection between the sets of edges preserving all the edge relations.  Moreover, it is required that corresponding connected components, vertices, inputs, and outputs have the same labels.  In what follows, we will identify isomorphic graphs.  We choose, once and for all, one representative from each isomorphism class of $(m,n)$-graphs.

\begin{example}
Here is a graphical representation of a $(5,3)$-graph $G$ with seven vertices and one connected component:
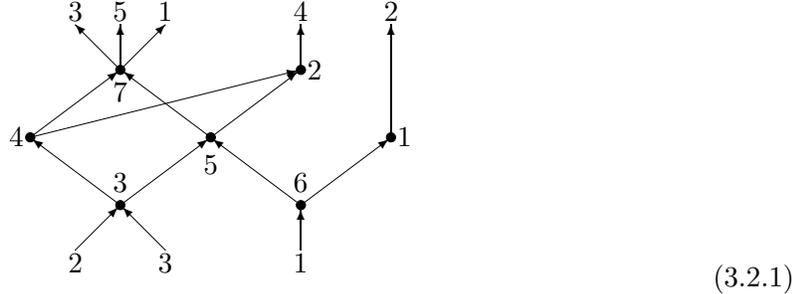
\begin{equation}
\label{35graph}
\setlength{\unitlength}{.6mm}
\begin{picture}(95,70)(5,12)
\put(30,30){\circle*{2}} 
\put(70,30){\circle*{2}}
\put(10,45){\circle*{2}}
\put(50,45){\circle*{2}}
\put(90,45){\circle*{2}}
\put(30,60){\circle*{2}}
\put(70,60){\circle*{2}}
\put(20,20){\vector(1,1){9.5}} 
\put(40,20){\vector(-1,1){9.5}}
\put(70,20){\vector(0,1){9}}
\put(30,30){\vector(-4,3){19.5}}
\put(30,30){\vector(4,3){19.5}}
\put(70,30){\vector(-4,3){19.5}}
\put(70,30){\vector(4,3){19.5}}
\put(10,45){\vector(4,3){19.5}}
\put(10,45){\vector(4,1){59.3}}
\put(50,45){\vector(-4,3){19.3}}
\put(50,45){\vector(4,3){19.3}}
\put(90,45){\vector(0,1){25}}
\put(30,60){\vector(-1,1){10}}
\put(30,60){\vector(0,1){10}}
\put(30,60){\vector(1,1){10}}
\put(70,60){\vector(0,1){10}}
\put(30,35){\makebox(0,0){$3$}} 
\put(70,35){\makebox(0,0){$6$}}
\put(7,45){\makebox(0,0){$4$}}
\put(50,39){\makebox(0,0){$5$}}
\put(93,45){\makebox(0,0){$1$}}
\put(30,55){\makebox(0,0){$7$}}
\put(73,60){\makebox(0,0){$2$}}
\put(20,17){\makebox(0,0){$2$}} 
\put(40,17){\makebox(0,0){$3$}}
\put(70,17){\makebox(0,0){$1$}}
\put(20,73){\makebox(0,0){$3$}}
\put(30,73){\makebox(0,0){$5$}}
\put(40,73){\makebox(0,0){$1$}}
\put(70,73){\makebox(0,0){$4$}}
\put(90,73){\makebox(0,0){$2$}}
\end{picture}
\end{equation}
A vertex is represented by a $\bullet$, and an edge is represented by a directed arrow.  The number closest to a vertex is its label.  The three numbers at the bottom are the labels of the inputs, and the five numbers at the top are the labels of the outputs.  The picture \eqref{35graph} uniquely determines the (isomorphism class of the) graph $G$.  Note that we have arranged the edges so that they all flow from the bottom (the inputs) to the top (the outputs).  We will continue to draw graphs with a bottom-to-top flow for the rest of this paper.\qed
\end{example}

We now decorate graphs with elements and colors from a fixed unital $\frakC$-colored PROP $\sfP$.

\begin{definition}
\label{def:decoratedgraph}
By a \textbf{$\sfP$-decorated $(m,n)$-graph}, or simply a \textbf{$\sfP$-decorated graph}, we mean a pair $(G, \xi)$ consisting of:
\begin{enumerate}
\item
An $(m,n)$-graph $G$ for some $m, n \geq 1$.
\item
A \textbf{decorating function}
\[
\xi \colon v(G) \sqcup e(G) \to \elt(\sfP) \sqcup \frakC
\]
with $\xi(v(G)) \subseteq \elt(\sfP)$ and $\xi(e(G)) \subseteq \frakC$.
\end{enumerate}
The decorating function $\xi$ is required to satisfy the following \textbf{color-matching property}: For a vertex $v \in v(G)$, denote by $in(v)_i$ and $out(v)_j$ the $i$th incoming edge and the $j$th outgoing edge of $v$ (from left to right in its graphical representation).  Then it is required that
\begin{equation}
\label{colormatching}
\xi(v) \in \sfP\binom{\xi(out(v)_1), \ldots , \xi(out(v)_s)}{\xi(in(v)_1), \ldots , \xi(in(v)_r)}
\end{equation}
for every $v \in v(G)$ with $r$ incoming and $s$ outgoing edges.  The image under $\xi$ of a vertex (or an edge) is called its \textbf{decoration}.
\end{definition}

In other words, what the color-matching property \eqref{colormatching} says is this: Let $v$ be an arbitrary vertex of $G$ with, say, $r$ incoming and $s$ outgoing edges.  If these edges connected to $v$ have decorations
\[
\xi(in(v)_i) = c_i \in \frakC \quad \text{and} \quad \xi(out(v)_j) = d_j \in \frakC,
\]
then the decoration $\xi(v) \in \elt(\sfP)$ of $v$ must have $\frakC$-profiles
\[
\binom{\ud}{\uc} = \binom{d_1, \ldots , d_s}{c_1, \ldots , c_r}.
\]

We can draw a $\sfP$-decorated graph $(G,\xi)$ by first drawing the underlying graph $G$.  The decorations of the vertices and the edges are then added to the picture.  The decoration $\xi(v)$ of a vertex $v$ is drawn next to $v$, just like its label.  The decoration of an edge can be drawn on the edge.  If the decorations on the edges are understood from the context, we will omit them from the picture of the decorated graph.  In what follows, we sometimes write $G$ to denote a $\sfP$-decorated graph $(G,\xi)$ if the decorating function $\xi$ is understood from the context.

To each $\sfP$-decorated graph $(G, \xi)$, there is an associated element
\begin{equation}
\label{grapheval}
ev((G, \xi)) \in \elt(\sfP),
\end{equation}
called the \textbf{evaluation} of $(G, \xi)$, which is defined as follows \cite{markl06,markl07}.  (It is called a \emph{contraction along $G$} in \cite{markl07}.)  First suppose that $G$ has only one connected component.  In this case, we can compose the decorations $\xi(v)$ $(v \in v(G))$ in $\sfP$ according to the graph $G$ using the colored PROP structure of $\sfP$.  For example, if $G$ is the graph in \eqref{35graph} and if
\[
\xi(v_i) = \alpha_i \in \elt(\sfP) \quad \text{for} \quad 1 \leq i \leq 7,
\]
then
\begin{equation}
\label{exampleeval}
ev((G,\xi)) = \sigma_1\left[\left(\alpha_7 \otimes \alpha_2 \otimes \texttt{1}_{d} \right) \circ \tau\left(\alpha_4 \otimes \alpha_5 \otimes \alpha_1\right) \circ \left(\alpha_3 \otimes \alpha_6\right)\right]\sigma_2.
\end{equation}
Here $\sigma_1$, $\sigma_2$, and $\tau$ are the permutations
\[
\sigma_1 = \binom{1, 2, 3, 4, 5}{3, 5, 1, 4, 2}, \quad \sigma_2^{-1} = \binom{1, 2, 3}{2, 3, 1}, \quad \tau = \binom{1, 2, 3, 4, 5}{1, 3, 2, 4, 5}.
\]
So $\sigma_1$ and $\sigma_2$ are the permutations at the top and the bottom of $G$, and $\tau$ is the permutation for the only crossing in $G$.  The element $\texttt{1}_{d} \in \sfP\binom{d}{d}$ is the unit in $\sfP$ corresponding to the color $d \in \frakC$, which is the output profile of $\xi(v_1) = \alpha_1$.  That the element $ev((G,\xi))$ \eqref{exampleeval} makes sense in $\sfP$ follows from the color-matching property \eqref{colormatching}.

In the general case, suppose that
\[
G = G^1 \sqcup \cdots \sqcup G^r.
\]
Then its \textbf{evaluation} is defined as the horizontal composition
\begin{equation}
\label{evalG}
ev((G,\xi)) = ev((G^1,\xi^1)) \otimes \cdots \otimes ev((G^r,\xi^r)) \in \sfP,
\end{equation}
where $\xi^i$ is the restriction of the decorating function $\xi$ to (the vertices and edges in) the connected component $G^i$.

Evaluations of $\sfP$-decorated graphs give us a way to keep track of the reduction laws in $\sfP$.  In other words, every $\sfP$-decorated graph $(G,\xi)$ gives a reduction law in $\sfP$, such as \eqref{exampleeval}, via evaluation.  Conversely, every reduction law in $\sfP$ can be represented as a $\sfP$-decorated graph whose evaluation gives the original equation.

\subsection{Slice construction for colored PROPs}
\label{subsec:sliceconstruction}

Now we define certain sets that constitute the components of the slice PROP $\sfP^+$ of the unital $\frakC$-colored PROP $\sfP$.

Suppose that $\alpha_i$ $(1 \leq i \leq k)$ and $\beta_j$ $(1 \leq j \leq l)$ are elements in $\sfP$ and that
\[
k_1 + \cdots + k_l = k
\]
is a partition of $k$ with each $k_j \geq 1$.  Define
\begin{equation}
\label{pbar+}
\pbar^+_{(k_1, \ldots , k_l)}\left(\beta_1, \ldots , \beta_l; \alpha_1, \ldots , \alpha_k\right)
\end{equation}
to be the set of $\sfP$-decorated graphs $(G, \xi)$ in which:
\begin{enumerate}
\item
$G$ has $l$ connected components $G^j$ $(1 \leq j \leq l)$;
\item
$G^j$ has $k_j$ vertices $(1 \leq j \leq l)$;
\item
for $1 \leq j \leq l$ and $1 \leq r \leq k_j$, one has
\begin{eqnarray}
\xi\left(v^j_r\right) &=& \alpha_{k_1 + \cdots + k_{j-1} + r}, \label{pbardecoration} \\
ev((G^j,\xi^j)) &=& \beta_j.\label{pbareval}
\end{eqnarray}
\end{enumerate}

The condition \eqref{pbardecoration} means that the labeled vertices in $G^1$, $G^2$, etc., are decorated by the elements $\alpha_1, \alpha_2$, etc., in this order.  The condition \eqref{pbareval} says that $\beta_1$ is a composite of $\alpha_1, \ldots , \alpha_{k_1}$ in $\sfP$, where the composition is expressed by the graph $G^1$, and similarly for the other $\beta_j$.  Note that \eqref{evalG} and \eqref{pbareval} together imply
\[
ev((G,\xi)) = \beta_1 \otimes \cdots \otimes \beta_l
\]
for each
\[
(G,\xi) \in \pbar^+_{(k_1, \ldots , k_l)}\left(\beta_1, \ldots , \beta_l; \alpha_1, \ldots , \alpha_k\right).
\]
In particular, the $\sfP$-decorated graph $(G,\xi)$ gives a way of expressing $\beta_1 \otimes \cdots \otimes \beta_l$ as a composite (horizontally and vertically, possibly with permutations) of $\alpha_1, \ldots , \alpha_k$ in $\sfP$.

There is another intermediate set that we need to define before $\sfP^+$.  Suppose that $\alpha_i$ $(1 \leq i \leq k)$ and $\beta$ are elements in $\sfP$.  Define the set
\begin{equation}
\label{pbardef}
\pbar^+\binom{\beta}{\alpha_1, \ldots , \alpha_k} = \coprod_{\substack{l \geq 1 \\ k = k_1 + \cdots + k_l \\ \beta = \beta_1 \otimes \cdots \otimes \beta_l}} \pbar^+_{(k_1, \ldots , k_l)}\left(\beta_1, \ldots , \beta_l; \alpha_1, \ldots , \alpha_k\right).
\end{equation}
This disjoint union is taken over:
\begin{itemize}
\item
all integers $l \geq 1$, and for each $l$,
\item
all possible partitions
\[
k = k_1 + \cdots + k_l
\]
of $k$ with each $k_j \geq 1$ and
\item
all possible decompositions of $\beta$ as
\[
\beta = \beta_1 \otimes \cdots \otimes \beta_l
\]
in $\sfP$.
\end{itemize}
The summand $\pbar^+_{(k_1, \ldots , k_l)}\left(\beta_1, \ldots , \beta_l; \alpha_1, \ldots , \alpha_k\right)$ is defined above \eqref{pbar+}.  Note that the elements in the set \eqref{pbardef} are $\sfP$-decorated graphs corresponding, via evaluations, exactly to the ways of expressing $\beta$ as a composite (horizontally and vertically, possibly with permutations) of $\alpha_1, \ldots , \alpha_k$ in $\sfP$.  In other words,  $\pbar^+\binom{\beta}{\alpha_1, \ldots , \alpha_k}$ \emph{is the set of all possible reduction laws in $\sfP$ relating $\alpha_1, \ldots , \alpha_k$ to $\beta$}.

Now we define a unital $\elt(\sfP)$-colored PROP $\sfP^+$ as follows.  Pick any $\alpha_i$ $(1 \leq i \leq s)$ and $\beta_j$ $(1 \leq j \leq r)$ in $\elt(\sfP)$.  Then the component of $\sfP^+$ corresponding to the $\elt(\sfP)$-profiles
\begin{equation}
\label{Pprofiles}
\binom{\ubeta}{\ualpha} = \binom{\beta_1, \ldots , \beta_r}{\alpha_1, \ldots , \alpha_s}
\end{equation}
is defined as
\begin{multline}
\label{p+def}
\sfP^+\binom{\ubeta}{\ualpha} = \coprod_{\substack{s = s_1 + \cdots + s_r \\ \sigma \in \Sigma_r,\, \tau \in \Sigma_s}} \pbar^+\binom{\beta_{\sigma(1)}}{\alpha_{\tau^{-1}(1)}, \ldots , \alpha_{\tau^{-1}(s_1)}} \times \cdots \\
\times \pbar^+\binom{\beta_{\sigma(r)}}{\alpha_{\tau^{-1}(s_1 + \cdots + s_{r-1} + 1)}, \ldots , \alpha_{\tau^{-1}(s)}}.
\end{multline}
This disjoint union is taken over:
\begin{itemize}
\item
all possible partitions
\[
s = s_1 + \cdots + s_r
\]
of $s$ into $r$ integers with each $s_j \geq 1$ and
\item
all permutations $\sigma \in \Sigma_r$ and $\tau \in \Sigma_s$.
\end{itemize}
The sets $\pbar^+\binom{\beta_{\sigma(j)}}{\cdots}$ are defined above \eqref{pbardef}.  An element in $\sfP^+\binom{\ubeta}{\ualpha}$ is a sequence of $r$ $\sfP$-decorated graphs, in which the $j$th graph has evaluation some $\beta_i$.

\begin{theorem}
\label{p+prop}
There is a unital $\elt(\sfP)$-colored PROP $\sfP^+$ with components \eqref{p+def}.
\end{theorem}

\begin{proof}
First note that
\[
\sfP^+\binom{\ubeta}{\ualpha} = \sfP^+\binom{\pi\ubeta}{\ualpha\mu}
\]
for any $\sfP$-profiles $\binom{\ubeta}{\ualpha}$ \eqref{Pprofiles} and permutations $\pi \in \Sigma_r$ and $\mu \in \Sigma_s$.  The map
\[
(\pi; \mu) \colon \sfP^+\binom{\ubeta}{\ualpha} \to \sfP^+\binom{\pi\ubeta}{\ualpha\mu},
\]
which is part of the $\elt(\sfP)$-colored $\Sigma$-bimodule $\sfP^+$, is defined as the identity map.  This defines $\sfP^+$ as an $\elt(\sfP)$-colored $\Sigma$-bimodule.

The horizontal composition
\[
\otimes \colon \sfP^+\binom{\ubeta}{\ualpha} \times \sfP^+\binom{\udelta}{\uepsilon} \to \sfP^+\binom{\ubeta, \udelta}{\ualpha, \uepsilon}
\]
in $\sfP^+$ is given by the obvious summand inclusion.  Graphically, the horizontal composition is the concatenation of two sequences of $\sfP$-decorated graphs, i.e., put them side-by-side.

The vertical composition
\[
\circ \colon \sfP^+\binom{\ubeta}{\ualpha} \times \sfP^+\binom{\ualpha}{\udelta} \to \sfP^+\binom{\ubeta}{\udelta}
\]
in $\sfP^+$ is defined on a typical summand as the map
\begin{equation}
\label{p+verticalcomp}
\begin{split}
& \left[\pbar^+\binom{\beta_{\sigma(1)}}{\alpha_{\tau^{-1}(1)}, \ldots , \alpha_{\tau^{-1}(s_1)}} \times \cdots \times \pbar^+\binom{\beta_{\sigma(r)}}{\alpha_{\tau^{-1}(s_1 + \cdots + s_{r-1} + 1)}, \ldots , \alpha_{\tau^{-1}(s)}}\right] \\
& \times \left[\pbar^+\binom{\alpha_{\pi(1)}}{\delta_{\mu^{-1}(1)}, \ldots , \delta_{\mu^{-1}(t_1)}} \times \cdots \times \pbar^+\binom{\alpha_{\pi(s)}}{\delta_{\mu^{-1}(t_1 + \cdots + t_{s-1} + 1)}, \ldots , \delta_{\mu^{-1}(t)}}\right] \\
& \xrightarrow{\circ} \left[\pbar^+\binom{\beta_{\sigma(1)}}{\delta_{\nu^{-1}(1)}, \ldots , \delta_{\nu^{-1}(\cdots)}} \times \cdots \times \pbar^+\binom{\beta_{\sigma(r)}}{\delta_{\nu^{-1}(\cdots)}, \ldots , \delta_{\nu^{-1}(t)}}\right] \subseteq \sfP^+\binom{\ubeta}{\udelta}.
\end{split}
\end{equation}
This map is given by \textbf{graph substitution}.  More precisely, suppose that
\[
\left((G_1,\xi_1), \ldots , (G_r,\xi_r); (H_1, \zeta_1), \ldots, (H_s, \zeta_s)\right)
\]
is a typical element in the domain of the map $\circ$ in \eqref{p+verticalcomp}.  Recall that an element
\[
(G_1,\xi_1) \in \pbar^+\binom{\beta_{\sigma(1)}}{\alpha_{\tau^{-1}(1)}, \ldots , \alpha_{\tau^{-1}(s_1)}}
\]
is a $\sfP$-decorated graph whose vertices are decorated by the indicated $\alpha$'s and whose evaluation is $\beta_{\sigma(1)}$.  We use the shorthand
\[
T_i = t_1 + t_2 + \cdots + t_i.
\]
There is a unique $i_1$ such that $\tau^{-1}(1) = \pi(i_1)$.  Now we replace the vertex $v$ in $(G_1,\xi_1)$ with decoration $\alpha_{\tau^{-1}(1)}$ by the $\sfP$-decorated graph
\[
(H_{i_1}, \zeta_{i_1}) \in \pbar^+\binom{\alpha_{\pi(i_1)}}{\delta_{\mu^{-1}(T_{i_1 - 1} + 1)}, \ldots , \delta_{\mu^{-1}(T_{i_1})}}.
\]
Repeat this graph substitution for the vertices decorated by $\alpha_{\tau^{-1}(2)}, \ldots , \alpha_{\tau^{-1}(s_1)}$ in $(G_1,\xi_1)$.  After these $s_1$ graph substitutions in $(G_1,\xi_1)$ and a suitable relabeling of the vertices, the resulting $\sfP$-decorated graph lies in
\[
\pbar^+\binom{\beta_{\sigma(1)}}{\delta_{\mu^{-1}(T_{i_1-1}+1)}, \ldots , \delta_{\mu^{-1}(T_{i_1})}, \ldots , \delta_{\mu^{-1}(T_{i_{s_1}-1}+1)}, \ldots , \delta_{\mu^{-1}(T_{i_{s_1}})}}.
\]
Now repeat the above graph substitution process for the other $(r-1)$ $\sfP$-decorated graphs $(G_2,\xi_2), \ldots , (G_r,\xi_r)$.  The resulting sequence of $r$ $\sfP$-decorated graphs lies in the desired target in \eqref{p+verticalcomp} when we define $\nu \in \Sigma_t$ by
\[
\nu^{-1}(1) = \mu^{-1}(T_{i_1-1}+1),\, \nu^{-1}(2) = \mu^{-1}(T_{i_1-1}+2),
\]
and so forth.

For an element $\alpha \in \elt(\sfP)$, the $\alpha$-colored unit in $\sfP^+$ is the $\sfP$-decorated $(m,n)$-graph
\begin{equation}
\setlength{\unitlength}{.6mm}
\label{1alpha}
\begin{picture}(90,35)(15,12)
\put(30,30){\circle*{2}}
\put(20,20){\vector(1,1){9.5}}
\put(40,20){\vector(-1,1){9.5}}
\put(30,30){\vector(-1,1){10}}
\put(30,30){\vector(1,1){10}}
\put(25,30){\makebox(0,0){$\alpha$}}
\put(31,23){\makebox(0,0){$\cdots$}}
\put(31,35){\makebox(0,0){$\cdots$}}
\put(20,16){\makebox(0,0){$1$}}
\put(40,16){\makebox(0,0){$n$}}
\put(20,43){\makebox(0,0){$1$}}
\put(40,43){\makebox(0,0){$m$}}
\put(5,30){\makebox(0,0){$\mathtt{1}_{\alpha} =$}}
\put(50,28){$\in \pbar^+_{(1)}(\alpha; \alpha) = \pbar^+\dbinom{\alpha}{\alpha} = \sfP^+\dbinom{\alpha}{\alpha}.$}
\end{picture}
\end{equation}
Here we are assuming that $\alpha \in \sfP\binom{\ud}{\uc}$ with $|\ud| = m$ and $|\uc| = n$.  The underlying  $(m,n)$-graph has only one vertex, which is decorated by $\alpha$.  The $n$ inputs are labeled $1, 2, \ldots , n$ from left to right.  The $j$th input is decorated by $c_j$ $(1 \leq j \leq n$).  The $m$ outputs of the graph are labeled $1, 2, \ldots , m$ from left to right, with the $i$th output decorated by $d_i$ $(1 \leq i \leq m)$.  Note that the decorations of the inputs and outputs are not displayed in the above graph.

The associativity of $\otimes$ and $\circ$ in $\sfP^+$ amount to the associativity of Cartesian products and graph substitutions, respectively.  The other $\elt(\sfP)$-colored PROP axioms (bi-equivariance, the interchange rule, and the unit axiom) are equally straightforward to check.
\end{proof}

\begin{remark}
\label{remark:P+chain}
The obvious analogue of Theorem ~\ref{p+prop} in the category $\kmodule$ of $\bk$-modules, where $\bk$ is a field of characteristic $0$, is also true.  Indeed, in this setting we take $\pbar^+_{(k_1, \ldots , k_l)}\left(\beta_1, \ldots , \beta_l; \alpha_1, \ldots , \alpha_k\right)$ to be the $\bk$-module generated by the $\sfP$-decorated graphs $(G,\xi)$ as specified on p.\pageref{pbar+}.  In \eqref{pbardef} and \eqref{p+def}, we replace $\coprod$ and $\times$ by direct sum $\oplus$ and tensor product $\otimes$ of $\bk$-modules, respectively.  The proof of Theorem ~\ref{p+prop} then goes through basically verbatim, giving a unital $\elt(\sfP)$-colored PROP $\sfP^+$ over $\kmodule$.
\end{remark}

\begin{proof}[Proof of Theorem ~\ref{thm:P+}]
Using the slice PROP $\sfP^+$ from Theorem ~\ref{p+prop}, it remains to establish the isomorphism \eqref{p+algebras} of categories.  We will construct two functors
\begin{equation}
\label{phipsi}
\partial \colon \propc/\sfP \rightleftarrows \alg(\sfP^+) \colon \int
\end{equation}
and observe that they are inverse isomorphisms of each other.  The choices of these notations will become clear below.

Let us begin with $\int$.  Suppose that $A = \{A_\alpha \colon \alpha \in \elt(\sfP)\}$ is a unital $\sfP^+$-algebra.  Given any elements $\alpha_i$ $(1 \leq i \leq s)$ and $\beta_j$ $(1 \leq j \leq r)$ in $\elt(\sfP)$, there is a $\sfP^+$-algebra structure map
\[
\lambda \colon \sfP^+\binom{\beta_1, \ldots , \beta_r}{\alpha_1, \ldots , \alpha_s} \times A_{\alpha_1} \times \cdots \times A_{\alpha_s} \to A_{\beta_1} \times \cdots \times A_{\beta_r}.
\]
If $\theta \in \sfP^+\binom{\ubeta}{\ualpha}$, then we write
\[
\lambda(\theta) \colon A_{\alpha_1} \times \cdots \times A_{\alpha_s} \to A_{\beta_1} \times \cdots \times A_{\beta_r}
\]
for the map induced by $\lambda$.

First we define $\int A$ as a $\frakC$-colored $\Sigma$-bimodule.  Given any $\frakC$-profiles $\ud = (d_1, \ldots , d_m)$ and $\uc = (c_1, \ldots , c_n)$, we define
\begin{equation}
\label{Adc}
\int A\binom{\ud}{\uc} = \coprod_{\alpha \in \sfP\binom{\ud}{\uc}} A_{\alpha}.
\end{equation}
In other words, $\int A$ at a typical pair of $\frakC$-profiles is obtained by ``integrating" the sets $A_\alpha$ for $\alpha \in \sfP\binom{\ud}{\uc}$.  Suppose that
\[
(\sigma; \tau) \colon \binom{\ud}{\uc} \to \binom{\sigma\ud}{\uc\tau}
\]
is a map of $\frakC$-profiles.  Then the map
\[
\int A(\sigma; \tau) \colon \int A\binom{\ud}{\uc} \to \int A\binom{\sigma\ud}{\uc\tau}
\]
is defined on a typical summand as the map
\begin{equation}
\label{Alambda}
A_\alpha \xrightarrow{\lambda(\sigma \mathtt{1}_{\alpha} \tau)} A_{\sigma\alpha\tau} \subseteq \int A\binom{\sigma\ud}{\uc\tau}.
\end{equation}
Here $\sigma\alpha\tau$ is the image of $\alpha$ under the map
\[
\sfP(\sigma;\tau) \colon \sfP\binom{\ud}{\uc} \to \sfP\binom{\sigma\ud}{\tau\uc}.
\]
The element $\sigma \mathtt{1}_{\alpha} \tau$ in $\sfP^+$ is the $\sfP$-decorated $(m,n)$-graph
\begin{equation}
\setlength{\unitlength}{.6mm}
\label{1alphatwisted}
\begin{picture}(50,50)(25,0)
\put(30,30){\circle*{2}}
\put(20,20){\vector(1,1){9.5}}
\put(40,20){\vector(-1,1){9.5}}
\put(30,30){\vector(-1,1){10}}
\put(30,30){\vector(1,1){10}}
\put(25,30){\makebox(0,0){$\alpha$}}
\put(31,23){\makebox(0,0){$\cdots$}}
\put(31,35){\makebox(0,0){$\cdots$}}
\put(12,16){\makebox(0,0){$\tau^{-1}(1)$}}
\put(48,16){\makebox(0,0){$\tau^{-1}(n)$}}
\put(15,43){\makebox(0,0){$\sigma(1)$}}
\put(48,43){\makebox(0,0){$\sigma(m)$}}
\put(-13,30){\makebox(0,0){$\sigma\mathtt{1}_{\alpha}\tau =$}}
\put(58,28){$\in \pbar^+_{(1)}(\sigma\alpha\tau; \alpha) = \sfP^+\dbinom{\sigma\alpha\tau}{\alpha}.$}
\end{picture}
\end{equation}
This is obtained from the $\sfP$-decorated graph $\mathtt{1}_{\alpha}$ \eqref{1alpha} by relabeling the inputs and outputs to $\tau^{-1}(1), \ldots , \tau^{-1}(n)$ and $\sigma(1), \ldots , \sigma(m)$, respectively.  Using the fact that $\lambda$ is compatible with the vertical composition in $\sfP^+$, it is straightforward to check that \eqref{Alambda} satisfies the required bi-equivariance axioms.

Next we define the vertical composition
\[
\circ \colon \int A\binom{\ud}{\uc} \times \int A\binom{\uc}{\ub} \to \int A\binom{\ud}{\ub}
\]
in $\int A$.  On a typical summand with $\alpha \in \sfP\binom{\ud}{\uc}$ and $\beta \in \sfP\binom{\uc}{\ub}$, this map is defined as
\[
A_\alpha \times A_\beta \xrightarrow{\lambda(G_{\alpha \circ \beta})} A_{\alpha \circ \beta} \subseteq \int A\binom{\ud}{\ub}.
\]
Here $G_{\alpha \circ \beta}$ is the $\sfP$-decorated graph
\begin{equation}
\setlength{\unitlength}{.6mm}
\label{Galphabeta}
\begin{picture}(50,58)(25,10)
\put(30,30){\circle*{2}}
\put(30,50){\circle*{2}}
\put(20,20){\vector(1,1){9.5}}
\put(40,20){\vector(-1,1){9.5}}
\put(30,50){\vector(-1,1){10}}
\put(30,50){\vector(1,1){10}}
\put(28,48){\vector(1,1){1.5}}
\put(32,48){\vector(-1,1){1.5}}
\qbezier(30,30)(15,39)(28,48)
\qbezier(30,30)(45,39)(32,48)
\put(20,30){\makebox(0,0){$\beta$}}
\put(20,50){\makebox(0,0){$\alpha$}}
\put(31,23){\makebox(0,0){$\cdots$}}
\put(31,39){\makebox(0,0){$\cdots$}}
\put(31,55){\makebox(0,0){$\cdots$}}
\put(17,18){\makebox(0,0){$1$}}
\put(45,18){\makebox(0,0){$|\ub|$}}
\put(17,63){\makebox(0,0){$1$}}
\put(45,63){\makebox(0,0){$|\ud|$}}
\put(-13,44){\makebox(0,0){$G_{\alpha \circ \beta} =$}}
\put(65,42){$\in \pbar^+_{(2)}(\alpha\circ\beta; \alpha, \beta) \subseteq \sfP^+\dbinom{\alpha\circ\beta}{\alpha, \beta}.$}
\end{picture}
\end{equation}
In this $\sfP$-decorated $(|\ud|,|\ub|)$-graph, there are two vertices, in which the upper vertex is labeled $1$ and is decorated by $\alpha$.  The lower vertex is labeled $2$ and is decorated by $\beta$.  The $|\ud|$ outputs are labeled $1, 2, \ldots, |\ud|$ from left to right, and they are decorated by the colors $d_1, d_2, \ldots $ that constitute the $\frakC$-profile $\ud$.  Likewise, the $|\ub|$ inputs are labeled $1, 2, \ldots , |\ub|$ from left to right, and they are decorated by the colors $b_1, b_2, \ldots $ that constitute the $\frakC$-profile $\ub$.  The only other edges are the $|\uc|$ edges from the lower vertex to the upper vertex.  They are decorated from left to right by the colors $c_1, c_2, \ldots $ that constitute the $\frakC$-profile $\uc$.

The horizontal composition
\[
\otimes \colon \int A\binom{\ud}{\uc} \times \int A\binom{\ub}{\ua} \to \int A\binom{\ud,\ub}{\uc,\ua}
\]
in $\int A$ is defined on a typical summand with $\alpha \in \sfP\binom{\ud}{\uc}$ and $\beta \in \sfP\binom{\ub}{\ua}$ as the map
\[
A_\alpha \times A_\beta \xrightarrow{\lambda(G_{\alpha\otimes\beta})} A_{\alpha \otimes \beta} \subseteq \int A\binom{\ud,\ub}{\uc,\ua}.
\]
Here $G_{\alpha\otimes\beta}$ is the $\sfP$-decorated $(|\ud| + |\ub|, |\uc| + |\ua|)$-graph
\begin{equation}
\setlength{\unitlength}{.6mm}
\label{Galphatensorbeta}
\begin{picture}(110,40)(15,12)
\put(30,30){\circle*{2}}      
\put(20,20){\vector(1,1){9.5}}
\put(40,20){\vector(-1,1){9.5}}
\put(30,30){\vector(-1,1){10}}
\put(30,30){\vector(1,1){10}}
\put(22,30){\makebox(0,0){$\alpha$}}
\put(31,23){\makebox(0,0){$\cdots$}}
\put(31,35){\makebox(0,0){$\cdots$}}
\put(20,16){\makebox(0,0){$1$}}
\put(40,16){\makebox(0,0){$|\uc|$}}
\put(20,43){\makebox(0,0){$1$}}
\put(40,43){\makebox(0,0){$|\ud|$}}
\put(60,30){\circle*{2}}      
\put(50,20){\vector(1,1){9.5}}
\put(70,20){\vector(-1,1){9.5}}
\put(60,30){\vector(-1,1){10}}
\put(60,30){\vector(1,1){10}}
\put(53,30){\makebox(0,0){$\beta$}}
\put(61,23){\makebox(0,0){$\cdots$}}
\put(61,35){\makebox(0,0){$\cdots$}}
\put(50,16){\makebox(0,0){$1$}}
\put(70,16){\makebox(0,0){$|\ua|$}}
\put(50,43){\makebox(0,0){$1$}}
\put(70,43){\makebox(0,0){$|\ub|$}}
\put(-15,30){\makebox(0,0){$G_{\alpha\otimes\beta} = \mathtt{1}_\alpha \sqcup \mathtt{1}_\beta =$}}
\put(80,28){$\in \pbar^+_{(1,1)}(\alpha,\beta; \alpha,\beta) \subseteq \sfP^+\dbinom{\alpha\otimes\beta}{\alpha,\beta}.$}
\end{picture}
\end{equation}
In other words, the $\sfP$-decorated graph $G_{\alpha\otimes\beta}$ has two connected components, $\mathtt{1}_\alpha$ and $\mathtt{1}_\beta$, which are defined in \eqref{1alpha}.

The associativity and bi-equivariance of the vertical and the horizontal compositions in $\int A$ are easy to check.  The interchange rule is an immediate consequence of the definitions \eqref{Galphabeta} and \eqref{Galphatensorbeta}.  So we have a (non-unital) $\frakC$-colored PROP $\int A$.

We now define a $\frakC$-colored PROP morphism
\[
f \colon \int A \to \sfP.
\]
Given any $\frakC$-profiles $\ud$ and $\uc$, $f$ is defined on a typical summand with $\alpha \in \sfP\binom{\ud}{\uc}$ as the unique map
\begin{equation}
\label{Aalphaf}
A_\alpha \xrightarrow{f} \{\alpha\} \subseteq \sfP\binom{\ud}{\uc}.
\end{equation}
It is straightforward to check that this defines a morphism of $\frakC$-colored PROPs.  The naturality of the construction
\[
A \mapsto \left(\int A \xrightarrow{f} \sfP\right)
\]
is clear, so we have defined the functor
\[
\int \colon \alg(\sfP^+) \to \propc/\sfP
\]
in \eqref{phipsi}.

Next we define the functor $\partial$ in \eqref{phipsi}.  So let
\[
g \colon \sfQ \to \sfP
\]
be a $\frakC$-colored PROP over $\sfP$.  First we define the underlying $\elt(\sfP)$-graded set of $\partial\sfQ$.  For each element $\alpha \in \sfP\binom{\ud}{\uc}$, we take the pre-image
\begin{equation}
\label{qalpha}
\partial\sfQ_\alpha = g^{-1}(\alpha) \subseteq \sfQ\binom{\ud}{\uc}.
\end{equation}
Then we have
\begin{equation}
\label{qdccoprod}
\sfQ\binom{\ud}{\uc} = \coprod_{\alpha \in \sfP\binom{\ud}{\uc}} \partial\sfQ_\alpha.
\end{equation}
In other words, the sets $\partial\sfQ_\alpha$ are obtained by ``dividing" the sets $\sfQ\binom{\ud}{\uc}$.

To define the unital $\sfP^+$-algebra structure
\begin{equation}
\label{p+algq}
\rho \colon \sfP^+\binom{\beta_1, \ldots , \beta_r}{\alpha_1, \ldots , \alpha_s} \times \partial\sfQ_{\alpha_1} \times \cdots \times \partial\sfQ_{\alpha_s} \to \partial\sfQ_{\beta_1} \times \cdots \times \partial\sfQ_{\beta_r}
\end{equation}
on $\partial\sfQ$, let $G = (G_1, \ldots , G_r)$ be an element in
\[
\pbar^+\binom{\beta_{\sigma(1)}}{\alpha_{\tau^{-1}(1)}, \ldots , \alpha_{\tau^{-1}(s_1)}} \times \cdots \times \pbar^+\binom{\beta_{\sigma(r)}}{\alpha_{\tau^{-1}(s_1 + \cdots + s_{r-1} + 1)}, \ldots , \alpha_{\tau^{-1}(s)}},
\]
which is a typical summand of $\sfP^+\binom{\ubeta}{\ualpha}$ \eqref{p+def}.  Also let $q_i \in \partial\sfQ_{\alpha_i}$ for $1 \leq i \leq s$.  The element
\[
\rho\left(G, q_1, \ldots , q_s\right) \in \partial\sfQ_{\beta_1} \times \cdots \times \partial\sfQ_{\beta_r}
\]
is defined by \textbf{decoration replacements} and \textbf{evaluations}:  In the sequence $G$ of $r$ $\sfP$-decorated graphs, replace the decoration $\alpha_i$ by $q_i$ for each $i$.  The result
\[
G' = (G'_1, \ldots , G'_r)
\]
is a sequence of $r$ $\sfQ$-decorated graphs because
\[
g(q_i) = \alpha_i
\]
for each $i$, so $q_i$ and $\alpha_i$ have the same $\frakC$-profiles.  For $1 \leq j \leq r$, since $g$ is a map of $\frakC$-colored PROPs, we have
\[
g\left(ev(G'_j)\right) = ev(G_j) = \beta_{\sigma(j)}.
\]
It follows from the definition \eqref{qalpha} that
\[
ev(G'_j) \in \partial\sfQ_{\beta_{\sigma(j)}},
\]
so we have
\[
\left(ev(G'_1), \ldots , ev(G'_r)\right) \in \partial\sfQ_{\beta_{\sigma(1)}} \times \cdots \times \partial\sfQ_{\beta_{\sigma(r)}}.
\]
Now we define
\[
\begin{split}
\rho\left(G, q_1, \ldots , q_s\right) &= \sigma^{-1}\left(ev(G'_1), \ldots , ev(G'_r)\right)\\
&= \left(ev(G'_{\sigma^{-1}(1)}), \ldots , ev(G'_{\sigma^{-1}(r)})\right) \in \partial\sfQ_{\beta_1} \times \cdots \times \partial\sfQ_{\beta_r}.
\end{split}
\]
Using the $\elt(\sfP)$-colored PROP structure of $\sfP^+$ (Theorem ~\ref{p+prop}), it is easy to check that $\rho$ \eqref{p+algq} gives $\partial\sfQ$ the structure of a unital $\sfP^+$-algebra.  The naturality of the construction
\[
\left(g \colon \sfQ \to \sfP\right) \mapsto \left\{\partial\sfQ_\alpha \colon \alpha \in \elt(\sfP)\right\}
\]
is also clear, so we have defined the functor $\partial$ in \eqref{phipsi}.

One observes from \eqref{Adc}, \eqref{Aalphaf}, \eqref{qalpha}, \eqref{qdccoprod} and the associated structure maps that $\partial$ and $\int$ are indeed inverses of each other, hence both of them are isomorphisms.
\end{proof}

\begin{remark}
\label{remark:P+alg}
As in Remark ~\ref{remark:P+chain}, the obvious analogue of Theorem ~\ref{thm:P+} in the category of $\bk$-modules is also true, where $\bk$ is a field of characteristic $0$.  Indeed, to adapt the above proof of Theorem ~\ref{thm:P+} to the case of $\bk$-modules, we merely need to replace $\times$ and $\coprod$ by tensor product $\otimes$ and direct sum $\oplus$ of $\bk$-modules, respectively.
\end{remark}

\section{Propertopes and propertopic sets}
\label{sec:propertopes}

In \S\ref{subsec:Ppropertope}, starting from a unital $\frakC$-colored PROP $\sfP$, we define the category of $\sfP$-\textbf{propertopes} $\bP(\sfP)$.  The objects in $\bP(\sfP)$ -- the $\sfP$-propertopes -- are obtained by iterating the slice PROP construction (Theorem ~\ref{thm:P+}).  These $\sfP$-propertopes serve as our shapes of higher cells.  If $\sfO$ is a colored operad, then the Baez-Donald $\sfO$-opetopes \cite{bd1} are among our $\sfO_{prop}$-propertopes (Remark ~\ref{remark:opetope}).

In \S\ref{subsec;draw} we discuss the combinatorics of drawing $\sfP$-propertopes.  In order to recover an $n$-dimensional $\sfP$-propertope, it suffices to remember a sequence of $n$-level \textbf{metagraphs}.  The top $n-1$ levels of such a metagraph contains only graphs, so it is completely combinatorial.  The bottom level contains certain elements in $\sfP$.

In \S\ref{subsec:propertopicset} we define \textbf{$\sfP$-propertopic sets} as presheaves of sets on the category $\bP(\sfP)$ of $\sfP$-propertopes.  The transition from $\sfP$-propertopes to $\sfP$-propertopic sets is somewhat analogous to going from the category $\Delta$ of non-empty finite ordered sets and non-decreasing functions to simplicial sets $\set^{\Delta^{op}}$.  We will use the familiar setting of simplicial sets as our guide as we make certain definitions.  However, one should keep in mind that we are using simplicial sets only as a soft analogy.  Our main goal is to construct higher dimensional $\sfP$-algebras using higher PROPs.  Therefore, at some points our choices are motivated by higher dimensional algebras and not by analogy with simplicial sets.

In \S\ref{subsec:horns} we define $\sfP$-propertopic analogues of cells, horns, and boundaries.  These objects are used in \S\ref{subsec:propertopicfibration} to describe fibrations of $\sfP$-propertopic sets and, in particular, fibrant $\sfP$-propertopic sets.  In the next section, we will use the concepts of cells, horns, and boundaries to define higher dimensional $\sfP$-algebras.  As we will see, weak-$\omega$ $\sfP$-algebras are exactly the fibrant $\sfP$-propertopic sets.  For $n < \infty$, weak-$n$ $\sfP$-algebras are analogous to homotopy $n$-types.

\subsection{$\sfP$-propertopes}
\label{subsec:Ppropertope}

Fix a unital $\frakC$-colored PROP $\sfP$ over $\set$ for the rest of this section.  By Theorem ~\ref{thm:P+} we know that its slice PROP $\sfP^+$ is unital and $\elt(\sfP)$-colored.  So we can apply the slice PROP construction to $\sfP^+$, 
and so forth.

\begin{definition}
\label{def:Pn+}
Set
\[
\elt(\sfP^{(-1)+}) = \frakC, \, \sfP^{0+} = \sfP,
\]
and inductively,
\[
\sfP^{n+} = (\sfP^{(n-1)+})^+
\]
for $n \geq 1$.  
The elements in $\elt(\sfP^{(n-1)+})$ are called \textbf{$n$-dimensional $\sfP$-propertopes}.  The \textbf{category of $\sfP$-propertopes}, denoted $\bP(\sfP)$, has the $n$-dimensional $\sfP$-propertopes $(n \geq 0)$ as objects.  Its morphisms are defined below.
\end{definition}

Note that $\frakC$ is the set of $0$-dimensional $\sfP$-propertopes, and $\elt(\sfP)$ is the set of $1$-dimensional $\sfP$-propertopes.  For $n \geq 1$, $\sfP^{n+}$ is a unital $\elt(\sfP^{(n-1)+})$-colored PROP.  The $(n+1)$-dimensional $\sfP$-propertopes (i.e., elements in the set $\elt(\sfP^{n+})$) are finite non-empty sequences of $\sfP^{(n-1)+}$-decorated graphs (Definition ~\ref{def:decoratedgraph}).  By Theorem ~\ref{thm:P+} there is a canonical isomorphism
\[
\prop^{\elt(\sfP^{(n-2)+})}/\sfP^{(n-1)+} \cong \alg(\sfP^{n+})
\]
of categories.  In other words, $\sfP^{n+}$ is the unital $\elt(\sfP^{(n-1)+})$-colored PROP for $\elt(\sfP^{(n-2)+})$-colored PROPs over $\sfP^{(n-1)+}$.

\begin{remark}
\label{remark:opetope}
Let $\sfO$ be a unital $\frakC$-colored operad.  Iterating the Baez-Dolan slice construction $\sfO^+$ \cite{bd1}, one calls the elements in $\sfO^{(n-1)+}$ \textbf{$n$-dimensional $\sfO$-opetopes}.  These $\sfO$-opetopes are actually among the $\sfO_{prop}$-propertopes (Definition ~\ref{def:Pn+}), where $\sfO_{prop}$ is the unital $\frakC$-colored PROP generated by $\sfO$ (Theorem ~\ref{operadpropadj}).  In fact, $\sfO$ and $\sfO_{prop}$ are both $\frakC$-colored, so $0$-dimensional $\sfO$-opetopes and $\sfO_{prop}$-propertopes are both the elements in $\frakC$.  We also know that $\sfO_{prop}$ contains all the elements in $\sfO$ \eqref{Oprop}, so $1$-dimensional $\sfO$-opetopes are among the $1$-dimensional $\sfO_{prop}$-propertopes.

Going one dimension higher, an element in the slice operad $\sfO^+$ is a certain $\sfO$-decorated tree.  On the other hand, an element in the slice PROP $(\sfO_{prop})^+$ is a finite non-empty sequence of $\sfO_{prop}$-decorated graphs.  Among these elements are the sequences of length $1$ of $\sfO_{prop}$-decorated graphs.  We can restrict ourselves further to just the elements in $\sfO$ (as opposed to all of $\sfO_{prop}$) for decorating vertices and to trees (as opposed to graphs). So the elements in $\sfO^+$ are among the elements in $(\sfO_{prop})^+$.  In other words, $2$-dimensional $\sfO$-opetopes are among the $2$-dimensional $\sfO_{prop}$-propertopes.  Inductively, essentially the same discussion applies to the elements in the higher operads $\sfO^{n+}$ and the higher PROPs $(\sfO_{prop})^{n+}$ for $n \geq 2$.  For more discussion of the Baez-Dolan opetopes, the reader is referred to \cite{bd1,cheng1,cheng2,cheng3}.
\end{remark}

Now we define the morphisms in the category $\bP(\sfP)$ of $\sfP$-propertopes.  One can think of an $n$-dimensional $\sfP$-propertope as a kind of generalized $n$-simplex.  A usual $n$-simplex $\gamma$ has $n+1$ faces $d_i\gamma$ $(0 \leq i \leq n)$, which are $(n-1)$-simplices.  Each of these faces $d_i\gamma$ has its own faces, which are $(n-2)$-simplices, and so on.  There are also some simplicial identities that the face maps must satisfy.  The morphisms between the $\sfP$-propertopes are similarly generated by certain face maps.

First suppose that $n \geq 1$ and that $\gamma \in \elt(\sfP^{(n-1)+})$ is an $n$-dimensional $\sfP$-propertope and $\alpha \in \elt(\sfP^{(n-2)+})$ is an $(n-1)$-dimensional $\sfP$-propertope.  To every occurrence of $\alpha$ as an input or output color of $\gamma$, we associate to it a unique morphism
\[
\gamma \to \alpha \in \bP(\sfP).
\]
In other words, if
\begin{equation}
\label{gammaPn-1}
\gamma \in \sfP^{(n-1)+}\binom{\beta_1, \ldots , \beta_s}{\alpha_1, \ldots , \alpha_r},
\end{equation}
then there is exactly one morphism $\gamma \to \alpha$ for every $\beta_i = \alpha$ or $\alpha_j = \alpha$.  So if $k$ of the $\beta_i$ are equal to $\alpha$ and if $l$ of the $\alpha_j$ are equal to $\alpha$, then the set $\bP(\sfP)(\gamma,\alpha)$ of morphisms $\gamma \to \alpha$ has cardinality $k + l$.  The diagram
\begin{equation}
\label{gammafacemaps}
\setlength{\unitlength}{.6mm}
\begin{picture}(60,30)(0,-5)
\put(35,20){\makebox(0,0){$\gamma$}}
\put(10,0){\makebox(0,0){$\beta_1$}}
\put(20,0){\makebox(0,0){$\cdots$}}
\put(30,0){\makebox(0,0){$\beta_s$}}
\put(40,0){\makebox(0,0){$\alpha_1$}}
\put(50,0){\makebox(0,0){$\cdots$}}
\put(60,0){\makebox(0,0){$\alpha_r$}}
\put(33,15){\vector(-1,-4){2.8}}
\put(38,15){\vector(1,-4){2.8}}
\put(30,15){\vector(-4,-3){15}}
\put(40,15){\vector(4,-3){15}}
\end{picture}
\end{equation}
depicts all $s+r$ morphisms in $\bP(\sfP)$ from $\gamma$ \eqref{gammaPn-1} to $(n-1)$-dimensional $\sfP$-propertopes.  A morphism of the form
\begin{equation}
\label{outfacemap}
g_j \colon \gamma \to \beta_j \quad (1 \leq j \leq s),
\end{equation}
from a $\sfP$-propertope $\gamma$ to one of its output colors, is called an \textbf{out-face map}.  Likewise, a morphism of the form
\begin{equation}
\label{infacemap}
f_i \colon \gamma \to \alpha_i \quad (1 \leq i \leq r),
\end{equation}
from a $\sfP$-propertope $\gamma$ to one of its input colors, is called an \textbf{in-face map}.  A \textbf{face map} is either an in-face map or an out-face map.

\emph{The face maps generate all the morphisms in $\bP(\sfP)$ subject to four consistency conditions to be specified below}. In other words, suppose that $\gamma$ is an $n$-dimensional $\sfP$-propertope and $\delta$ is a $k$-dimensional $\sfP$-propertope.
\begin{enumerate}
\item
If $k \geq n$, then
\[
\bP(\sfP)(\gamma,\delta) =
\begin{cases}
\{1_\gamma\} & \text{if } \gamma = \delta,\\
\varnothing & \text{otherwise}.
\end{cases}
\]
\item
If $k < n$, then an element in $\bP(\sfP)(\gamma,\delta)$ is a sequence
\begin{equation}
\label{Pmorphism}
\gamma = \delta_n \xrightarrow{h_n} \delta_{n-1} \xrightarrow{h_{n-1}} \cdots \xrightarrow{h_{k+2}} \delta_{k+1} \xrightarrow{h_{k+1}} \delta_k = \delta,
\end{equation}
in which each
\[
h_l \colon \delta_l \to \delta_{l-1}
\]
is a face map.  So each $\delta_l \in \elt(\sfP^{(l-1)+})$ is an $l$-dimensional $\sfP$-propertope.  Each map $h_l$ records a specific occurrence of $\delta_{l-1} \in \elt(\sfP^{(l-2)+})$ as an input or output color of $\delta_l$.
\end{enumerate}
These morphisms are subject to the following four \textbf{consistency conditions}.
\begin{enumerate}
\item
\textbf{Horizonal Consistency Conditions:} For $n \geq 2$ and $(n-1)$-dimensional $\sfP$-propertopes
\[
\alpha \in \sfP^{(n-2)+}\binom{\uepsilon}{\udelta} \quad\text{and}\quad \beta \in \sfP^{(n-2)+}\binom{\uepsilon'}{\udelta'},
\]
all the \textbf{horizonal consistency diagrams} of face maps
\begin{equation}
\label{hconsistency1}
\SelectTips{cm}{10}
\SelectTips{cm}{10}
\xymatrix{
\alpha \ar[d]_-{\text{in}} & G_{\alpha \otimes \beta} \ar[l]_-{\text{in}} \ar[d]^-{\text{out}} \ar[r]^-{\text{in}} & \beta \ar[d]^-{\text{in}}\\
\delta_i & \alpha\otimes\beta \ar[l]_-{\text{in}} \ar[r]^-{\text{in}} & \delta'_k
}
\end{equation}
and
\begin{equation}
\label{hconsistency2}
\SelectTips{cm}{10}
\SelectTips{cm}{10}
\xymatrix{
\alpha \ar[d]_-{\text{out}} & G_{\alpha \otimes \beta} \ar[l]_-{\text{in}} \ar[d]^-{\text{out}} \ar[r]^-{\text{in}} & \beta \ar[d]^-{\text{out}}\\
\varepsilon_j & \alpha\otimes\beta \ar[l]_-{\text{out}} \ar[r]^-{\text{out}} & \varepsilon'_l
}
\end{equation}
are required to commute.  The $n$-dimensional $\sfP$-propertope
\[
G_{\alpha \otimes \beta} \in \sfP^{(n-1)+}\binom{\alpha\otimes\beta}{\alpha,\beta}
\]
is defined in \eqref{Galphatensorbeta}.  In the diagrams above (and below),
\[
\alpha \to \delta_i
\]
is the $i$th in-face map of $\alpha$,
\[
\alpha\otimes\beta \to \varepsilon'_l
\]
is the $(|\uepsilon| + l)$th out-face map of $\alpha\otimes\beta$, and so forth.
\item
\textbf{Vertical Consistency Condition}: For $n \geq 2$ and $(n-1)$-dimensional $\sfP$-propertopes
\[
\alpha \in \sfP^{(n-2)+}\binom{\uepsilon}{\udelta} \quad\text{and}\quad \beta \in \sfP^{(n-2)+}\binom{\udelta}{\ugamma},
\]
all the \textbf{vertical consistency diagrams} of face maps
\begin{equation}
\label{vconsistency}
\SelectTips{cm}{10}
\xymatrix{
\alpha \ar[d]_-{\text{out}} & G_{\alpha \circ \beta} \ar[l]_-{\text{in}} \ar[d]^-{\text{out}} \ar[r]^-{\text{in}} & \beta \ar[d]^-{\text{in}}\\
\varepsilon_j & \alpha\circ\beta \ar[l]_-{\text{out}} \ar[r]^-{\text{in}} & \gamma_i
}
\end{equation}
are required to commute.  The $n$-dimensional $\sfP$-propertope
\[
G_{\alpha\circ\beta} \in \sfP^{(n-1)+}\binom{\alpha\circ\beta}{\alpha,\beta}
\]
is defined in \eqref{Galphabeta}.
\item
\textbf{Unital Consistency Condition}: For $n \geq 1$ and $(n-1)$-dimensional $\sfP$-propertopes
\[
\alpha_i \in \elt(\sfP^{(n-2)+}) \quad (1 \leq i \leq m),
\]
all the \textbf{unital consistency diagrams}
\begin{equation}
\label{uconsistency}
\SelectTips{cm}{10}
\xymatrix{
\ttone_{\alpha_1} \otimes \cdots \otimes \ttone_{\alpha_m} \ar[d]_-{i^\text{th}\text{ in}} \ar@{=}[rr] & & \ttone_{\alpha_1} \otimes \cdots \otimes \ttone_{\alpha_m} \ar[d]^-{i^{\text{th}}\text{ out}}\\
\alpha_i \ar@{=}[rr] & & \alpha_i
}
\end{equation}
are required to commute.  The $n$-dimensional $\sfP$-propertope $\ttone_{\alpha_i}$ is defined in \eqref{1alpha}, and
\[
\ttone_{\alpha_1} \otimes \cdots \otimes \ttone_{\alpha_m} \in \sfP^{(n-1)+}\binom{\alpha_1, \ldots , \alpha_m}{\alpha_1, \ldots , \alpha_m}
\]
is a horizontal composite in $\sfP^{(n-1)+}$.
In the diagram \eqref{uconsistency}, the left and the right vertical maps are the $i$th in-face map and the $i$th out-face map, respectively.
\item
\textbf{Equivariance Consistency Condition}: For $n \geq 1$, an $n$-dimensional $\sfP$-propertope
\[
\gamma \in \sfP^{(n-1)+}\binom{\beta_1, \ldots , \beta_s}{\alpha_1, \ldots , \alpha_r} = \sfP\binom{\ubeta}{\ualpha},
\]
and permutations $\sigma \in \Sigma_s$ and $\tau \in \Sigma_r$, all the \textbf{equivariance consistency diagrams} of face maps
\begin{equation}
\label{econsistency}
\SelectTips{cm}{10}
\xymatrix{
\sigma\ttone_\gamma\tau \ar[rr]^-{\text{out}} \ar[d]_-{\text{in}} & & \sigma\gamma\tau \ar[d]^-{\sigma^{-1}(j)^{\text{th}}\text{ out}}\\
\gamma \ar[rr]_-{j^{\text{th}} \text{ out}} & & \beta_j
}
\end{equation}
are required to commute.  The $(n+1)$-dimensional $\sfP$-propertope
\[
\sigma\ttone_\gamma\tau \in \sfP^{n+}\binom{\sigma\gamma\tau}{\gamma}
\]
is defined in \eqref{1alphatwisted}.
\end{enumerate}

Composition of non-identity morphisms in $\bP(\sfP)$ is achieved by splicing together chains of face maps of the form \eqref{Pmorphism}.

\begin{remark}
\label{remark:PP'}
Note that we have defined $\bP(\sfP)$ as a quotient category.  First we defined the category $\bP(\sfP)'$ whose objects are the $\sfP$-propertopes and whose non-identity morphisms are finite chains of face maps \eqref{Pmorphism} without further conditions.  Then we obtained $\bP(\sfP)$ from $\bP(\sfP)'$ by imposing the four consistency conditions (i.e., by insisting that the diagrams \eqref{hconsistency1} -- \eqref{econsistency} be commutative).  There is a quotient functor
\begin{equation}
\label{piP}
\pi \colon \bP(\sfP)' \to \bP(\sfP),
\end{equation}
which is the identity map on objects, the $\sfP$-propertopes, and is surjective on maps.
\end{remark}

\begin{remark}
\label{PPkmodules}
Recall that Theorems ~\ref{thm:P+} and ~\ref{p+prop} are true with the category of $\bk$-modules in place of $\set$ (Remarks ~\ref{remark:P+chain} and ~\ref{remark:P+alg}).  It is easy to see that the above definition of the category $\bP(\sfP)$ of $\sfP$-propertopes also makes sense if $\sfP$ is a unital $\frakC$-colored PROP over $\bk$-modules.
\end{remark}

The horizontal, vertical, and equivariance consistency conditions (\eqref{hconsistency1}, \eqref{hconsistency2}, \eqref{vconsistency}, and \eqref{econsistency}) involve compositions of two face maps.  They are our $\sfP$-propertopic analogues of the simplicial identities
\[
d_id_j = d_{j-1}d_i \quad (i < j)
\]
in a simplicial set.  \emph{But why do these consistency conditions make sense?}  Consider, for example, the horizontal consistency conditions.  The commutativity of \eqref{hconsistency1} says that in $G_{\alpha\otimes\beta}$ \eqref{Galphatensorbeta}:
\begin{enumerate}
\item
the $i$th input color $\delta_i$ of $\alpha$ is also the $i$th input color of $\alpha\otimes\beta$, and
\item
the $k$th input color $\delta'_k$ of $\beta$ is also the $(|\udelta| + k)$th input color of $\alpha\otimes\beta$.
\end{enumerate}
This makes sense from the picture \eqref{Galphatensorbeta} and also from the definition of the horizontal composition, which gives
\[
\alpha \otimes \beta \in \sfP^{(n-2)+}\binom{\uepsilon,\uepsilon'}{\udelta,\udelta'}.
\]
Likewise, the commutativity of \eqref{hconsistency2} says essentially the same thing for the output colors of $\alpha$, $\beta$, and $\alpha \otimes \beta$.  The commutativity of the vertical and equivariance consistency diagrams (\eqref{vconsistency} and \eqref{econsistency}) can be similarly interpreted by looking at the definitions \eqref{Galphabeta} of $G_{\alpha\circ\beta}$ and \eqref{1alphatwisted} of $\sigma\ttone_\gamma\tau$.

\subsection{Combinatorics of $\sfP$-propertopes}
\label{subsec;draw}

Before moving on to the discussion of $\sfP$-propertopic sets, here we discuss the combinatorics of representing $\sfP$-propertopes graphically.  The discussion below about \textbf{metagraphs} can be regarded as a $\sfP$-propertopic generalization of the \emph{metatree} notation developed in \cite{bd1} for opetopes.  Other combinatorial descriptions of opetopes are given in \cite{bjkm,cheng4}.

Since we begin with a unital $\frakC$-colored PROP $\sfP$, we take for granted the elements in $\frakC$ and $\sfP$, i.e., the $0$-dimensional and $1$-dimensional $\sfP$-propertopes.  Our aim is to represent $n$-dimensional $\sfP$-propertopes for $n > 1$ in terms of elements in $\sfP$ and some purely combinatorial data.

We first consider a simple example involving a $3$-dimensional $\sfP$-propertope.  Consider the elements
\[
\alpha \in \sfP\binom{b_1,b_2,b_3,b_4}{a},\quad
\beta \in \sfP\binom{c_1,c_2}{b_1,b_2,b_3,b_4},\quad
\gamma \in \sfP\binom{d_1,d_2,d_3}{c_1,c_2}.
\]
Using the notations in \eqref{1alpha} and \eqref{Galphabeta}, we have the $2$-dimensional $\sfP$-propertopes
\[
\ttone_\gamma \in \sfP^+\binom{\gamma}{\gamma},\quad
G_{\beta\circ\alpha} \in \sfP^+\binom{\beta\circ\alpha}{\beta,\alpha},\quad
G_{\gamma\circ(\beta\circ\alpha)} \in \sfP^+\binom{\gamma\circ\beta\circ\alpha}{\gamma,\beta\circ\alpha}.
\]
In the horizontal composite
\[
\ttone_\gamma \otimes G_{\beta\circ\alpha} \in \sfP^{+}\binom{\gamma, \beta\circ\alpha}{\gamma,\beta,\alpha},
\]
the output profile is equal to the input profile in $G_{\gamma\circ(\beta\circ\alpha)}$, namely $(\gamma, \beta\circ\alpha)$.  Thus, the vertical composite
\[
G_{\gamma\circ\beta\circ\alpha} := G_{\gamma\circ(\beta\circ\alpha)} \circ (\ttone_\gamma \otimes G_{\beta\circ\alpha}) \in \sfP^+\binom{\gamma\circ\beta\circ\alpha}{\gamma,\beta,\alpha}
\]
makes sense.  Using the notation in \eqref{Galphabeta} again, we can thus form the $3$-dimensional $\sfP$-propertope
\[
\setlength{\unitlength}{.6mm}
\begin{picture}(50,58)(25,10)
\put(30,30){\circle*{2}}
\put(30,50){\circle*{2}}
\put(20,20){\vector(1,1){9.5}}
\put(40,20){\vector(-1,1){9.5}}
\put(30,20){\vector(0,1){9.5}}
\put(30,50){\vector(0,1){10}}
\put(28,48){\vector(1,1){1.5}}
\put(32,48){\vector(-1,1){1.5}}
\qbezier(30,30)(23,39)(28,48)
\qbezier(30,30)(37,39)(32,48)
\put(51,29){\makebox(0,0){$\ttone_\gamma \otimes G_{\beta\circ\alpha}$}}
\put(48,49){\makebox(0,0){$G_{\gamma\circ(\beta\circ\alpha)}$}}
\put(-23,44){\makebox(0,0){$G' = G_{G_{\gamma\circ(\beta\circ\alpha)} \circ (\ttone_\gamma \otimes G_{\beta\circ\alpha})} =$}}
\put(75,42){$\in \sfP^{2+}\dbinom{G_{\gamma\circ\beta\circ\alpha}}{G_{\gamma\circ(\beta\circ\alpha)}, \ttone_\gamma \otimes G_{\beta\circ\alpha}}$.}
\end{picture}
\]
In this graph, the bottom three edges are labeled $1$, $2$, and $3$ from left to right, and are decorated by $\gamma$, $\beta$, and $\alpha$, respectively.  The middle two edges are decorated by $\gamma$ and $\beta \circ\alpha$, respectively.  The top edge is decorated by $\gamma$.  The top (resp. bottom) vertex is labeled $1$ (resp. $2$) and is decorated by $G_{\gamma\circ(\beta\circ\alpha)}$ (resp. $\ttone_\gamma \otimes G_{\beta\circ\gamma}$).

We want to represent the $3$-dimensional $\sfP$-propertope $G'$ using graphs (Definition ~\ref{def:graph}) and elements in $\sfP$.  To simplify the graphs below, we will draw directed edges without the arrows, keeping in mind that they are always assumed to flow from the bottom to the top.  Also, we omit drawing the labels of the input or output edges if they are labeled consecutively $1, 2, \ldots$ from left to right.

First, from the $3$-dimensional $\sfP$-propertope $G'$, we obtain the following ``graph of graphs" or \textbf{metagraph} $\bM(G')$:
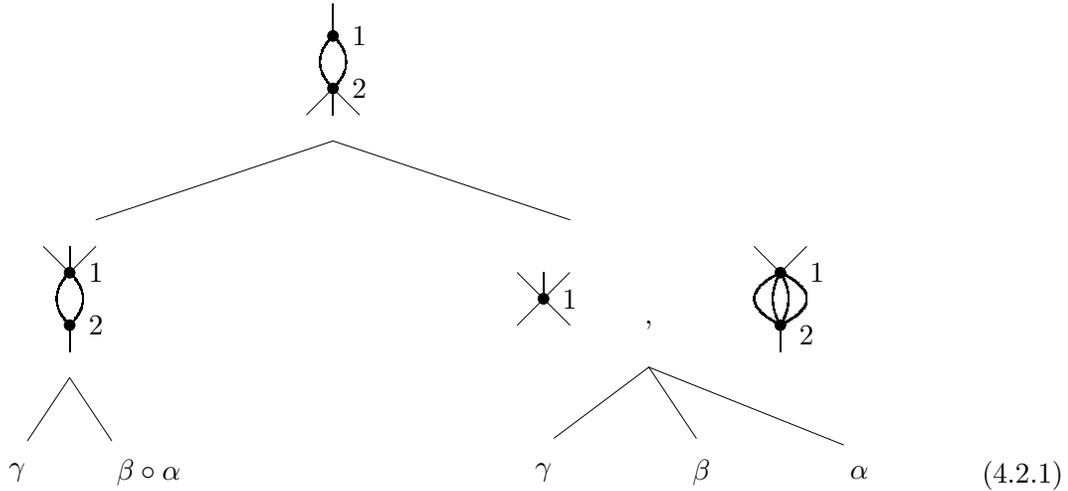
\begin{equation}
\label{graphofgraphs}
\setlength{\unitlength}{.7mm}
\begin{picture}(110,100)
\put(-35,2){\makebox(0,0){$\gamma$}}
\put(-10,2){\makebox(0,0){$\beta\circ\alpha$}}
\put(65,2){\makebox(0,0){$\gamma$}}
\put(95,2){\makebox(0,0){$\beta$}}
\put(125,2){\makebox(0,0){$\alpha$}}
\put(-25,30){\circle*{2}}
\put(-25,40){\circle*{2}}
\put(-25,25){\line(0,1){5}}
\put(-25,40){\line(-1,1){5}}
\put(-25,40){\line(0,1){5}}
\put(-25,40){\line(1,1){5}}
\qbezier(-25,30)(-30,35)(-25,40)
\qbezier(-25,30)(-20,35)(-25,40)
\put(-20,40){\makebox(0,0){$1$}}
\put(-20,30){\makebox(0,0){$2$}}
\put(65,35){\circle*{2}}
\put(60,30){\line(1,1){5}}
\put(70,30){\line(-1,1){5}}
\put(65,35){\line(-1,1){5}}
\put(65,35){\line(0,1){5}}
\put(65,35){\line(1,1){5}}
\put(70,35){\makebox(0,0){$1$}}
\put(110,30){\circle*{2}}
\put(110,40){\circle*{2}}
\put(110,25){\line(0,1){5}}
\put(110,40){\line(-1,1){5}}
\put(110,40){\line(1,1){5}}
\qbezier(110,30)(100,35)(110,40)
\qbezier(110,30)(107,35)(110,40)
\qbezier(110,30)(120,35)(110,40)
\qbezier(110,30)(113,35)(110,40)
\put(117,40){\makebox(0,0){$1$}}
\put(115,28){\makebox(0,0){$2$}}
\put(85,30){\makebox(0,0){,}}
\put(25,75){\circle*{2}}
\put(25,85){\circle*{2}}
\put(20,70){\line(1,1){5}}
\put(25,70){\line(0,1){5}}
\put(30,70){\line(-1,1){5}}
\put(25,85){\line(0,1){6}}
\qbezier(25,75)(20,80)(25,85)
\qbezier(25,75)(30,80)(25,85)
\put(30,85){\makebox(0,0){$1$}}
\put(30,75){\makebox(0,0){$2$}}
\put(-25,20){\line(-2,-3){8}}
\put(-25,20){\line(2,-3){8}}
\put(85,22){\line(-4,-3){18}}
\put(85,22){\line(2,-3){9}}
\put(85,22){\line(5,-2){37}}
\put(25,65){\line(-3,-1){45}}
\put(25,65){\line(3,-1){45}}
\end{picture}
\end{equation}
In the $3$-level metagraph $\bM(G')$ \eqref{graphofgraphs}, the top rocket-shaped graph is the underlying graph of $G'$.  In the middle row, the left-most graph with the shape of an upside-down rocket is the underlying graph of $G_{\gamma\circ(\beta\circ\alpha)}$, in which the vertex labeled $1$ (resp. $2$) is decorated by $\gamma$ (resp. $\beta\circ\alpha$).  Also in the middle row, the graph with only one vertex is the underlying graph of $\ttone_\gamma$.  The right-most graph is the underlying graph of $G_{\beta\circ\alpha}$, in which the vertex labeled $1$ (resp. $2$) is decorated by $\beta$ (resp. $\alpha$).  These two graphs on the right together, separated by a comma, is the underlying sequence of graphs of $\ttone_\gamma \otimes G_{\beta\circ\alpha}$.  In the bottom row are the elements in $\sfP$ that decorate the $2$-dimensional $\sfP$-propertopes $G_{\gamma\circ(\beta\circ\alpha)}$ and $\ttone_\gamma \otimes G_{\beta\circ\alpha}$.

Note that the large-scale shape of the metagraph $\bM(G')$ \eqref{graphofgraphs} is the following graph.
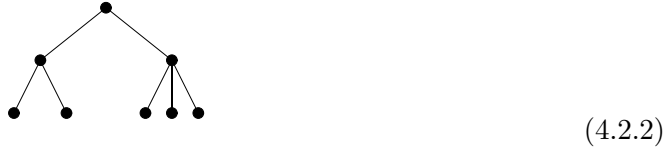
\begin{equation}
\label{largescaleshape}
\setlength{\unitlength}{.7mm}
\begin{picture}(50,35)
\put(10,5){\circle*{2}}
\put(20,5){\circle*{2}}
\put(35,5){\circle*{2}}
\put(40,5){\circle*{2}}
\put(45,5){\circle*{2}}
\put(15,15){\circle*{2}}
\put(40,15){\circle*{2}}
\put(27.5,25){\circle*{2}}
\put(15,15){\line(5,4){12.5}}
\put(40,15){\line(-5,4){12.5}}
\put(10,5){\line(1,2){5}}
\put(20,5){\line(-1,2){5}}
\put(35,5){\line(1,2){5}}
\put(40,5){\line(0,1){10}}
\put(45,5){\line(-1,2){5}}
\end{picture}
\end{equation}
From the perspectively of the graph \eqref{largescaleshape}, the description of the metagraph $\bM(G')$ \eqref{graphofgraphs} in the previous paragraph amounts to the following.  From a vertex $u$ in \eqref{largescaleshape}, the $i$th edge below it (from left to right, as always) extends to the underlying graph of the decoration of the $i$th vertex in $u$.  Of course, if $u$ is in the middle row, the $i$th edge below a vertex $u$ extends simply to the element in $\sfP$ decorating the $i$th vertex in $u$.

The $3$-level metagraph $\bM(G')$ \eqref{graphofgraphs} actually has enough information to uniquely determine the $3$-dimensional $\sfP$-propertope $G' \in \sfP^{2+}$.  Indeed, we can use the elements in the bottom row of the metagraph $\bM(G')$ to decorate the graphs above them.  We use the vertex labels of the graphs in the middle row to keep track of the vertex decorations.  The results are exactly the $2$-dimensional $\sfP$-propertopes $G_{\gamma\circ(\beta\circ\alpha)}$ and $\ttone_\gamma \otimes G_{\beta\circ\alpha}$ in $\sfP^+$.  In these two $\sfP$-decorated graphs, the decorations of the edges are uniquely determined by the input and output profiles of the vertex decorations.  Now we repeat the same process, starting at the middle row of \eqref{graphofgraphs}, which now contains the two elements in $\sfP^+$ from the previous step.  Using its vertex labels, we decorate the top graph in the metagraph $\bM(G')$ using the elements $G_{\gamma\circ(\beta\circ\alpha)}$ and $\ttone_\gamma \otimes G_{\beta\circ\alpha}$ in $\sfP^+$.  The result is exactly the element $G' \in \sfP^{2+}$.

In general, given an $n$-dimensional $\sfP$-propertope $\zeta \in \sfP^{(n-1)+}$ with $n \geq 2$, the procedure we used above for $G' \in \sfP^{2+}$ can be iterated and used on $\zeta$.  The example $G'$ illustrates that, in order to recover a $\sfP$-propertope $\zeta$, one needs to remember a finite number of graphs, which is purely combinatorial, and a finite number of elements in $\sfP$, which is algebraic.  The combinatorial data and the algebraic data fit together in a metagraph (or, more generally, a sequence of metagraphs), which uniquely determines the element $\zeta$.

Here is the procedure for obtaining the sequence of metagraphs $\bM(\zeta)$ of an $n$-dimensional $\sfP$-propertope $\zeta$ for $n \geq 2$.  Suppose that
\[
\zeta = (\zeta_1, \ldots , \zeta_r)  \in \sfP^{(n-1)+}
\]
is an $n$-dimensional $\sfP$-propertope, where each $\zeta_i$ is a $\sfP^{(n-2)+}$-decorated graph.  We consider each $\zeta_i$ separately.  Let $H$ be a typical connected component in $\zeta_i$.  Then $H$ gives rise to an $n$-level metagraph $\bM(H)$ as follows.
\begin{enumerate}
\item
First note that $H$ is a $\sfP^{(n-2)+}$-decorated graph.  At the top level (i.e., level $n$ counting from the bottom up) of the metagraph $\bM(H)$, draw the underlying graph of $H$.
\item
Each vertex $t \in v(H)$ in $H$ has a decoration $\xi(t) \in \elt(\sfP^{(n-2)+})$, which is itself a finite sequence of $\sfP^{(n-3)+}$-decorated graphs.  For each such decoration $\xi(t)$, at level $n-1$ of the metagraph $\bM(H)$, draw the underlying sequence of graphs of $\xi(t)$.  Separate the entries of the underlying sequence of graphs of $\xi(t)$ by commas if $\xi(t)$ is a sequence of length $> 1$.  These underlying sequences of graphs are arranged in the $(n-1)$st level of the metagraph from left to right, according to the labels of the vertices $t \in v(H)$.
\item
For each vertex $t \in v(H)$, draw an edge in the metagraph $\bM(H)$ from the top level to the underlying sequence of graphs of $\xi(t)$ in the $(n-1)$st level.
\item
The steps above are repeated, starting at the $(n-1)$st level.  In other words, each vertex decoration $\xi(u) \in \elt(\sfP^{(n-3)+})$ in each $\xi(t) \in \elt(\sfP^{(n-2)+})$ is a finite sequence of $\sfP^{(n-4)+}$-decorated graphs.  The underlying sequences of graphs of these $\xi(u)$ are drawn at the $(n-2)$nd level of the metagraph $\bM(H)$ as described above.  Also draw edges from the $(n-1)$st level to the $(n-2)$st level as described above.
\item
This process is done $n-1$ times, in which the last step is modified slightly.  In level $2$ of the metagraph $\bM(H)$, we have the underlying sequences of graphs of some elements in $\sfP^+$.  In level $1$ of the metagraph, we write down the elements in $\sfP$ that decorate these elements in $\sfP^+$ together with the corresponding edges.
\end{enumerate}

Let $k_i$ be the number of connected components in $\zeta_i$, and let $H^i_j$ be the $j$th connected component in $\zeta_i$.  Repeating the above steps for all $k_i$ connected components in $\zeta_i$, we obtain $k_i$ $n$-level metagraphs $\bM(H^i_j)$ $(1 \leq j \leq k_i)$.  The same process is performed on the other entries $\zeta_l$ in $\zeta$.  In the end, we obtain the sequence
\[
\bM(\zeta) = \left((\bM(H^1_1) \cdots \bM(H^1_{k_1})), \ldots , (\bM(H^r_1) \cdots \bM(H^r_{k_r}))\right)
\]
of $n$-level metagraphs.  The bottom level of this sequence of metagraphs contains elements in $\sfP$, some algebraic data.  The $n-1$ levels above it contains finite sequences of graphs, some purely combinatorial data.

The sequence $\bM(\zeta)$ of $n$-level metagraphs uniquely determines the $n$-dimensional $\sfP$-propertope $\zeta$.  To recover $\zeta$, as in the example $G'$ above, one starts at the bottom level.  Using their vertex labels, decorate the graphs in level $2$ in $\bM(\zeta)$ using the elements in $\sfP$ in level $1$.  We then obtain some finite sequences of $\sfP$-decorated graphs, i.e., elements in $\sfP^+$.  Now repeat this going-up process starting at level $2$ of $\bM(\zeta)$, now containing elements in $\sfP^+$.  This process is repeated $n-1$ times, after which we recover the $n$-dimensional $\sfP$-propertope $\zeta$.

\subsection{$\sfP$-propertopic sets}
\label{subsec:propertopicset}

Just as one defines simplicial objects as presheaves on the category $\Delta$, we now define $\sfP$-propertopic sets as presheaves on the category of $\sfP$-propertopes.

\begin{definition}
Given a unital $\frakC$-colored PROP $\sfP$ over $\set$, the functor category $\propset$ is called the category of \textbf{$\sfP$-propertopic sets}.
\end{definition}

\begin{definition}
Given a unital $\frakC$-colored PROP $\sfP$ over $\kmodule$ ($=$ the category of modules over a field $\bk$ of characteristic $0$), the functor category $\propmodule$ is called the category of \textbf{$\sfP$-propertopic $\bk$-modules}.
\end{definition}

Since the morphisms of $\sfP$-propertopes are generated by the face maps, it makes sense that a $\sfP$-propertopic set can be described by what it does to the $\sfP$-propertopes and the face maps.

\begin{proposition}
\label{propertopicset}
A $\sfP$-propertopic set $X \in \propset$ consists of exactly the following data:
\begin{enumerate}
\item
It assigns to each $n$-dimensional $\sfP$-propertope $\gamma \in \elt(\sfP^{(n-1)+})$ $(n \geq 0)$ a set $X(\gamma)$.
\item
It assigns to each \textbf{face map}
\[
f \colon \gamma \to \alpha \in \bP(\sfP)
\]
a function
\[
X(f) \colon X(\gamma) \to X(\alpha).
\]
\end{enumerate}
Moreover, the images under $X$ of the consistency diagrams \eqref{hconsistency1} -- \eqref{econsistency} are commutative.

Likewise, a map
\[
F \colon X \to Y
\]
of $\sfP$-propertopic sets consists of exactly the following data:  It assigns to each $n$-dimensional $\sfP$-propertope $\gamma \in \elt(\sfP^{(n-1)+})$ $(n \geq 0)$ a function
\[
F(\gamma) \colon X(\gamma) \to Y(\gamma)
\]
such that, for each \textbf{face map} $f \colon \gamma \to \alpha$, the square
\[
\SelectTips{cm}{10}
\xymatrix{
X(\gamma) \ar[r]^-{F(\gamma)} \ar[d]_-{X(f)} & Y(\gamma) \ar[d]^-{Y(f)}\\
X(\alpha) \ar[r]^-{F(\alpha)} & Y(\alpha)
}
\]
is commutative.
\end{proposition}

\begin{proof}
A $\sfP$-propertopic set $X$ has at least the stated data.  Because of the way composition is defined in $\bP(\sfP)$, the image under $X$ of a general morphism of $\sfP$-propertopes of the form \eqref{Pmorphism} must be the composition
\[
X(h_{k+1}) \circ \cdots \circ X(h_{n-1}) \circ X(h_n) \colon X(\gamma) \to X(\delta)
\]
of functions.  Since each $h_l$ is by definition a face map, it follows that the above function is determined by what $X$ does to the face maps.  The assertion about a map of $\sfP$-propertopic sets is proved similarly.
\end{proof}

\begin{remark}
The obvious $\bk$-module analogue of Proposition ~\ref{propertopicset} is also true.  Simply replace \emph{sets} by $\bk$-modules and \emph{functions} by $\bk$-linear maps.
\end{remark}

For a $\sfP$-propertopic set $X$ and an in-face/out-face map $f \in \bP(\sfP)$, we call $X(f)$ an \textbf{in-face/out-face map} in $X$.  A \textbf{face map} in $X$ is either an in-face map or an out-face map in $X$.

\begin{example}[\textbf{Standard $\sfP$-propertopic sets}]
Let $\gamma \in \elt(\sfP^{(n-1)+})$ be an $n$-dimensional $\sfP$-propertope for some $n \geq 0$.  Then there is a $\sfP$-propertopic set
\begin{equation}
\label{Deltagamma}
\Delta_\gamma = \bP(\sfP)(\gamma, -),
\end{equation}
given by the functor corepresented by $\gamma$.  In other words, if $\alpha$ is a $\sfP$-propertope, then
\[
\Delta_\gamma(\alpha) = \bP(\sfP)(\gamma,\alpha)
\]
is the set of morphisms
\[
\gamma \to \alpha \in \bP(\sfP).
\]
Given a map
\[
h \colon \alpha \to \alpha' \in \bP(\sfP),
\]
the function
\[
\Delta_\gamma(h) \colon \Delta_\gamma(\alpha) = \bP(\sfP)(\gamma,\alpha) \to \bP(\sfP)(\gamma,\alpha') = \Delta_\gamma(\alpha')
\]
is induced by composition of morphisms in $\bP(\sfP)$.  We call $\Delta_\gamma$ the \textbf{standard $\sfP$-propertopic set of shape $\gamma$}.  It is an analogue of the standard $n$-simplex $\Delta^n$ in the category of simplicial sets, or the $n$-dimensional disk $D^n$ (or the topological standard $n$-simplex) in the category of topological spaces.
\qed
\end{example}

\subsection{Cells, horns, and boundaries}
\label{subsec:horns}

A very important concept about simplicial sets is that of a \emph{Kan fibration}, which  is a map of simplicial sets with a certain lifting property with respect to some \emph{horn inclusions} (see, e.g., \cite[Chapter I]{gj}).  A \emph{Kan complex} is a simplicial set $A$ for which the map $A \to *$ is a Kan fibration.  In the standard model category structure of simplicial sets \cite{gj,hovey,quillen}, the Kan complexes are exactly the fibrant objects, which we think of as the \emph{good} objects.  We want an analogue of a Kan fibration/complex for $\sfP$-propertopic sets, so we first need to define the relevant concepts of cells, horns, and boundaries.

In a simplicial set $A \in \set^{\Delta^{op}}$, an element $a \in A(\textbf{n})$ is called an $n$-cell or an $n$-simplex, where $\textbf{n} = \{0 < 1 < \cdots < n\} \in \Delta$.  In our world of $\sfP$-propertopes, we have to replace the object $\textbf{n} \in \Delta$ with the \emph{set} $\elt(\sfP^{(n-1)+})$ of $n$-dimensional $\sfP$-propertopes.

\begin{definition}
\label{n-cells}
Let $X$ be a $\sfP$-propertopic set.  The set of \textbf{$n$-cells} in $X$ is defined as the disjoint union
\[
X_n = \coprod_{\gamma \in \elt(\sfP^{(n-1)+})} X(\gamma),
\]
indexed by the set $\elt(\sfP^{(n-1)+})$ of $n$-dimensional $\sfP$-propertopes.  An element $x \in X_n$ is called an \textbf{$n$-cell} in $X$.  If
\[
x \in X(\gamma) \subseteq X_n,
\]
then we call $\gamma$ the \textbf{shape} of the $n$-cell $x$ and call $x$ a \textbf{$\gamma$-cell}.
\end{definition}

If $X \in \propmodule$ is a $\sfP$-propertopic $\bk$-module, then its set of \textbf{$n$-cells} is defined as the direct sum
\[
X_n = \bigoplus_{\gamma \in \elt(\sfP^{(n-1)+})} X(\gamma).
\]

In a simplicial set $A$, each $n$-cell is represented by a map
\[
\Delta^n \to A
\]
of simplicial sets from the standard $n$-simplex $\Delta^n$.  There is a similar description for $\gamma$-cells in a $\sfP$-propertopic set.  The roles of $A(\textbf{n})$ and $\Delta^n$ are now played by the set $X(\gamma)$ of $\gamma$-cells and the standard $\sfP$-propertopic set $\Delta_\gamma$, respectively.

\begin{proposition}
\label{gammacellsmaps}
Let $X$ be a $\sfP$-propertopic set, and let $\gamma \in \elt(\sfP^{(n-1)+})$ be an $n$-dimensional $\sfP$-propertope.  Then there is a canonical bijection
\[
X(\gamma) \cong \propset(\Delta_\gamma,X),
\]
where $\Delta_\gamma$ is the standard $\sfP$-propertopic set defined in \eqref{Deltagamma}.
\end{proposition}

\begin{proof}
This is simply the Yoneda Lemma.  In one direction, the bijection sends an element $\eta \in \propset(\Delta_\gamma,X)$ to
\[
x_\eta = \eta(\gamma)(1_\gamma) \in X(\gamma),
\]
which is a $\gamma$-cell in $X$.
\end{proof}

Suppose that
\[
\gamma \in \sfP^{(n-1)+}\binom{\beta_1, \ldots , \beta_s}{\alpha_1, \ldots , \alpha_r}
\]
is an $n$-dimensional $\sfP$-propertope as in \eqref{gammaPn-1}.  There are face maps
\[
\begin{split}
X(f_i)& \colon X(\gamma) \to X(\alpha_i) \quad (1 \leq i \leq r),\\
X(g_j)& \colon X(\gamma) \to X(\beta_j) \quad (1 \leq j \leq s)
\end{split}
\]
in $X$, one for each face map out of $\gamma$ (\eqref{outfacemap} and \eqref{infacemap}).  If $x \in X(\gamma)$ is a $\gamma$-cell in $X$, then the elements
\begin{equation}
\label{Xinoutface}
\begin{split}
y_i &= X(f_i)(x) \in X(\alpha_i) \subseteq X_{n-1} \quad (1 \leq i \leq r),\\
z_j &= X(g_j)(x) \in X(\beta_j) \subseteq X_{n-1} \quad (1 \leq j \leq s)
\end{split}
\end{equation}
are $(n-1)$-cells in $X$.

\begin{definition}
For a $\sfP$-propertopic set $X$ and a $\gamma$-cell $x \in X(\gamma) \subseteq X_n$, we call $y_i$ and $z_j$ in \eqref{Xinoutface} the \textbf{$i$th in-face of $x$} and the \textbf{$j$th out-face of $x$}, respectively.
\end{definition}

We depict an $n$-cell $x$ with all of its in-faces and out-faces as
\begin{equation}
\label{xfaces}
(y_1, \ldots , y_r) \xrightarrow{x} (z_1, \ldots , z_s).
\end{equation}
From a categorical view point, the $n$-cells are exactly the $n$-morphisms.  An $n$-morphism is a way of composing $(n-1)$-morphisms.  So we think of the $n$-cell $x \in X(\gamma)$ as a \emph{way of composition} with \emph{sources} the $(n-1)$-cells $(y_1, \ldots , y_r)$, i.e., the in-faces of $x$.  The $(n-1)$-cells $(z_1, \ldots , z_s)$ -- the out-faces of $x$ -- are \emph{composites} of $(y_1, \ldots , y_r)$.  We do not say \emph{the} composites of $(y_1, \ldots , y_r)$ because there may be another $n$-cell $x' \in X(\gamma')$ that also has in-faces the $(n-1)$-cells $(y_1, \ldots , y_r)$.  Given such an $x'$, we have another way of composition
\[
(y_1, \ldots , y_r) \xrightarrow{x'} (z_1', \ldots , z_s'),
\]
giving rise to possibly different composites $(z_1', \ldots , z_s')$.  This discussion leads naturally to analogues of horns and boundaries in the $\sfP$-propertopic world.

\begin{definition}
\label{gammahorndef}
Let
\[
\gamma \in \sfP^{(n-1)+}\binom{\beta_1, \ldots , \beta_s}{\alpha_1, \ldots , \alpha_r}
\]
be an $n$-dimensional $\sfP$-propertope for some $n \geq 1$, and let $X$ be a $\sfP$-propertopic set.
\begin{enumerate}
\item
A \textbf{$\gamma$-horn} in $X$ consists of $(n-1)$-cells $y_i \in X(\alpha_i)$ for $1 \leq i \leq r$.  We also write such a $\gamma$-horn as
\begin{equation}
\label{gammahorn}
(y_1, \ldots , y_r) \xrightarrow{?} ?.
\end{equation}
An \textbf{$n$-dimensional horn} in $X$ is a $\gamma$-horn in $X$ for some $n$-dimensional $\sfP$-propertope $\gamma$.
\item
A \textbf{filling} of the $\gamma$-horn \eqref{gammahorn} is a $\gamma$-cell $x \in X(\gamma)$ whose $i$th in-face is $y_i$ for $1 \leq i \leq r$.
\item
A \textbf{$\gamma$-boundary} in $X$ consists of $(n-1)$-cells $y_i \in X(\alpha_i)$ for $1 \leq i \leq r$ and $(n-1)$-cells $z_j \in X(\beta_j)$ for $1 \leq j \leq s$. We also write such a $\gamma$-boundary as
\begin{equation}
\label{gammaboundary}
(y_1, \ldots , y_r) \xrightarrow{?} (z_1, \ldots , z_s).
\end{equation}
An \textbf{$n$-dimensional boundary} in $X$ is a $\gamma$-boundary in $X$ for some $n$-dimensional $\sfP$-propertope $\gamma$.
\item
A \textbf{filling} of the $\gamma$-boundary \eqref{gammaboundary} is a $\gamma$-cell $x \in X(\gamma)$ whose $i$th in-face is $y_i$ for $1 \leq i \leq r$ and whose $j$th out-face is $z_j$ for $1 \leq j \leq s$.
\end{enumerate}
\end{definition}
The same definitions can be made if $\sfP$ is a colored PROP over $\kmodule$ and $X$ is a $\sfP$-propertopic $\bk$-module.

So a filling of a $\gamma$-horn (or $\gamma$-boundary) is an extension of \eqref{gammahorn} (or \eqref{gammaboundary}) to \eqref{xfaces}, and vice versa.  Our horns, boundaries, and fillings correspond to \emph{niches}, \emph{frames}, and \emph{occupants} in the opetopic setting of Baez and Dolan \cite{bd1}.  We prefer the terms \emph{horn} and \emph{boundary} because they sound more familiar.  Our $n$-dimensional horns are analogues of the $k$th horns $\Lambda^n_k \hookrightarrow \Delta^n$ inside the standard $n$-simplex.  Likewise, our $n$-dimensional boundaries are analogues of the boundary $\partial\Delta^n \hookrightarrow \Delta^n$ or the topological sphere $S^{n-1}$.

As in the case of $\gamma$-cells, the sets of $\gamma$-horns and $\gamma$-boundaries in a $\sfP$-propertopic set $X$ correspond to maps from certain universal $\sfP$-propertopic sets.  Indeed, let
\[
\gamma \in \sfP^{(n-1)+}\binom{\beta_1, \ldots , \beta_s}{\alpha_1, \ldots , \alpha_r}
\]
be an $n$-dimensional $\sfP$-propertope for some $n \geq 1$ as in Definition ~\ref{gammahorndef}.  To describe $\gamma$-boundaries as maps, we need a $\sfP$-propertopic set that is generated by the input and the output colors of $\gamma$, i.e., the $\alpha_i$ for $1 \leq i \leq r$ and the $\beta_j$ for $1 \leq j \leq s$.  Likewise, to describe $\gamma$-horns, we need a $\sfP$-propertopic set that is generated by the input colors of $\gamma$.  The precise definitions are given below.

\begin{definition}
\label{standardboundary}
Define the $\sfP$-propertopic set $\partial\Delta_\gamma$ by setting
\[
\partial\Delta_\gamma(\alpha) =
\begin{cases}
\bP(\sfP)(\gamma,\alpha) & \text{if $\alpha \in \elt(\sfP^{(k-1)+})$ with $k < n$},\\
\varnothing & \text{otherwise}.
\end{cases}
\]
At any map in $\bP(\sfP)$, $\partial\Delta_\gamma$ is induced by composition of maps in $\bP(\sfP)$.  We call $\partial\Delta_\gamma$ the \textbf{standard $\gamma$-boundary}.
\end{definition}

\begin{definition}
\label{standardhorn}
Define the $\sfP$-propertopic set $\Lambda_\gamma$ by setting
\[
\Lambda_\gamma(\alpha) =
\begin{cases}
\bP(\sfP)(\gamma,\alpha) &\text{if $\bP(\sfP)(\alpha_i,\alpha) \not= \varnothing$ for some $i \in \{1, \ldots , r\}$},\\
\varnothing &\text{otherwise}.
\end{cases}
\]
At any map in $\bP(\sfP)$, $\Lambda_\gamma$ is induced by composition of maps in $\bP(\sfP)$.  We call $\Lambda_\gamma$ the \textbf{standard $\gamma$-horn}.
\end{definition}

Now we have the analogue of Proposition ~\ref{gammacellsmaps} for $\gamma$-boundaries and $\gamma$-horns.

\begin{theorem}
\label{standardboundarymap}
Let $X$ be a $\sfP$-propertopic set, and let $\gamma \in \sfP^{(n-1)+}\binom{\ubeta}{\ualpha}$ be an $n$-dimensional $\sfP$-propertope for some $n \geq 1$ as in Definition ~\ref{gammahorndef}.  Then there are canonical bijections
\[
\left(\text{$\gamma$-boundaries in $X$}\right) \cong \propset(\partial\Delta_\gamma,X)
\]
and
\[
\left(\text{$\gamma$-horns in $X$}\right) \cong \propset(\Lambda_\gamma,X).
\]
\end{theorem}

\begin{proof}
We prove the first bijection.  The second bijection can be proved similarly.  Suppose given
\[
\eta \in \propset(\partial\Delta_\gamma,X).
\]
We want to associate to $\eta$ a $\gamma$-boundary in $X$.  Suppose that
\[
f_i \colon \gamma \to \alpha_i \in \bP(\sfP)
\]
is the in-face map that records the $i$th input color $\alpha_i$ of $\gamma$ \eqref{infacemap} for some $i \in \{1, \ldots , r\}$.  Then $f_i \in (\partial\Delta_\gamma)(\alpha_i)$, and we obtain the $\alpha_i$-cell
\[
y_i = \eta(\alpha_i)(f_i) \in X(\alpha_i).
\]
Likewise, the map
\[
g_j \colon \gamma \to \beta_j \in \bP(\sfP)
\]
that records the $j$th output color $\beta_j$ of $\gamma$ \eqref{outfacemap} gives rise to the $\beta_j$-cell
\[
z_j = \eta(\beta_j)(g_j) \in X(\beta_j)
\]
for each $j \in \{1, \ldots , s\}$.  Putting these $(n-1)$-cells in $X$ together, we obtain
\[
\partial_\eta = \left((y_1, \ldots , y_r) \xrightarrow{?} (z_1, \ldots , z_s)\right),
\]
which is a $\gamma$-boundary in $X$.

Conversely, suppose given a $\gamma$-boundary $\partial$ in $X$ as in \eqref{gammaboundary}.  Define an element
\[
\varepsilon_\partial \in \propset(\partial\Delta_\gamma,X)
\]
as follows.  If $\delta \in \elt(\sfP^{(k-1)+})$ is a $k$-dimensional $\sfP$-propertope with $k \geq n$, then
\[
\varepsilon_\partial(\delta) \colon (\partial\Delta_\gamma)(\delta) = \varnothing \to X(\delta)
\]
is the trivial map.

Now suppose that $k < n$ and that $(\partial\Delta_\gamma)(\delta) \not= \varnothing$.  In this case, a typical element in
\[
(\partial\Delta_\gamma)(\delta) = \bP(\sfP)(\gamma,\delta)
\]
is a finite chain of face maps of the form \eqref{Pmorphism}.  The first face map
\[
h_n \colon \gamma \to \delta_{n-1}
\]
in such a typical element must be one of the in-face maps
\[
f_i \colon \gamma \to \alpha_i \quad (1 \leq i \leq r)
\]
or one of the out-face maps
\[
g_j \colon \gamma \to \beta_j \quad (1 \leq j \leq s).
\]
Define the map
\[
\varepsilon_\partial(\delta) \colon (\partial\Delta_\gamma)(\delta) \to X(\delta)
\]
by setting
\begin{multline*}
\varepsilon_\partial(\delta)\left(\gamma \xrightarrow{h_n} \delta_{n-1} \xrightarrow{h_{n-1}} \cdots \xrightarrow{h_{k+1}} \delta\right) =\\
\begin{cases}
y_i & \text{if $h_n = f_i \colon \gamma \to \alpha_i$ and $k = n-1$},\\
\left(X(h_{k+1}) \circ \cdots \circ X(h_{n-1})\right)(y_i) & \text{if $h_n = f_i \colon \gamma \to \alpha_i$ and $k < n-1$},\\
z_j & \text{if $h_n = g_j \colon \gamma \to \beta_j$ and $k = n-1$},\\
\left(X(h_{k+1}) \circ \cdots \circ X(h_{n-1})\right)(z_j) & \text{if $h_n = g_j \colon \gamma \to \beta_j$ and $k < n-1$}.
\end{cases}
\end{multline*}
One can check by direct inspection that the square
\[
\SelectTips{cm}{10}
\xymatrix{
(\partial\Delta_\gamma)(\delta) \ar[rr]^-{\varepsilon_\partial(\delta)} \ar[d]_-{(\partial\Delta_\gamma)(h)} & & X(\delta) \ar[d]^-{X(h)}\\
(\partial\Delta_\gamma)(\delta') \ar[rr]^-{\varepsilon_\partial(\delta')} & & X(\delta')
}
\]
is commutative for any face map $h \colon \delta \to \delta'$.  So $\varepsilon_\partial$ is indeed an element in $\propset(\partial\Delta_\gamma,X)$.  It is also not hard to check that the constructions
\[
\eta \mapsto \partial_\eta \quad\text{and}\quad\partial \mapsto \varepsilon_\partial
\]
are inverses of each other, so we have constructed the desired bijection between the set of $\gamma$-boundaries in $X$ and $\propset(\partial\Delta_\gamma,X)$.
\end{proof}

\subsection{$\sfP$-propertopic fibrations}
\label{subsec:propertopicfibration}

To define a fibration for $\sfP$-propertopic sets, let us first recall a Kan fibration. A map $p \colon A \to A'$ of simplicial sets is a \textbf{Kan fibration} if every solid-arrow commutative diagram
\begin{equation}
\label{Kanfibration}
\SelectTips{cm}{10}
\xymatrix{
\Lambda^n_k \ar[r] \ar@{^{(}->}[d] & A \ar[d]^-{p} \\
\Delta^n \ar[r] \ar@{.>}[ur]^-{\theta} & A'
}
\end{equation}
has a dotted-arrow lift $\theta \colon \Delta^n \to A$ that makes both resulting triangles commute.  The solid-arrow commutative square is equivalent to a $k$th horn
\[
(a_0, \ldots , \hat{a}_k, \ldots , a_n)
\]
in $A$ and an $n$-cell $a'$ in $A'$ such that
\[
d_ia' = p(a_i)
\]
for $i \not= k$.  The existence of the dotted arrow $\theta$ is equivalent to an $n$-cell $a$ in $A$ such that
\begin{enumerate}
\item
$a$ is a lift of $a'$ in the sense that $p(a) = a'$, and
\item
the $i$th face $d_ia$ is $a_i$ for $i \not= k$.
\end{enumerate}
In short, the map $p \colon A \to A'$ is a Kan fibration if every $k$th horn in $A$ with a compatible $n$-cell in $A'$ can be extended to a compatible $n$-cell in $A$.

Now we define an analogue of a Kan fibration for $\sfP$-propertopic sets.  A fibration for $\sfP$-propertopic sets should be a map such that every $\gamma$-horn in the source with a compatible $\gamma$-cell in the target can be extended to a compatible $\gamma$-cell in the source.  More precisely, we make the following definition.

\begin{definition}
\label{propertopicfibration}
A map
\[
p \colon X \to X'
\]
of $\sfP$-propertopic sets is called a \textbf{$\sfP$-propertopic fibration}, or simply a \textbf{fibration}, if it has the following \textbf{horn-filling property}: Suppose given
\begin{itemize}
\item
a $\gamma$-horn
\[
(y_1, \ldots , y_r) \xrightarrow{?} ?
\]
in $X$ as in Definition ~\ref{gammahorndef}, and
\item
a $\gamma$-cell
\[
x' \in X'(\gamma) \subseteq X'_n
\]
such that
\[
y_i' = p(\alpha_i)(y_i)
\]
for $1 \leq i \leq r$, where
\[
y_i' = X'(f_i)(x') \in X'(\alpha_i) \subseteq X'_{n-1}
\]
is the $i$th in-face of $x'$ \eqref{Xinoutface}.
\end{itemize}
Then there exists a $\gamma$-cell
\[
x \in X(\gamma) \subseteq X_n
\]
such that
\begin{enumerate}
\item
$x$ is a lift of $x'$ in the sense that
\[
p(\gamma)(x) = x',
\]
and
\item
$x$ is a filling of the $\gamma$-horn $(y_1, \ldots , y_r) \xrightarrow{?} ?$ in $X$ (Definition ~\ref{gammahorndef}).
\end{enumerate}
\end{definition}

As in the case of a Kan fibration, there is a diagrammatic way to describe $\sfP$-propertopic fibrations.  In fact, for any $n$-dimensional $\sfP$-propertope $\gamma$ with $n \geq 1$, there are entrywise inclusions
\begin{equation}
\label{standardinclusion}
\Lambda_\gamma \xrightarrow{i} \partial\Delta_\gamma \xrightarrow{i} \Delta_\gamma.
\end{equation}
At each $\sfP$-propertope $\delta$, each $i(\delta)$ is either the trivial map
\[
\varnothing \to \bP(\sfP)(\gamma,\delta)
\]
or the identity map on $\bP(\sfP)(\gamma,\delta)$.  We call the map $\partial\Delta_\gamma \hookrightarrow \Delta_\gamma$ the \textbf{$\gamma$-boundary inclusion} and the maps $\Lambda_\gamma \hookrightarrow \partial\Delta_\gamma$ and $\Lambda_\gamma \hookrightarrow \Delta_\gamma$ the \textbf{$\gamma$-horn inclusions}.

Below is a $\sfP$-propertopic analogue of the diagrammatic description \eqref{Kanfibration} of a Kan fibration.

\begin{proposition}
\label{fibrationdiagram}
Let
\[
p \colon X \to X'
\]
be a map of $\sfP$-propertopic sets. Then $p$ is a $\sfP$-propertopic fibration if and only if for every $n$-dimensional $\sfP$-propertope $\gamma$ with $n \geq 1$, every solid-arrow commutative diagram
\[
\SelectTips{cm}{10}
\xymatrix{
\Lambda_\gamma \ar[r] \ar@{^{(}->}[d]_-{i} & X \ar[d]^-{p} \\
\Delta_\gamma \ar[r] \ar@{.>}[ur]^-{\theta} & X'
}
\]
in $\propset$ admits a dotted-arrow lift
\[
\theta \colon \Delta_\gamma \to X
\]
that makes the two resulting triangles commute.
\end{proposition}

\begin{proof}
This is a restatement of Definition ~\ref{propertopicfibration} using Proposition ~\ref{gammacellsmaps} and Theorem ~\ref{standardboundarymap}.
\end{proof}

There is a terminal object in the category of $\sfP$-propertopic sets.  Indeed, the object $* \in \propset$ defined as
\begin{equation}
\label{terminalpropset}
*(\gamma) = \{*\}
\end{equation}
for every $\sfP$-propertope $\gamma$ is the terminal object.  Likewise, in the $\bk$-linear setting, the terminal $\sfP$-propertopic $\bk$-module $*$ has
\[
*(\gamma) = 0
\]
for every $\sfP$-propertope $\gamma$.

\begin{definition}
\label{def:fibrantpropset}
A $\sfP$-propertopic set (or $\bk$-module) $X$ is said to be \textbf{fibrant} if the unique map
\[
X \to *
\]
to the terminal object is a $\sfP$-propertopic fibration.
\end{definition}

A special case of Proposition ~\ref{fibrationdiagram} (with $X' = *$) gives the following descriptions of a fibrant $\sfP$-propertopic set.

\begin{corollary}
\label{fibrantpropset}
Let $X$ be a $\sfP$-propertopic set.  Then the following statements are equivalent.
\begin{enumerate}
\item
The $\sfP$-propertopic set $X$ is fibrant.
\item
For each $n \geq 1$, every $n$-dimensional horn in $X$ admits a filling.
\item
For each $n$-dimensional $\sfP$-propertope $\gamma$ with $n \geq 1$, every solid-arrow diagram in $\propset$
\[
\SelectTips{cm}{10}
\xymatrix{
\Lambda_\gamma \ar[r] \ar@{^{(}->}[d]_-{i} & X \\
\Delta_\gamma \ar@{.>}[ur]_-{\theta} &
}
\]
admits a dotted-arrow lift
\[
\theta \colon \Delta_\gamma \to X
\]
that makes the triangle commute.
\end{enumerate}
\end{corollary}

\section{Higher dimensional $\sfP$-algebras}
\label{sec:weakalgebra}

Fix a unital $\frakC$-colored PROP $\sfP$ over $\set$.

In \S\ref{subsec:weakndef} we define higher dimensional $\sfP$-algebras, called \emph{weak-$n$ $\sfP$-algebras}, for $0 \leq n \leq \infty$.  They are defined as $\sfP$-propertopic sets with certain lifting properties with respect to $\gamma$-horns and $\gamma$-boundaries.  The definition of a weak-$n$ $\sfP$-algebra is somewhat similar to that of a space with trivial homotopy groups in dimensions $\geq n + 1$, i.e., a homotopy $n$-type.

One should think of a weak-$n$ $\sfP$-algebra as an \emph{$n$-time categorified $\sfP$-algebra}.  There are two extreme cases.  When $n = \infty$, we have weak-$\infty$ $\sfP$-algebras, which we prefer to call \emph{weak-$\omega$ $\sfP$-algebras}.  From its very definition, a weak-$\omega$ $\sfP$-algebra is exactly a fibrant $\sfP$-propertopic set, which is analogous to a Kan complex in the category of simplicial sets.  When $n = 0$, we observe in \S\ref{subsec:weak0} that the category of weak-$0$ $\sfP$-algebras is equivalent to that of $\sfP$-algebras (Theorem ~\ref{thm:weak0algebra}).

For $1 \leq n < \infty$, the study of a weak-$n$ $\sfP$-algebra $X$ splits into two steps (Corollary ~\ref{cor1:P'}).  First, in \S\ref{subsec:EM} we look at an object $p(X)$ \eqref{pX}, which essentially consists of the $k$-cells in $X$ for $k \geq n$ and the face maps in those dimensions.  As we will explain below (Definition ~\ref{connectedweakn}), such an object $p(X)$ is somewhat analogous to an Eilenberg-Mac Lane space $K(\pi,n)$.  We call $p(X)$ an \emph{Eilenberg-Mac Lane weak-$n$ $\sfP$-algebra}.  An analogue of Theorem ~\ref{thm:weak0algebra} for higher values of $n$ (Theorem ~\ref{thm:EMweakn}) tells us that there is an equivalence between the categories of Eilenberg-Mac Lane weak-$n$ $\sfP$-algebras and $\sfP^{n+}$-algebras.

Second, we look at the $k$-cells in a weak-$n$ $\sfP$-algebra $X$ for $k \leq n+1$.  Combining the information from the two steps, we give a categorical description of 
weak-$n$ $\sfP$-algebras in \S\ref{subsec:catweakn}.  Briefly, in a weak-$n$ $\sfP$-algebra $X$ $(1 \leq n \leq \infty)$, composition of $k$-cells via $(k+1)$-cells for $0 \leq k \leq n - 1$ is, in general, not a function but a \emph{multi-valued function}, satisfying some consistency conditions.  Composition is an honest operation only for the \emph{top cells}, i.e., the $n$-cells when $n < \infty$.  These top dimensional compositions give the top cells the structure of a $\sfP^{n+}$-algebra.

As we will explain in more details below (p.\pageref{postnikovstragegy}), this two-step strategy for understanding weak-$n$ $\sfP$-algebras is analogous to the Postnikov tower of a homotopy $n$-type $Y$ in homotopy theory.  The Eilenberg-Mac Lane weak-$n$ $\sfP$-algebra $p(X)$ plays the role of the fiber $K(\pi_n(Y),n)$ at the $n$th stage of the Postnikov tower of $Y$.  The restriction of $X$ to $k$-dimensional cells for $k \leq n+1$ is analogous to the $(n-1)$st Postnikov approximation $Y_{n-1}$ of $Y$.


One feature of our theory of weak-$n$ $\sfP$-algebras is that coherence laws are treated as higher (multi-valued) compositions.  In fact, the coherence laws for the $k$-cells for $k < n$ are their compositions, which are governed by the $(k+1)$-cells.  When $n < \infty$, the coherence laws for the top cells ($= n$-cells) are encoded in the $\sfP^{n+}$-algebra structure on these cells.  To reiterate our point:
\begin{quote}
We make no difference between coherence laws and compositions.
\end{quote}
This feature of our theory of weak-$n$ $\sfP$-algebras is similar to Leinster's definition \cite[Chapter 9]{leinster} of a weak $n$-category, in which coherence laws and compositions are also treated as the same concept.

For $n < \infty$, we observe in \S\ref{subsec;underlyingcat} that a weak-$n$ $\sfP$-algebra $X$ has an underlying category $\Xtilde$ (Theorem ~\ref{thm:weakncategory}).  The objects in $\Xtilde$ are the $(n-1)$-cells in $X$, and the morphisms in $\Xtilde$ are the corresponding fillings of $n$-dimensional boundaries.

If $\varphi \colon \sfP \to \sfQ$ is a morphism of $\frakC$-colored PROPs, we observe in \S\ref{subsec:pullback} that there is a well-behaved induced functor
\[
\Phi \colon \bP(\sfP) \to \bP(\sfQ)
\]
from $\sfP$-propertopes to $\sfQ$-propertopes (Theorem ~\ref{varphipropertope}).  Moreover, the pullback functor
\[
\varphi^* \colon \propsetQ \to \propset
\]
restricts to a pullback functor
\[
\varphi^* \colon \algn(\sfQ) \to \algn(\sfP)
\]
from weak-$n$ $\sfQ$-algebras to weak-$n$ $\sfP$-algebras (Corollary ~\ref{cor2:pullback}).

\subsection{Definitions of weak-$n$ $\sfP$-algebras}
\label{subsec:weakndef}

Our weak-$n$ $\sfP$-algebras are $\sfP$-propertopic analogues of homotopy $n$-types in spaces or simplicial sets.  Recall that a $\sfP$-propertopic analogue of the $k$-dimensional disk $D^k$ is the standard $\sfP$-propertopic set $\Delta_\gamma$ of shape $\gamma$, where $\gamma$ is a $k$-dimensional $\sfP$-propertope \eqref{Deltagamma}.  An analogue of the sphere $S^{k-1}$ is the standard $\gamma$-boundary $\partial\Delta_\gamma$ (Definition ~\ref{standardboundary}).  An analogue of the boundary inclusion
\[
S^{k-1} \hookrightarrow D^k
\]
is the $\gamma$-boundary inclusion
\[
i \colon \partial\Delta_\gamma \to \Delta_\gamma
\]
defined in \eqref{standardinclusion}.  Using this analogy, we now make the following definitions.

\begin{definition}
\label{def:weakPalgebra}
Let $n \geq 0$ be an integer or $\infty$.  A \textbf{weak-$n$ $\sfP$-algebra} is defined as a $\sfP$-propertopic set $X \in \propset$ that satisfies the following three conditions:
\begin{enumerate}
\item
For $1 \leq k \leq n$, every $k$-dimensional horn in $X$ admits a filling  (Definition ~\ref{gammahorndef}).
\item
Every $(n+1)$-dimensional horn in $X$ admits a \textbf{unique} filling.
\item
For $N \geq n + 2$, every $N$-dimensional boundary in $X$ admits a \textbf{unique} filling.
\end{enumerate}
A \textbf{morphism} of weak-$n$ $\sfP$-algebras is a map of the underlying $\sfP$-propertopic sets.  The category of weak-$n$ $\sfP$-algebras is denoted by $\algn(\sfP)$.
\end{definition}

We also call a weak-$\infty$ $\sfP$-algebra a \textbf{weak-$\omega$ $\sfP$-algebra}, since in higher category theory the term \emph{$\omega$-category} is often used for $\infty$-category.

\begin{remark}
One might wonder why a morphism of weak-$n$ $\sfP$-algebras is not required to preserve the (unique) fillings in weak-$n$ $\sfP$-algebras.  In fact, such a morphism does preserve the (unique) fillings, which is a consequence of the definition of a map of $\sfP$-propertopic sets.  Recall that a map $F$ of $\sfP$-propertopic sets is compatible with the face maps (Proposition ~\ref{propertopicset}).  It follows that $F$ must send a filling of a horn/boundary to a filling of the corresponding horn/boundary in the image of $F$.  Therefore, in Definition ~\ref{def:weakPalgebra}, a morphism of weak-$n$ $\sfP$-algebras necessarily preserves the (unique) fillings that weak-$n$ $\sfP$-algebras are supposed to have.  This is why we do not have to put this condition in the definition of a morphism of weak-$n$ $\sfP$-algebras.
\end{remark}

\begin{remark}
If $\sfP$ is a unital $\frakC$-colored PROP over $\bk$-modules, then we can also define weak-$n$ $\sfP$-algebras as above.  Indeed, in this case, a \textbf{weak-$n$ $\sfP$-algebra} is defined as a $\sfP$-propertopic $\bk$-module $X \in \propmodule$ that satisfies the three conditions in Definition ~\ref{def:weakPalgebra}.
\end{remark}

Just as a $\sfP$-propertopic fibration can be described in terms of diagrams (Proposition ~\ref{fibrationdiagram}), weak-$n$ $\sfP$-algebras can also be described diagrammatically.

\begin{lemma}
A $\sfP$-propertopic set $X \in \propset$ is a weak-$n$ $\sfP$-algebra if and only if the following two conditions hold:
\begin{enumerate}
\item
For each $k$-dimensional $\sfP$-propertope $\gamma$ with $1 \leq k \leq n+1$, every solid-arrow diagram in $\propset$
\[
\SelectTips{cm}{10}
\xymatrix{
\Lambda_\gamma \ar[r] \ar@{^{(}->}[d]_-{i} & X \\
\Delta_\gamma \ar@{.>}[ur]_-{\theta} &
}
\]
admits a dotted-arrow lift
\[
\theta \colon \Delta_\gamma \to X
\]
that makes the triangle commute.  Moreover, the lift $\theta$ is unique if $\gamma$ has dimension $n+1$.
\item
For each $N$-dimensional $\sfP$-propertope $\gamma$ with $N \geq n+2$, every solid-arrow diagram in $\propset$
\[
\SelectTips{cm}{10}
\xymatrix{
\partial\Delta_\gamma \ar[r] \ar@{^{(}->}[d]_-{i} & X \\
\Delta_\gamma \ar@{.>}[ur]_-{\theta} &
}
\]
admits a \textbf{unique} dotted-arrow lift
\[
\theta \colon \Delta_\gamma \to X
\]
that makes the triangle commute.
\end{enumerate}
\end{lemma}

\begin{proof}
This is a restatement of Definition ~\ref{def:weakPalgebra} using Proposition ~\ref{gammacellsmaps}, Theorem ~\ref{standardboundarymap}, the $\gamma$-horn inclusion, and the $\gamma$-boundary inclusion \eqref{standardinclusion}.
\end{proof}

We make the following simple observations about the two extreme cases.

\begin{proposition}
\label{weak0algebra}
Let $X \in \propset$ be a $\sfP$-propertopic set.  Then:
\begin{enumerate}
\item
The object $X$ is a weak-$0$ $\sfP$-algebra if and only if its $1$-dimensional horns and $N$-dimensional boundaries (for $N \geq 2$) have unique fillings.
\item
The object $X$ is a weak-$\omega$ $\sfP$-algebra if and only if it is fibrant.
\end{enumerate}
\end{proposition}

\begin{proof}
The first statement follows immediately from the definition.  For the second statement, note that when $n = \infty$ in Definition ~\ref{def:weakPalgebra}, only condition (1) applies, which is equivalent to $X$ being fibrant by Corollary ~\ref{fibrantpropset}.
\end{proof}

\subsection{Weak-$0$ $\sfP$-algebras as $\sfP$-algebras}
\label{subsec:weak0}

More can be said about weak-$0$ $\sfP$-algebras.  From Definition ~\ref{def:weakPalgebra} or Proposition ~\ref{weak0algebra}, it is not entirely obvious that weak-$n$ $\sfP$-algebras have anything to do with $\sfP$-algebras.  To justify this terminology and the claim that weak-$n$ $\sfP$-algebras should be thought of as $n$-time categorified $\sfP$-algebras, we first show that weak-$0$ $\sfP$-algebras are equivalent to $\sfP$-algebras.

\begin{theorem}
\label{thm:weak0algebra}
There exist functors
\begin{equation}
\label{weak0equivalence}
\phi \colon \alg^0(\sfP) \rightleftarrows \alg(\sfP) \colon \psi
\end{equation}
that give an equivalence between the categories $\alg^0(\sfP)$ of weak-$0$ $\sfP$-algebras and $\alg(\sfP)$ of $\sfP$-algebras.
\end{theorem}

\begin{proof}
First we construct the functor $\phi$.  Let $X \in \propset$ be a weak-$0$ $\sfP$-algebra.  There is a set $X(\gamma)$ for each $n$-dimensional $\sfP$-propertope $\gamma$ for $n \geq 0$.  In particular, when $n = 0$, we have a set
\[
X_c = X(c)
\]
for each $c \in \frakC$ ($=$ the set of $0$-dimensional $\sfP$-propertopes).  To define the $\sfP$-algebra structure map on these sets, pick elements $y_i \in X_{c_i}$ for $1 \leq i \leq n$ and a $1$-dimensional $\sfP$-propertope
\begin{equation}
\label{gamma1propertope}
\gamma \in \sfP\binom{d_1, \ldots , d_m}{c_1, \ldots , c_n} = \sfP\binom{\ud}{\uc}.
\end{equation}
Then we have a $1$-dimensional $\gamma$-horn
\[
(y_1, \ldots , y_n)\xrightarrow{?} ?
\]
in $X$.  Since $X$ is a weak-$0$ $\sfP$-algebra, there exists a unique filling
\[
x \in X(\gamma)
\]
of this $1$-dimensional $\gamma$-horn (Proposition ~\ref{weak0algebra}). In particular, the $m$ out-faces \eqref{Xinoutface} of $x$ give an element
\[
\left(X(g_1)(x), \ldots , X(g_m)(x)\right) = (z_1, \ldots , z_m) \in X_{d_1} \times \cdots \times X_{d_m}.
\]
Since the $\gamma$-cell $x$ is uniquely determined by $\gamma$ and the $y_i$ $(1 \leq i \leq n)$, we thus have an operation
\begin{equation}
\label{weak0lambda}
\lambda \colon \sfP\binom{d_1, \ldots , d_m}{c_1, \ldots , c_n} \times X_{c_1} \times \cdots \times X_{c_n} \to X_{d_1} \times \cdots \times X_{d_m}
\end{equation}
with
\[
\lambda(\gamma, y_1, \ldots , y_n) = (z_1, \ldots , z_m)
\]
as above.  We claim that the operations $\lambda$ \eqref{weak0lambda} give the $\frakC$-graded set $\{X_c \colon c \in \frakC\}$ the structure of a $\sfP$-algebra. This will be proved in Lemma ~\ref{lem1:weak0algebra} below.  We then define
\[
\phi(X) = \{X_c \colon c \in \frakC\}
\]
with the $\sfP$-algebra structure maps $\lambda$.

We still need to specify what $\phi$ does to maps.  Let
\[
F \colon X \to X' \in \propset
\]
be a map of weak-$0$ $\sfP$-algebras.  For each $0$-dimensional $\sfP$-propertope $c \in \frakC$, we have a map
\begin{equation}
\label{Fcweak0}
F_c = F(c) \colon X(c) = X_c \to X'_c = X'(c).
\end{equation}
It will be proved in Lemma ~\ref{lem2:weak0algebra} below that
\[
\phi(F) = \{F_c \colon X_c \to X'_c \colon c \in \frakC\}
\]
is a map of $\sfP$-algebras from $\phi(X) = \{X_c\}$ to $\phi(X') = \{X'_c\}$.  The naturality of $\phi$ is clear from its definition.  Modulo Lemmas ~\ref{lem1:weak0algebra} and ~\ref{lem2:weak0algebra}, we have defined the functor
\[
\phi \colon \alg^0(\sfP) \to \alg(\sfP)
\]
in \eqref{weak0equivalence}.

Next we construct the functor $\psi$ in \eqref{weak0equivalence}.  Let $X = \{X_c \colon c \in \frakC\}$ be a $\sfP$-algebra.  There are $\sfP$-algebra structure maps $\lambda$ as in \eqref{weak0lambda}.  For any $0$-dimensional $\sfP$-propertope $c \in \frakC$, we set
\[
X(c) = X_c.
\]
Now let $\gamma \in \sfP\binom{\ud}{\uc}$ be a $1$-dimensional $\sfP$-propertope as in \eqref{gamma1propertope}.  We set
\begin{equation}
\label{Xgammadim1}
X(\gamma) = X_{c_1} \times \cdots \times X_{c_n}.
\end{equation}
The $i$th in-face map
\begin{equation}
\label{Xgammainface}
X(f_i) \colon X(\gamma) \to X(c_i) = X_{c_i} \quad (1 \leq i \leq n)
\end{equation}
is the projection onto the $i$th factor.  To define the out-face maps from $X(\gamma)$, pick elements $y_i \in X(c_i)$ $(1 \leq i \leq n)$.  The $\sfP$-algebra structure on $X$ gives an element
\[
\lambda(\gamma, y_1, \ldots , y_n) = (z_1, \ldots , z_m) \in X_{d_1} \times \cdots \times X_{d_m}.
\]
Define the $j$th out-face map
\begin{equation}
\label{1dimoutface}
X(g_j) \colon X(\gamma) \to X(d_j) = X_{d_j}
\end{equation}
by
\[
X(g_j)(y_1, \ldots , y_n) = z_j
\]
for $1 \leq j \leq m$.

Inductively, suppose that $N \geq 2$ and that we have already defined the sets $X(\alpha)$ for all $k$-dimensional $\sfP$-propertopes for $k < N$ and all the face maps in those dimensions.  Let
\[
\gamma \in \sfP\binom{\beta_1, \ldots , \beta_s}{\alpha_1, \ldots , \alpha_r} = \sfP\binom{\ubeta}{\ualpha}
\]
be an $N$-dimensional $\sfP$-propertope.  Define the set
\begin{equation}
\label{XgammadimN}
X(\gamma) = X(\alpha_1) \times \cdots \times X(\alpha_r) \times X(\beta_1) \times \cdots \times X(\beta_s),
\end{equation}
with face-maps
\begin{equation}
\label{XgammaNface}
\begin{split}
X(f_i) &\colon X(\gamma) \to X(\alpha_i) \quad (1 \leq i \leq r),\\
X(g_j) &\colon X(\gamma) \to X(\beta_j) \quad (1 \leq j \leq s)
\end{split}
\end{equation}
the corresponding projections.  By induction we have defined a functor
\[
X \colon \bP(\sfP)' \to \set,
\]
where $\bP(\sfP)'$ is the category in Remark ~\ref{remark:PP'}.

In the solid-arrow diagram
\begin{equation}
\label{KanextensionX}
\SelectTips{cm}{10}
\xymatrix{
\bP(\sfP)' \ar[d]_-{\pi} \ar[r]^-{X} & \set \\
\bP(\sfP) \ar@{.>}[ur]_-{\psi(X)} & ,
}
\end{equation}
the right Kan extension $\psi(X)$ of $X$ along $\pi$ exists because $\bP(\sfP)'$ is a small category and $\set$ is complete \cite[p.239 Corollary 2]{maclane2}.  Here $\pi$ is the quotient functor \eqref{piP} discussed in Remark ~\ref{remark:PP'}.  This defines the $\sfP$-propertopic set $\psi(X)$.  That $\psi(X)$ is a weak-$0$ $\sfP$-algebra can be seen follows.  From the definitions \eqref{Xgammadim1} and \eqref{Xgammainface} and the universal property of Kan extensions, it follows that $1$-dimensional horns in $\psi(X)$ have unique fillings.  Likewise, it follows from \eqref{XgammadimN}, \eqref{XgammaNface}, and the universal property of Kan extensions that $N$-dimensional boundaries in $\psi(X)$ for $N \geq 2$ have unique fillings.  Thus, by Proposition ~\ref{weak0algebra} $\psi(X)$ is a weak-$0$ $\sfP$-algebra.

The above construction of $\psi(X)$ is natural.  Indeed, for each $n$-dimensional $\sfP$-propertope $\gamma$ for $n \geq 1$, one observes from \eqref{Xgammadim1} and \eqref{XgammadimN} that the set $X(\gamma)$ is a finite product of some $X_c$ for $c \in \frakC$.  Thus, given a map
\[
F = \{F_c \colon X_c \to X'_c\}
\]
of $\sfP$-algebras, we can define the map
\[
F(\gamma) \colon X(\gamma) \to X'(\gamma)
\]
as the corresponding product of maps $F_c$.  These maps $F(\gamma)$ clearly commute with the face maps in $X$ and $X'$ that are defined as projections.  The only face maps that are not defined as projections are the lowest dimensional out-face maps $X(g_j)$ \eqref{1dimoutface}.  These out-face maps are defined as the components of the $\sfP$-algebra structure maps $\lambda$.  Since $F$ is a map of $\sfP$-algebras, the maps $F_c$ are compatible with the structure maps $\lambda$.  Thus, we have a map
\[
F \colon X \to X'
\]
in the functor category $\set^{\bP(\sfP)'}$.  This gives rise to the map
\[
\psi(F) \colon \psi(X) \to \psi(X') \in \propset
\]
by the naturality of right Kan extensions.  Therefore, we have defined the functor
\[
\psi \colon \alg(\sfP) \to \alg^0(\sfP)
\]
in \eqref{weak0equivalence}.  One can check that the functors $\phi$ and $\psi$ constructed above give an equivalence of the categories $\alg^0(\sfP)$ and $\alg(\sfP)$.
\end{proof}

To finish the construction of the functor $\phi$ in Theorem ~\ref{thm:weak0algebra}, we still need to prove the following two Lemmas.

\begin{lemma}
\label{lem1:weak0algebra}
Let $X$ be a weak-$0$ $\sfP$-algebra.  Then the operations $\lambda$ \eqref{weak0lambda} give the $\frakC$-graded set $\{X_c \colon c \in \frakC\}$ (where $X_c = X(c)$) the structure of a $\sfP$-algebra.
\end{lemma}

\begin{proof}
To show that the operations $\lambda$ give $\{X_c\}$ the structure of a $\sfP$-algebra, we need to show that they are bi-equivariant and are compatible with the horizontal composition $\otimes$, the vertical composition $\circ$, and the units in the unital $\frakC$-colored PROP $\sfP$.  Below we use the notation
\[
\begin{split}
X_{\ud} &= X_{d_1} \times \cdots \times X_{d_m}\\
&= X(d_1) \times \cdots \times X(d_m)
\end{split}
\]
for any $\frakC$-profile $\ud = (d_1, \ldots , d_m)$.

The compatibility of $\lambda$ with the horizontal composition $\otimes$ in $\sfP$ means the commutativity of the diagram
\begin{equation}
\label{Xlambdahorizontal}
\SelectTips{cm}{10}
\xymatrix{
\sfP\dbinom{\ud}{\uc} \times \sfP\dbinom{\ub}{\ua} \times X_{\uc,\ua} \ar[rr]^-{\text{shuffle}} \ar[d]_-{(\otimes, Id)} & & \left[\sfP\dbinom{\ud}{\uc} \times X_{\uc}\right] \times \left[\sfP\dbinom{\ub}{\ua} \times X_{\ua}\right] \ar[d]^{(\lambda,\lambda)}\\
\sfP\dbinom{\ud,\ub}{\uc,\ua} \times X_{\uc,\ua} \ar[rr]^-{\lambda} & & X_{\ud} \times X_{\ub} = X_{\ud,\ub}.
}
\end{equation}
To check that \eqref{Xlambdahorizontal} is commutative, pick $1$-dimensional $\sfP$-propertopes
\begin{equation}
\label{gammagamma'}
\gamma \in \sfP\binom{\ud}{\uc} = \sfP\binom{d_1, \ldots , d_m}{c_1, \ldots , c_n},\quad
\gamma' \in \sfP\binom{\ub}{\ua} = \sfP\binom{b_1, \ldots , b_k}{a_1, \ldots , a_l}
\end{equation}
and a $\gamma$-horn and a $\gamma'$-horn:
\begin{equation}
\label{yy'}
(y_1, \ldots , y_n) \in X_{\uc},\quad
(y_1', \ldots , y_l') \in X_{\ua}.
\end{equation}
These $1$-dimensional horns in $X$ have unique fillings,
\[
x \in X(\gamma) \quad\text{and}\quad
x' \in X(\gamma'),
\]
respectively.  The vertical maps $\lambda$ in \eqref{Xlambdahorizontal} are defined as the out-faces of $x$ and $x'$:
\begin{equation}
\label{lambdagammay}
\begin{split}
\lambda(\gamma, (y_1, \ldots , y_n)) &= (z_1, \ldots , z_m) = \left(X(g_1)(x), \ldots , X(g_m)(x)\right) \in X_{\ud},\\
\lambda(\gamma', (y_1', \ldots , y_l')) &= (z_1', \ldots , z_k') = \left(X(g_1)(x'), \ldots , X(g_k)(x')\right) \in X_{\ub}.
\end{split}
\end{equation}
On the other hand, the $1$-dimensional $(\gamma \otimes \gamma')$-horn
\[
(y_1, \ldots , y_n, y_1', \ldots , y_l') \in X_{\uc,\ua}
\]
in $X$ also has a unique filling
\[
\xbar \in X(\gamma \otimes \gamma').
\]
The horizontal $\lambda$ in \eqref{Xlambdahorizontal} is then defined as the out-faces of $\xbar$:
\[
\begin{split}
\lambda\left(\gamma \otimes \gamma', (y_1, \ldots , y_n, y_1', \ldots , y_l')\right)
&= (\zbar_1, \ldots , \zbar_m, \zbar'_1, \ldots , \zbar'_k)\\
&= \left(X(g_1)(\xbar), \ldots , X(g_{m+k})(\xbar)\right) \in X_{\ud,\ub}.
\end{split}
\]
The commutativity of the diagram \eqref{Xlambdahorizontal} is then equivalent to the equality
\begin{equation}
\label{Xlambdahorizontal'}
(z_1, \ldots , z_m, z_1', \ldots , z_k') = (\zbar_1, \ldots , \zbar_m, \zbar'_1, \ldots , \zbar'_k)
\end{equation}
in $X_{\ud,\ub}$.  We prove this equality using the horizontal consistency diagram \eqref{hconsistency2} with $\alpha = \gamma$ and $\beta = \gamma'$.

Consider the $2$-dimensional $\sfP$-propertope
\[
\setlength{\unitlength}{.6mm}
\begin{picture}(110,40)(15,12)
\put(30,30){\circle*{2}}      
\put(20,20){\vector(1,1){9.5}}
\put(40,20){\vector(-1,1){9.5}}
\put(30,30){\vector(-1,1){10}}
\put(30,30){\vector(1,1){10}}
\put(22,30){\makebox(0,0){$\gamma$}}
\put(31,23){\makebox(0,0){$\cdots$}}
\put(31,35){\makebox(0,0){$\cdots$}}
\put(20,16){\makebox(0,0){$1$}}
\put(40,16){\makebox(0,0){$n$}}
\put(20,43){\makebox(0,0){$1$}}
\put(40,43){\makebox(0,0){$m$}}
\put(60,30){\circle*{2}}      
\put(50,20){\vector(1,1){9.5}}
\put(70,20){\vector(-1,1){9.5}}
\put(60,30){\vector(-1,1){10}}
\put(60,30){\vector(1,1){10}}
\put(53,30){\makebox(0,0){$\gamma'$}}
\put(61,23){\makebox(0,0){$\cdots$}}
\put(61,35){\makebox(0,0){$\cdots$}}
\put(50,16){\makebox(0,0){$1$}}
\put(72,16){\makebox(0,0){$l$}}
\put(50,43){\makebox(0,0){$1$}}
\put(72,43){\makebox(0,0){$k$}}
\put(-4,30){\makebox(0,0){$G_{\gamma \otimes \gamma'} =$}}
\put(80,28){$\in \sfP^+\dbinom{\gamma \otimes \gamma'}{\gamma, \gamma'}$}
\end{picture}
\]
first defined in \eqref{Galphatensorbeta}.  From the previous paragraph, we have a $2$-dimensional $G_{\gamma\otimes\gamma'}$-boundary
\begin{equation}
\label{2dboundary}
(x,x') \xrightarrow{?} \xbar
\end{equation}
in $X$.  Since $X$ is a weak-$0$ $\sfP$-algebra, there is a unique filling
\[
w \in X(G_{\gamma\otimes\gamma'})
\]
of the $2$-dimensional boundary \eqref{2dboundary}.  The image under $X \in \propset$ of the horizontal consistency diagrams \eqref{hconsistency2} in this case are the commutative diagrams
\begin{equation}
\label{Xhconsistency2}
\SelectTips{cm}{10}
\xymatrix{
X(\gamma) \ar[d]_-{\text{out}} & X(G_{\gamma\otimes\gamma'}) \ar[l]_-{\text{in}} \ar[d]^-{\text{out}} \ar[r]^-{\text{in}} & X(\gamma') \ar[d]^-{\text{out}}\\
X(d_i) & X(\gamma\otimes\gamma') \ar[l]_-{\text{out}} \ar[r]^-{\text{out}} & X(b_j)
}
\end{equation}
for $1 \leq i \leq m$ and $1 \leq j \leq k$.  In the commutative diagram \eqref{Xhconsistency2}, the two top horizontal maps are in-face maps and the rest are the obvious out-face maps.  Starting at the element $w \in X(G_{\gamma\otimes\gamma'})$, its images in the lower-left corner $X(d_i)$ under the two composites in \eqref{Xhconsistency2} are $z_i$ and $\zbar_i$.  Thus, by commutativity we have
\[
z_i = \zbar_i.
\]
Likewise, the images of $w$ in the lower-right corner $X(b_j)$ in \eqref{Xhconsistency2} are $z_j'$ and $\zbar'_j$, so we have
\[
z_j' = \zbar'_j.
\]
This proves the equality \eqref{Xlambdahorizontal'} and hence the commutativity of the diagram \eqref{Xlambdahorizontal}.  Thus, we have shown that the operation $\lambda$ \eqref{weak0lambda} is compatible with the horizontal composition $\otimes$ in $\sfP$.

The compatibility of $\lambda$ with the vertical composition and the units in $\sfP$ are proved similarly using the vertical and the unital consistency diagrams \eqref{vconsistency} and \eqref{uconsistency}, respectively.  In proving the compatibility of $\lambda$ with the vertical composition $\circ$, one uses the $2$-dimensional $\sfP$-propertope
\[
\setlength{\unitlength}{.6mm}
\begin{picture}(50,58)(25,10)
\put(30,30){\circle*{2}}
\put(30,50){\circle*{2}}
\put(20,20){\vector(1,1){9.5}}
\put(40,20){\vector(-1,1){9.5}}
\put(30,50){\vector(-1,1){10}}
\put(30,50){\vector(1,1){10}}
\put(28,48){\vector(1,1){1.5}}
\put(32,48){\vector(-1,1){1.5}}
\qbezier(30,30)(15,39)(28,48)
\qbezier(30,30)(45,39)(32,48)
\put(20,30){\makebox(0,0){$\gamma'$}}
\put(20,50){\makebox(0,0){$\gamma$}}
\put(31,23){\makebox(0,0){$\cdots$}}
\put(31,39){\makebox(0,0){$\cdots$}}
\put(31,55){\makebox(0,0){$\cdots$}}
\put(17,18){\makebox(0,0){$1$}}
\put(45,18){\makebox(0,0){$|\ub|$}}
\put(17,63){\makebox(0,0){$1$}}
\put(45,63){\makebox(0,0){$|\ud|$}}
\put(-10,44){\makebox(0,0){$G_{\gamma\circ\gamma'} =$}}
\put(55,42){$\in \sfP^+\dbinom{\gamma\circ\gamma'}{\gamma,\gamma'}$}
\end{picture}
\]
first defined in \eqref{Galphabeta}.  Finally, the bi-equivariance of $\lambda$ is also proved by essentially the same argument as above using the equivariance consistency diagrams \eqref{econsistency}.
\end{proof}

\begin{lemma}
\label{lem2:weak0algebra}
If
\[
F \colon X \to X'
\]
is a map of weak-$0$ $\sfP$-algebras, then the maps $F_c$ $(c \in \frakC)$ \eqref{Fcweak0} give a map
\[
\phi(F) = \{F_c \colon X_c \to X'_c \colon c \in \frakC\}
\]
of $\sfP$-algebras from $\phi(X) = \{X_c\}$ to $\phi(X') = \{X'_c\}$.
\end{lemma}

\begin{proof}
We use the notations in the proof of Lemma ~\ref{lem1:weak0algebra}.  The map $\phi(F)$ is a map of $\sfP$-algebras if and only if the diagram
\begin{equation}
\label{Flambdacompatibility}
\SelectTips{cm}{10}
\xymatrix{
\sfP\dbinom{\ud}{\uc} \times X_{\uc} \ar[rr]^-{\lambda} \ar[d]_-{(Id,F_{\uc})} & & X_{\ud} \ar[d]^{F_{\ud}}\\
\sfP\dbinom{\ud}{\uc} \times X'_{\uc} \ar[rr]^-{\lambda} & & X'_{\ud}
}
\end{equation}
is commutative, where
\[
F_{\uc} = F_{c_1} \times \cdots \times F_{c_n}
\]
for $\uc = (c_1, \ldots , c_n)$.  Let $\gamma \in \sfP\binom{\ud}{\uc}$ and $(y_1, \ldots , y_n) \in X_{\uc}$ be as in \eqref{gammagamma'} and \eqref{yy'}, respectively.  We have
\[
F_{\ud}\lambda(\gamma, (y_1, \ldots , y_n)) = \left(F_{d_1}(z_1), \ldots , F_{d_m}(z_m)\right) \in X'_{\ud},
\]
where the
\[
z_j = X(g_j)(x) \in X(d_j)
\]
are the out-faces of $x \in X(\gamma)$ as in \eqref{lambdagammay}.

On the other hand, we have the $1$-dimensional $\gamma$-horn
\begin{equation}
\label{Fchorn}
\left(F_{c_1}(y_1), \ldots , F_{c_n}(y_n)\right) \in X'_{\uc}
\end{equation}
in the weak-$0$ $\sfP$-algebra $X'$.  There is a unique filling $x' \in X'(\gamma)$ of this $\gamma$-horn in $X'$.  Then we have
\[
\lambda\left(\gamma, \left(F_{c_1}(y_1), \ldots , F_{c_n}(y_n)\right)\right) = (X'(g_1)(x'), \ldots , X'(g_m)(x')),
\]
where
\[
g_j \colon \gamma \to d_j
\]
is the $j$th out-face map from $\gamma$.  The commutativity of \eqref{Flambdacompatibility} is equivalent to the equality
\begin{equation}
\label{FdzX'd}
F_{d_j}(z_j) = X'(g_j)(x') \in X'(d_j)
\end{equation}
for $1 \leq j \leq m$.

To prove \eqref{FdzX'd}, we first claim that
\begin{equation}
\label{Fgammaxx'}
x' = F(\gamma)(x).
\end{equation}
Indeed, since $F$ is a map of $\sfP$-propertopic sets, the diagram
\[
\SelectTips{cm}{10}
\xymatrix{
X(\gamma) \ar[r]^-{F(\gamma)} \ar[d]_{X(f_i)} & X'(\gamma) \ar[d]^-{X'(f_i)}\\
X(c_i) \ar[r]^-{F_{c_i}} & X'(c_i)
}
\]
is commutative, where
\[
f_i \colon \gamma \to c_i
\]
is the $i$th in-face map from $\gamma$.  Applied to $x \in X(\gamma)$, the commutativity of this diagram means that
\[
\begin{split}
F_{c_i}(y_i) &= F_{c_i}(X(f_i)(x))\\
&= X'(f_i)(F(\gamma)(x)).
\end{split}
\]
Since this equality holds for $1 \leq i \leq n$, we conclude that $F(\gamma)(x) \in X'(\gamma)$ is also a filling of the $1$-dimensional $\gamma$-horn \eqref{Fchorn}.  The uniqueness of the filling $x' \in X'(\gamma)$ of this $\gamma$-horn then gives the equality \eqref{Fgammaxx'}.

Since $F$ is compatible with all the face maps, we also have the commutative diagram
\begin{equation}
\label{Xgammagd}
\SelectTips{cm}{10}
\xymatrix{
X(\gamma) \ar[r]^-{F(\gamma)} \ar[d]_{X(g_j)} & X'(\gamma) \ar[d]^-{X'(g_j)}\\
X(d_j) \ar[r]^-{F_{d_j}} & X'(d_j).
}
\end{equation}
Therefore, we have
\[
\begin{split}
X'(g_j)(x') &= X'(g_j)(F(\gamma)(x)) \quad (\text{by \eqref{Fgammaxx'}})\\
&= F_{d_j}(X(g_j)(x)) \quad (\text{by \eqref{Xgammagd}})\\
&= F_{d_j}(z_j).
\end{split}
\]
This proves the equality \eqref{FdzX'd} and hence the commutativity of the diagram \eqref{Flambdacompatibility}.
\end{proof}

This completes the proof of Theorem ~\ref{thm:weak0algebra}.

\subsection{Eilenberg-Mac Lane weak-$n$ $\sfP$-algebras}
\label{subsec:EM}

Theorem ~\ref{thm:weak0algebra} has a generalization for weak-$n$ $\sfP$-algebras for $n \geq 1$.  Indeed, a close inspection of its proof reveals that much of it depends only on the existence of unique fillings of $1$-dimensional horns and $N$-dimensional boundaries for $N \geq 2$.  Weak-$n$ $\sfP$-algebras for $n \geq 1$ also have unique fillings of $(n+1)$-dimensional horns and $N$-dimensional boundaries for $N \geq n+2$.  These unique fillings are only about the $r$-cells for $r \geq n$.  In particular, to obtain a generalization of Theorem ~\ref{thm:weak0algebra} to $n \geq 1$, we should consider a version of a weak-$n$ $\sfP$-algebra that has trivial $k$-dimensional cells for $k < n$. 

\begin{definition}
\label{connectedweakn}
For $0 \leq n \leq \infty$, denote by $\algn(\sfP)'$ the full subcategory of $\algn(\sfP)$ consisting of the weak-$n$ $\sfP$-algebras $X$ such that
\[
X(\gamma) = \{*\}
\]
for any $k$-dimensional $\sfP$-propertope $\gamma \in \elt(\sfP^{(k-1)+})$ with $0 \leq k < n$.  Objects in $\algn(\sfP)'$ are called \textbf{Eilenberg-Mac Lane weak-$n$ $\sfP$-algebras}.
\end{definition}

Note that
\[
\alg^0(\sfP)' = \alg^0(\sfP).
\]
So Eilenberg-Mac Lane weak-$0$ $\sfP$-algebras are really just weak-$0$ $\sfP$-algebras, which by Theorem ~\ref{thm:weak0algebra} are equivalent to $\sfP$-algebras.  When $n = \infty$, the definition above is of little interest because the only object in $\alg^\infty(\sfP)'$ is the terminal $\sfP$-propertopic set $*$ \eqref{terminalpropset}.  So Eilenberg-Mac Lane weak-$n$ $\sfP$-algebras are only interesting when $n < \infty$.

Let us explain the terminology in Definition ~\ref{connectedweakn}.  Recall from the discussion just before Definition ~\ref{def:weakPalgebra} that weak-$n$ $\sfP$-algebras are $\sfP$-propertopic analogues of homotopy $n$-types, which are simplicial sets with trivial homotopy groups in dimensions $> n$.  Moreover, a simplicial set $A$ with $A(\bk) = \{*\}$ for all $k \leq n-1$ is \emph{$(n-1)$-connected}, i.e., has trivial homotopy groups in dimensions $\leq n-1$.  An $(n-1)$-connected simplicial set $A$ that is also a homotopy $n$-type has a non-trivial homotopy group only in dimension $n$.  By definition such a simplicial set $A$ is called an \emph{Eilenberg-Mac Lane space}.

An object $X \in \algn(\sfP)'$ has, by definition, a unique $\gamma$-cell for each $k$-dimensional $\sfP$-propertope $\gamma$ with $k \leq n-1$.  So we can think of it as an $(n-1)$-connected version of a weak-$n$ $\sfP$-algebra.  From the discussion of the previous paragraph, therefore, it makes sense to consider an object $X \in \algn(\sfP)'$ as a kind of Eilenberg-Mac Lane object.  This explains our terminology.

The following result says that the category of Eilenberg-Mac Lane weak-$n$ $\sfP$-algebras is a \emph{reflection} of the category of weak-$n$ $\sfP$-algebras.

\begin{proposition}
\label{P'reflective}
The category $\algn(\sfP)'$ is a reflective subcategory of $\algn(\sfP)$.  In other words, the inclusion functor
\[
i \colon \algn(\sfP)' \hookrightarrow \algn(\sfP)
\]
has a left adjoint
\[
p \colon \algn(\sfP) \to \algn(\sfP)'.
\]
\end{proposition}

\begin{proof}
The functor $p$ is defined as follows.  Suppose that $X \in \algn(\sfP)$.  For a $k$-dimensional $\sfP$-propertope $\gamma \in \elt(\sfP^{(k-1)+})$, define
\begin{equation}
\label{pX}
p(X)(\gamma) =
\begin{cases}
X(\gamma) &\text{if $k \geq n$},\\
\{*\} &\text{otherwise}.
\end{cases}
\end{equation}
The face maps in $p(X)$ are those in $X$ if the targets have dimensions $\geq n$.  Otherwise, they are the unique maps to the one-element set $\{*\}$.  One can check that $p(X)$ is indeed an object in $\algn(\sfP)'$.  It is obvious how $p$ is defined at a map in $\algn(\sfP)$.  One can then check that $p$ is the left adjoint to the inclusion functor $i$.
\end{proof}

\begin{corollary}
\label{cor1:P'}
A weak-$n$ $\sfP$-algebra $X \in \algn(\sfP)$ is uniquely determined by:
\begin{enumerate}
\item
its image $p(X) \in \algn(\sfP)'$, and
\item
the restriction diagram of $X \in \propset$ to $k$-dimensional $\sfP$-propertopes for $k \leq n + 1$ and the face maps in those dimensions.
\end{enumerate}
\end{corollary}

\begin{proof}
This follows immediately from the definition of the functor $p$ \eqref{pX}.
\end{proof}

In view of this Corollary, the study of weak-$n$ $\sfP$-algebras splits into the following two parts.
\begin{enumerate}
\item
Understand the category $\algn(\sfP)'$ of Eilenberg-Mac Lane weak-$n$ $\sfP$-algebras.
\item
Study the restriction diagram of a weak-$n$ $\sfP$-algebra $X$ to the full subcategory
\begin{equation}
\label{Pn+1}
\bP^{n+1}(\sfP) \subseteq \bP(\sfP)
\end{equation}
consisting of the $k$-dimensional $\sfP$-propertopes for $k \leq n + 1$.
\end{enumerate}
When $n = 0$, this reduces to understanding $\alg^0(\sfP)' = \alg^0(\sfP)$.  This was done in Theorem ~\ref{thm:weak0algebra}, which says that $\alg^0(\sfP)$ is equivalent to the category of $\sfP$-algebras.

In the general case $n \geq 0$, Corollary ~\ref{cor1:P'} is somewhat analogous to a basic principle in homotopy theory.  Since this is how we will try to understand weak-$n$ $\sfP$-algebras below, it is worth recalling this basic piece of homotopy theory.  Every connected CW complex $Y$ has a \emph{Postnikov tower}\label{postnikovstragegy} \cite[Chapter 4]{hatcher}.  In particular, if $Y$ is a homotopy $n$-type, then its Postnikov tower is determined by the bottom $n$ stages.  The top of this $n$-stage Postnikov tower is the diagram
\[
\SelectTips{cm}{10}
\xymatrix{
Y \ar@{>>}[d]^-{\varphi} & K(\pi_n(Y),n) \ar[l]_-{i}\\
Y_{n-1}.
}
\]
Here $\varphi$ is a fibration that induces an isomorphism in homotopy groups in dimensions $\leq n-1$, and $Y_{n-1}$ is a homotopy $(n-1)$-type.  The fiber of this fibration is the Eilenberg-Mac Lane space $K(\pi_n(Y),n)$.  Although the spaces $Y_{n-1}$ and $ K(\pi_n(Y),n)$ contain all the homotopy groups of $Y$, they do \emph{not} determine the homotopy type of $Y$.  One needs to know how $Y_{n-1}$ and $K(\pi_n(Y),n)$ are glued together, and this is what the fibration $\varphi$ does.

Analogously, a weak-$n$ $\sfP$-algebra $X$ is a $\sfP$-propertopic version of a homotopy $n$-type. The restriction diagram of $X \in \propset$ to $k$-dimensional $\sfP$-propertopes for $k \leq n$ contains the lower dimensional information of $X$.  The image $p(X) \in \algn(\sfP)'$ contains the higher dimensional information of $X$.  Although these two pieces of $X$ together have all the cells and face maps in $X$ (and even $X_n$ in common), they do not determine $X$.  One needs to know how these pieces are glued together as well.  The consistency conditions (\eqref{hconsistency1} -- \eqref{vconsistency} and \eqref{econsistency}) for face maps between dimensions $n-1$, $n$, and $n+1$ are the necessary \emph{gluing data}.  This is why, in Corollary ~\ref{cor1:P'}, we need the restriction diagram of $X$ to $k$-dimensional $\sfP$-propertopes for $k \leq n+1$.  Had we used $k \leq n$, the gluing data would not have been accounted for.

Following the recipe above, we now study the higher dimensional information of weak-$n$ $\sfP$-algebras.  Recall the unital colored PROP $\sfP^{n+}$ from Definition ~\ref{def:Pn+}.  Its set of colors is $\elt(\sfP^{(n-1)+})$, the set of $n$-dimensional $\sfP$-propertopes.  We now observe that $\algn(\sfP)'$ is equivalent to the relatively well-behaved category of algebras over the colored PROP $\sfP^{n+}$.  This is the analogue of Theorem ~\ref{thm:weak0algebra} for $n \geq 1$.

\begin{theorem}
\label{thm:EMweakn}
For $1 \leq n < \infty$, there exist functors
\[
\phi^n \colon \algn(\sfP)' \rightleftarrows \alg(\sfP^{n+}) \colon \psi^n
\]
that give an equivalence between the categories $\algn(\sfP)'$ of Eilenberg-Mac Lane weak-$n$ $\sfP$-algebras and $\alg(\sfP^{n+})$ of $\sfP^{n+}$-algebras.
\end{theorem}

\begin{proof}
As discussed just before Definition ~\ref{connectedweakn}, the proof of Theorem ~\ref{thm:weak0algebra} can be used here with only cosmetic changes.  Indeed, suppose that $X \in \algn(\sfP)'$ and that
\[
\gamma \in \sfP^{n+}\binom{\beta_1, \ldots , \beta_s}{\alpha_1, \ldots , \alpha_r} = \sfP^{n+}\binom{\ubeta}{\ualpha}
\]
is an $(n+1)$-dimensional $\sfP$-propertope.  If $y_i \in X(\alpha_i)$ is an $n$-dimensional $\alpha_i$-cell in $X$ for $1 \leq i \leq r$, then
\[
(y_1, \ldots , y_r) \xrightarrow{?} ?
\]
is an $(n+1)$-dimensional $\gamma$-horn in $X$.  Since $X$ is a weak-$n$ $\sfP$-algebra, this horn has a unique filling
\[
x \in X(\gamma).
\]
This allows us to define the operation
\[
\lambda \colon \sfP^{n+}\binom{\ubeta}{\ualpha} \times X(\alpha_1) \times \cdots \times X(\alpha_r) \to X(\beta_1) \times \cdots \times X(\beta_s)
\]
with
\begin{equation}
\label{lambdagammayz}
\lambda\left(\gamma, (y_1, \ldots , y_r)\right)
= (z_1, \ldots , z_s),
\end{equation}
where $z_j \in X(\beta_j)$ is the $j$th out-face of $x$.  Then essentially the same argument as in the proof of Lemma ~\ref{lem1:weak0algebra} shows that
\[
\phi^n(X) = \{X(\alpha) \colon \alpha \in \elt(\sfP^{(n-1)+})\}
\]
is a $\sfP^{n+}$-algebra with structure maps $\lambda$.  Together with a minor variation of Lemma ~\ref{lem2:weak0algebra}, this gives the functor $\phi^n$.

The functor $\psi^n$ is similarly adapted from $\psi$ \eqref{weak0equivalence}.  If $X = \{X_\alpha\}$ is a $\sfP^{n+}$-algebra, then we set
\[
X(\gamma) = \{*\}
\]
for any $k$-dimensional $\sfP$-propertope $\gamma$ with $k \leq n - 1$.  Face maps in $X$ landing in dimensions $< n$ are the unique maps to the one-element set.  The sets $X(\gamma)$ for $\gamma$ of dimensions $\geq n+1$ are defined as certain products of the $X_\alpha$ $(\alpha \in \elt(\sfP^{(n-1)+})$ as in \eqref{Xgammadim1} and \eqref{XgammadimN}.  The face maps in these dimensions are the corresponding projection maps.  The out-face maps from $(n+1)$-cells to $n$-cells in $X$ are defined by the $\sfP^{n+}$-algebra structure on $X$ as in \eqref{1dimoutface}.  The data defined so far is a functor
\[
X \colon \bP(\sfP)' \to \set.
\]
As in \eqref{KanextensionX}, one takes the right Kan extension of this $X$ along $\pi$ to obtain $\psi^n(X) \in \propset$.  Using the universal properties of Kan extensions, one checks that this $\psi^n(X)$ is actually an object in $\algn(\sfP)'$ and that $\phi^n$ and $\psi^n$ give an equivalence of categories.
\end{proof}


\subsection{Categorical description of weak-$n$ $\sfP$-algebras}
\label{subsec:catweakn}

Here we give a categorical description of 
weak-$n$ $\sfP$-algebras for $1 \leq n \leq \infty$. In the rest of this section, the condition $i \leq k \leq n$ means $k \geq i$ if $n = \infty$.

Let $X \in \propset$ be a weak-$n$ $\sfP$-algebra for some $n$ in the range $1 \leq n \leq \infty$.  For each $k$ in the range $0 \leq k \leq n$, $X$ has a set of \textbf{$k$-cells}
\[
X_k = \coprod_{\alpha \in \elt(\sfP^{(k-1)+})} X(\alpha).
\]
We also call the elements in $X_0$ the \textbf{objects} in $X$. If $n < \infty$, then the $n$-cells are also called the \textbf{top cells}.

For $1 \leq k \leq n$, if
\[
\alpha \in \sfP^{(k-1)+}\binom{\varepsilon_1, \ldots , \varepsilon_s}{\delta_1, \ldots , \delta_r} = \sfP^{(k-1)+}\binom{\uepsilon}{\udelta}
\]
is a $k$-dimensional $\sfP$-propertope, then its in-face and out-face maps are
\begin{equation}
\label{Xnfacemaps}
\SelectTips{cm}{10}
\xymatrix{
 & & X(\alpha) \ar@{>>}[dll]_-{(X(f_1), \ldots , X(f_r))} \ar[drr]^-{(X(g_1), \ldots , X(g_s))} & & \\
X(\delta_1) \times \cdots \times X(\delta_r) & & & & X(\varepsilon_1) \times \cdots \times X(\varepsilon_s).
}
\end{equation}
Since $X$ is a weak-$n$ $\sfP$-algebra, $k$-dimensional horns have fillings for $1 \leq k \leq n$.  So the combined in-face map
\begin{equation}
\label{Xncombindedinface}
(X(f_1), \ldots , X(f_r)) \colon X(\alpha) \to X(\delta_1) \times \cdots \times X(\delta_r)
\end{equation}
in \eqref{Xnfacemaps} is surjective.  We thus have a \textbf{multi-valued composition function of $(k-1)$-cells}:
\begin{equation}
\label{Xnmulticomposition}
\alpha \colon X(\delta_1) \times \cdots \times X(\delta_r) \to X(\varepsilon_1) \times \cdots \times X(\varepsilon_s).
\end{equation}
The image of a sequence $\uy \in \prod X(\delta_i)$ of $(k-1)$-cells under the multi-valued composition function $\alpha$ is the non-empty subset
\[
\alpha(\uy) = \left\{(X(g_1), \ldots , X(g_s))(x) \colon x \in X(\alpha),\, (X(f_1), \ldots , X(f_r))(x) = \uy \right\}
\]
of $\prod X(\varepsilon_j)$.

If $x \in X(\alpha)$ is a $k$-cell for $1 \leq k \leq n$, then we call the sequences of $(k-1)$-cells
\[
\begin{split}
(y_1, \ldots , y_r) &= (X(f_1), \ldots , X(f_r))(x) \in X(\delta_1) \times \cdots \times X(\delta_r),\\
(z_1, \ldots, z_s) &= (X(g_1), \ldots , X(g_s))(x) \in X(\varepsilon_1) \times \cdots \times X(\varepsilon_s)
\end{split}
\]
the \textbf{source} and the \textbf{target} of $x$, respectively.  We depict such a $k$-cell $x$ together with its source and target as
\[
(y_1, \ldots , y_r) \xrightarrow{x} (z_1, \ldots , z_s).
\]
We think of the target $(z_1, \ldots , z_s)$ as either:
\begin{enumerate}
\item
a composite of $(y_1, \ldots , y_r)$ of shape $\alpha$ or
\item
the composite of $(y_1, \ldots , y_r)$ via $x$.
\end{enumerate}
The surjectivity of the combined in-face map \eqref{Xncombindedinface} implies that each sequence $(y_1, \ldots , y_n) \in \prod X(\delta_i)$ of $(k-1)$-cells has \emph{at least one} composite of shape $\alpha$.  In general, there may be many different composites of the $(k-1)$-cells $(y_1, \ldots , y_n)$.

The source and target (i.e., the in-faces and the out-faces) of $k$-cells for $1 \leq k \leq n$ satisfy the horizontal, vertical, unital, and equivariance consistency conditions of $X \in \propset$.  These conditions are the images under $X$ of the commutative diagrams \eqref{hconsistency1} -- \eqref{econsistency} in the specified dimensions.

If $n = \infty$, then we have given a categorical description of a weak-$\omega$ $\sfP$-algebra.

If $n < \infty$, then we still need to describe compositions of the top cells, i.e., the $n$-cells.  If
\[
\gamma \in \sfP^{n+}\binom{\beta_1, \ldots , \beta_s}{\alpha_1, \ldots , \alpha_r} = \sfP^{n+}\binom{\ubeta}{\ualpha}
\]
is an $(n+1)$-dimensional $\sfP$-propertope, then its in-face and out-face maps are
\begin{equation}
\label{weaknhigherface}
\SelectTips{cm}{10}
\xymatrix{
 & & X(\gamma) \ar[dll]^-{\cong}_-{(X(f_1), \ldots , X(f_r))} \ar[drr]^-{(X(g_1), \ldots , X(g_s))} & & \\
X(\alpha_1) \times \cdots \times X(\alpha_r) & & & & X(\beta_1) \times \cdots \times X(\beta_s).
}
\end{equation}
Since $X$ is a weak-$n$ $\sfP$-algebra with $n < \infty$, every $(n+1)$-dimensional horn has a \emph{unique} filling.  So the combined in-face map in \eqref{weaknhigherface} is a bijection as indicated.  Using the inverse of this bijection and the combined out-face map in \eqref{weaknhigherface}, we thus have a \textbf{composition function of the top cells}
\begin{equation}
\label{ndimcomposition}
\gamma \colon X(\alpha_1) \times \cdots \times X(\alpha_r) \to X(\beta_1) \times \cdots \times X(\beta_s).
\end{equation}
These composition functions, the source, and target of the top cells satisfy the consistency conditions of $X \in \propset$ (\eqref{hconsistency1} -- \eqref{econsistency}).

Moreover, the composition functions $\gamma$ \eqref{ndimcomposition} of the top cells give the $\elt(\sfP^{(n-1)+})$-graded set
\begin{equation}
\label{topcellPnalgebra}
\{X(\alpha) \colon \alpha \in \elt(\sfP^{(n-1)+})\}
\end{equation}
the structure of a $\sfP^{n+}$-algebra.  This is a consequence of the consistency conditions of $X$ in dimensions $\geq n$ and the existence of unique fillings of $(n+1)$-dimensional horns and $N$-dimensional boundaries for $N \geq n+2$  (Proposition ~\ref{P'reflective} and Theorem ~\ref{thm:EMweakn}).

In summary, in a weak-$n$ $\sfP$-algebra $X$ with $1 \leq n \leq \infty$:
\begin{enumerate}
\item
There is a set $X_k$ of \textbf{$k$-cells} for each $k$ in the range $0 \leq k \leq n$.
\item
For $0 \leq k < n$, composition of $k$-cells via $(k+1)$-cells is \emph{not} a function.  Instead, it is in general a \textbf{multi-valued composition function} \eqref{Xnmulticomposition}, satisfying some consistency conditions (\eqref{hconsistency1} -- \eqref{econsistency}).  The multi-valued composition of $k$-cells $(0 \leq k < n)$ via $(k+1)$-cells encodes the \textbf{coherence laws} of the $k$-cells.
\item
If $n < \infty$, then composition of the \textbf{top cells} (i.e., the $n$-cells) is an honest operation \eqref{ndimcomposition}.  This operation gives the top cells the structure of a $\sfP^{n+}$-algebra \eqref{topcellPnalgebra}.  This $\sfP^{n+}$-algebra encodes the \textbf{coherence laws} of the top cells.
\item
The composition function, the source, and the target of the top cells also satisfy the consistency conditions (\eqref{hconsistency1} -- \eqref{econsistency}).
\end{enumerate}

\subsection{Underlying category of a weak-$n$ $\sfP$-algebra}
\label{subsec;underlyingcat}

In a weak-$n$ $\sfP$-algebra $X$ with $1 \leq n < \infty$, we have seen that composition of the top cells (i.e., the $n$-cells) is an honest operation.  So it appears that $X$ should have an underlying category whose objects are the $(n-1)$-cells and whose morphisms are certain $n$-cells.  Here we prove directly that this is true, without referring to the $\sfP^{n+}$-algebra structure on $p(X) \in \algn(\sfP)'$.  So for each unital $\frakC$-colored PROP $\sfP$ over $\set$, every weak-$n$ $\sfP$-algebra with $1 \leq n < \infty$ has an underlying category.

\begin{theorem}
\label{thm:weakncategory}
Let $\sfP$ be a unital $\frakC$-colored PROP, and let $X$ be a weak-$n$ $\sfP$-algebra with $1 \leq n < \infty$.  Then there is a category $\Xtilde$ (without identity) in which:
\begin{enumerate}
\item
the objects are the $(n-1)$-cells
\[
X_{n-1} = \coprod_{\alpha \in \elt(\sfP^{(n-2)+})} X(\alpha)
\]
in $X$;
\item
the set of morphisms $\Xtilde(y,z)$, with $y \in X(\alpha)$ and $z \in X(\beta)$, is the set of fillings of the $n$-dimensional $\gamma$-boundary
\[
y \xrightarrow{?} z
\]
in $X$, where $\gamma$ runs through the $\sfP$-propertopes in $\sfP^{(n-1)+}\binom{\beta}{\alpha}$.
\end{enumerate}
\end{theorem}

\begin{proof}
First we define composition of morphisms in $\Xtilde$.  Consider morphisms
\[
f \colon y_1 \to y_2 \in \Xtilde \quad\text{and}\quad
g \colon y_2 \to y_3 \in \Xtilde
\]
with $y_i \in X(\alpha_i)$ $(1 \leq i \leq 3)$, $f \in X(\alpha)$, $g \in X(\beta)$, $\alpha \in \sfP^{(n-1)+}\binom{\alpha_2}{\alpha_1}$, and $\beta \in \sfP^{(n-1)+}\binom{\alpha_3}{\alpha_2}$.  To define the composition $gf$, consider the $(n+1)$-dimensional $G_{\beta \circ \alpha}$-horn
\[
(g,f) \xrightarrow{?} ?
\]
in $X$, where
\begin{equation}
\label{Gbetaalpha}
\setlength{\unitlength}{.6mm}
\begin{picture}(50,53)(25,10)
\put(30,30){\circle*{2}}
\put(30,45){\circle*{2}}
\put(30,45){\vector(0,1){10}}
\put(30,30){\vector(0,1){14.5}}
\put(30,20){\vector(0,1){9.5}}
\put(25,30){\makebox(0,0){$\alpha$}}
\put(25,45){\makebox(0,0){$\beta$}}
\put(-3,40){\makebox(0,0){$G_{\beta\circ\alpha} =$}}
\put(45,38){$\in \sfP^{n+}\dbinom{\beta\circ\alpha}{\beta,\alpha}$}
\end{picture}
\end{equation}
is first defined in \eqref{Galphabeta}.  Since $X$ is a weak-$n$ $\sfP$-algebra, this $(n+1)$-dimensional $G_{\beta\circ\alpha}$-horn has a unique filling
\[
w_{g,f} \in X(G_{\beta\circ\alpha}).
\]
The out-face of $w_{g,f}$ in $X(\beta\circ\alpha)$ is, by definition, the composition $gf$, so we have
\[
(g,f) \xrightarrow{w_{g,f}} gf.
\]
We need to check the following statements:
\begin{enumerate}
\item
The composition $gf \in X(\beta\circ\alpha)$ is actually a morphism
\[
gf \colon y_1 \to y_3 \in \Xtilde.
\]
In other words, the in-face of $gf$ is $y_1$, and the out-face of $gf$ is $y_3$.
\item
Composition is associative.
\end{enumerate}

To check that $gf \in \Xtilde(y_1,y_3)$, consider the vertical consistency condition for $X$ \eqref{vconsistency} when applied to $G_{\beta\circ\alpha}$.  In this case, the vertical consistency condition says that the following diagram is commutative:
\[
\SelectTips{cm}{10}
\xymatrix{
X(\beta) \ar[d]_-{\text{out}} & X(G_{\beta\circ\alpha}) \ar[l]_-{\text{in}} \ar[d]^-{\text{out}} \ar[r]^-{\text{in}} & X(\alpha) \ar[d]^-{\text{in}}\\
X(\alpha_3) & X(\beta\circ\alpha) \ar[l]_-{\text{out}} \ar[r]^-{\text{in}} & X(\alpha_1).
}
\]
Each map in this commutative diagram is either an in-face map or an out-face map as indicated.  Starting with the $(n+1)$-cell $w_{g,f} \in X(G_{\beta\circ\alpha})$, the ``left-followed-by-down" composite yields $y_3$.  On the other hand, the ``down-followed-by-left" composite yields the out-face of $gf$.  Since the left square is commutative, we conclude that the out-face of $gf$ is $y_3$.  Similarly, by considering the images of $w_{g,f}$ in $X(\alpha_1)$ under the two composites in the right square, we see that the in-face of $gf$ is $y_1$.  This shows that $gf \in \Xtilde(y_1,y_3)$.

Next we check that composition in $\Xtilde$ is associative.  In other words, suppose that
\[
h \colon y_3 \to y_4 \in \Xtilde
\]
is a morphism with $y_4 \in X(\alpha_4)$, $h \in X(\gamma)$, and $\gamma \in \sfP\binom{\alpha_4}{\alpha_3}$.  Then we need to show that
\begin{equation}
\label{Xtildeassociative}
h(gf) = (hg)f \in X(\gamma\circ\beta\circ\alpha).
\end{equation}
Consider the $(n+1)$-dimensional $G_{\gamma\circ\beta\circ\alpha}$-horn
\begin{equation}
\label{hgfhorn1}
(h,g,f) \xrightarrow{?} ?
\end{equation}
in $X$, where, using the notations in \eqref{pbar+},
\begin{equation}
\label{Ggammabetaalpha}
\setlength{\unitlength}{.6mm}
\begin{picture}(50,70)(50,10)
\put(30,30){\circle*{2}}
\put(30,45){\circle*{2}}
\put(30,60){\circle*{2}}
\put(30,20){\vector(0,1){9.5}}
\put(30,30){\vector(0,1){14.5}}
\put(30,45){\vector(0,1){14.5}}
\put(30,60){\vector(0,1){10}}
\put(25,30){\makebox(0,0){$\alpha$}}
\put(25,45){\makebox(0,0){$\beta$}}
\put(25,60){\makebox(0,0){$\gamma$}}
\put(-3,45){\makebox(0,0){$G_{\gamma\circ\beta\circ\alpha} =$}}
\put(45,43){$\in \overline{\sfP^{(n-1)+}}^{+}_{(3)}(\gamma\circ\beta\circ\alpha; \gamma, \beta, \alpha) \subseteq  \sfP^{n+}\dbinom{\gamma\circ\beta\circ\alpha}{\gamma,\beta,\alpha}$.}
\end{picture}
\end{equation}
Here the three vertices are labeled $1$, $2$, and $3$, respectively, from top to bottom.  The four edges are decorated by $\alpha_4$, $\alpha_3$, $\alpha_2$, and $\alpha_1$, respectively, from top to bottom.  The $(n+1)$-dimensional $G_{\gamma\circ\beta\circ\alpha}$-horn \eqref{hgfhorn1} has a unique filling
\[
w_{h,g,f} \in X(G_{\gamma\circ\beta\circ\alpha}).
\]
Define
\[
hgf = \text{the out-face of $w_{h,g,f}$ in $X(\gamma\circ\beta\circ\alpha)$}.
\]
To prove the required equality \eqref{Xtildeassociative}, it suffices to show that
\begin{equation}
\label{Xtildeass}
h(gf) = hgf \quad\text{and}\quad (hg)f = hgf.
\end{equation}
We will show that
\begin{equation}
\label{Xtildeass1}
h(gf) = hgf
\end{equation}
only, since the second equality in \eqref{Xtildeass} can be proved by essentially the same argument.

To prove \eqref{Xtildeass1}, first we spell out how $h(gf)$ is defined.  From the very definition of composition, we need to consider the $(n+1)$-dimensional $G_{\gamma\circ(\beta\circ\alpha)}$-horn
\[
(h, gf) \xrightarrow{?} ?
\]
in $X$, where
\begin{equation}
\label{Ggammabetaalpha'}
\setlength{\unitlength}{.6mm}
\begin{picture}(50,53)(25,10)
\put(30,30){\circle*{2}}
\put(30,45){\circle*{2}}
\put(30,45){\vector(0,1){10}}
\put(30,30){\vector(0,1){14.5}}
\put(30,20){\vector(0,1){9.5}}
\put(20,30){\makebox(0,0){$\beta\circ\alpha$}}
\put(25,45){\makebox(0,0){$\gamma$}}
\put(-12,40){\makebox(0,0){$G_{\gamma\circ(\beta\circ\alpha)} =$}}
\put(45,38){$\in \sfP^{n+}\dbinom{\gamma\circ\beta\circ\alpha}{\gamma,\beta\circ\alpha}$.}
\end{picture}
\end{equation}
This $(n+1)$-dimensional $G_{\gamma\circ(\beta\circ\alpha)}$-horn has a unique filling
\[
w_{h,gf} \in X(G_{\gamma\circ(\beta\circ\alpha)}),
\]
whose out-face in $X(\gamma\circ\beta\circ\alpha)$ is, by definition, $h(gf)$.  To prove \eqref{Xtildeass1}, we need to understand the relationships between $w_{h,g,f}$ and $w_{h,gf}$.  Thus, we should first understand the relationships between their shapes, i.e., the various $G$'s in $\sfP^{n+}$.

Consider the vertical composition
\[
\sfP^{n+}\binom{\gamma\circ\beta\circ\alpha}{\gamma, \beta\circ\alpha} \times \sfP^{n+}\binom{\gamma, \beta\circ\alpha}{\gamma,\beta,\alpha} \xrightarrow{\circ} \sfP^{n+}\binom{\gamma\circ\beta\circ\alpha}{\gamma,\beta,\alpha}
\]
in $\sfP^{n+}$, which is defined by graph substitution \eqref{p+verticalcomp}.  Under this vertical composition, we have
\begin{equation}
\label{Gequality}
G_{\gamma\circ(\beta\circ\alpha)} \circ (\ttone_\gamma \otimes G_{\beta\circ\alpha}) = G_{\gamma\circ\beta\circ\alpha},
\end{equation}
where
\[
\setlength{\unitlength}{.6mm}
\begin{picture}(40,35)(15,17)
\put(30,30){\circle*{2}}
\put(30,20){\vector(0,1){9.5}}
\put(30,30){\vector(0,1){10}}
\put(25,30){\makebox(0,0){$\gamma$}}
\put(10,30){\makebox(0,0){$\mathtt{1}_{\gamma} =$}}
\put(40,28){$\in \sfP^{n+}\dbinom{\gamma}{\gamma}$}
\end{picture}
\]
is first defined in \eqref{1alpha}.  Since \eqref{Gequality} is a reduction law in $\sfP^{n+}$, it can be represented as an element in $\sfP^{(n+1)+}$, namely, the ``rocket" $\sfP^{n+}$-decorated graph
\begin{equation}
\setlength{\unitlength}{.6mm}
\label{Gprime}
\begin{picture}(50,58)(25,10)
\put(30,30){\circle*{2}}
\put(30,50){\circle*{2}}
\put(20,20){\vector(1,1){9.5}}
\put(40,20){\vector(-1,1){9.5}}
\put(30,20){\vector(0,1){9.5}}
\put(30,50){\vector(0,1){10}}
\put(28,48){\vector(1,1){1.5}}
\put(32,48){\vector(-1,1){1.5}}
\qbezier(30,30)(23,39)(28,48)
\qbezier(30,30)(37,39)(32,48)
\put(51,29){\makebox(0,0){$\ttone_\gamma \otimes G_{\beta\circ\gamma}$}}
\put(48,49){\makebox(0,0){$G_{\gamma\circ(\beta\circ\alpha)}$}}
\put(-23,44){\makebox(0,0){$G' = G_{G_{\gamma\circ(\beta\circ\alpha)} \circ (\ttone_\gamma \otimes G_{\beta\circ\gamma})} =$}}
\put(75,42){$\in \sfP^{(n+1)+}\dbinom{G_{\gamma\circ\beta\circ\alpha}}{G_{\gamma\circ(\beta\circ\alpha)}, \ttone_\gamma \otimes G_{\beta\circ\gamma}}$.}
\end{picture}
\end{equation}
From left to right, the bottom three edges are labeled $1$, $2$, and $3$, and are decorated by $\gamma$, $\beta$, and $\alpha$, respectively.  The two mid-level edges are decorated by $\gamma$ and $\beta\circ\alpha$, respectively.  The top edge is decorated by $\gamma\circ\beta\circ\alpha$.  We will come back to this element $G' \in \elt(\sfP^{(n+1)+})$ shortly.

The $(n+1)$-dimensional $(\ttone_\gamma \otimes G_{\beta\circ\gamma})$-horn
\[
(h,g,f) \xrightarrow{?} ?
\]
in $X$ has a unique filling
\[
w \in X(\ttone_\gamma \otimes G_{\beta\circ\gamma}).
\]
Consider the $(n+2)$-dimensional $G'$-boundary
\[
(w_{h,gf}, w) \xrightarrow{?} w_{h,g,f}
\]
in $X$.  Since $X$ is a weak-$n$ $\sfP$-algebra, this $(n+2)$-dimensional boundary has a unique filling
\[
u \in X(G').
\]
Part of the vertical consistency condition for $X$ \eqref{vconsistency}, when applied to $G'$, gives the commutative square
\[
\SelectTips{cm}{10}
\xymatrix{
X(G') \ar[r]^-{\text{out}} \ar[d]_-{\text{in}} & X(G_{\gamma\circ\beta\circ\alpha}) \ar[d]^-{\text{out}}\\
X(G_{\gamma\circ(\beta\circ\alpha)}) \ar[r]^-{\text{out}} & X(\gamma\circ\beta\circ\alpha).
}
\]
Starting at the $G'$-cell $u$, the ``down-followed-by-right" composite gives $h(gf)$. The other composite gives $hgf$.  Since this square is commutative, we have proved the equality \eqref{Xtildeass1}.
\end{proof}

\subsection{Pullback weak-$n$ $\sfP$-algebras}
\label{subsec:pullback}

Fix a map
\[
\varphi \colon \sfP \to \sfQ
\]
of unital $\frakC$-colored PROPs.  Here we observe that weak-$n$ $\sfQ$-algebras have pullback weak-$n$ $\sfP$-algebra structures.

First we observe that $\varphi$ induces a functor from $\sfP$-propertopes to $\sfQ$-propertopes.

\begin{theorem}
\label{varphipropertope}
Let $\varphi \colon \sfP \to \sfQ$ be a map of unital $\frakC$-colored PROPs.  Then there exists a functor
\[
\Phi \colon \bP(\sfP) \to \bP(\sfQ)
\]
such that:
\begin{enumerate}
\item
For $\alpha \in \elt(\sfP)$, one has
\begin{equation}
\label{Phi=phi}
\Phi(\alpha) = \varphi(\alpha).
\end{equation}
\item
For
\[
\gamma \in \sfP^{n+}\binom{\beta_1, \ldots , \beta_s}{\alpha_1, \ldots , \alpha_r}
\]
with $n \geq 0$, one has
\begin{equation}
\label{Phigamma}
\Phi(\gamma) \in \sfQ^{n+}\binom{\Phi(\beta_1), \ldots , \Phi(\beta_s)}{\Phi(\alpha_1), \ldots , \Phi(\alpha_r)}.
\end{equation}
\item
For $\gamma, \gamma' \in \sfP^{n+}$ with $n \geq 0$ and permutations $\sigma \in \Sigma_s$ and $\tau \in \Sigma_r$, one has
\begin{equation}
\label{Phioperations}
\begin{split}
\Phi(\gamma \otimes \gamma') &= \Phi(\gamma) \otimes \Phi(\gamma'),\\
\Phi(\gamma \circ \gamma') &= \Phi(\gamma) \circ \Phi(\gamma'),\\
\Phi(\sigma\gamma\tau) &= \sigma\Phi(\gamma)\tau.
\end{split}
\end{equation}
\end{enumerate}
\end{theorem}

\begin{proof}
Since both $\sfP$ and $\sfQ$ are $\frakC$-colored PROPs, $0$-dimensional $\sfP$-propertopes are exactly the $0$-dimensional $\sfQ$-propertopes.  Thus, we can define
\[
\Phi(c) = c
\]
for $c \in \frakC = \elt(\sfP^{(-1)+})$.

Going one dimensional higher, for $\alpha \in \sfP\binom{\ud}{\uc}$, we define
\[
\Phi(\alpha) = \varphi(\alpha) \in \sfQ\binom{\ud}{\uc} = \sfQ\binom{\Phi(\ud)}{\Phi(\uc)}.
\]
Here for $\ud = (d_1, \ldots , d_m)$, we used (and will use) the shorthand
\[
\Phi(\ud) = (\Phi(d_1), \ldots , \Phi(d_m)).
\]
In particular, \eqref{Phi=phi} and the $n = 0$ cases of \eqref{Phigamma} and \eqref{Phioperations} all hold.  If $h_i \colon \alpha \to b_i$ is a face map out of $\alpha$ in $\bP(\sfP)$, then
\[
\Phi(h_i) \colon \Phi(\alpha) = \varphi(\alpha) \to b_i = \Phi(b_i)
\]
is the corresponding face map out of $\Phi(\alpha)$ in $\bP(\sfQ)$.  We extend $\Phi$ to higher dimensional $\sfP$-propertopes and face maps by induction.

Suppose that $n \geq 1$.  Inductively, suppose we have defined $\Phi$ on the subcategory of $\bP(\sfP)$ consisting of all the $k$-dimensional $\sfP$-propertopes for $k \leq n$ such that \eqref{Phi=phi} -- \eqref{Phioperations} are satisfied in these dimensions.  Let $\gamma \in \elt(\sfP^{n+})$ be an $(n+1)$-dimensional $\sfP$-propertope as in the statement of this Theorem.  By the construction of $\sfP^{n+} = (\sfP^{(n-1)+})^+$, we have
\[
\gamma = (G_1, \ldots , G_l),
\]
where each $G_i$ is a $\sfP^{(n-1)+}$-decorated graph.

Define $\Phi(\gamma)$ using \textbf{decoration replacement} as follows.  In $G_i$, if a typical vertex $u \in v(G_i)$ has decoration
\[
\xi(u) = \alpha \in \sfP^{(n-1)+}\binom{\uepsilon}{\udelta},
\]
then, using the induction hypothesis, we replace this decoration by
\[
\Phi(\alpha) \in \sfQ^{(n-1)+}\binom{\Phi(\uepsilon)}{\Phi(\udelta)}.
\]
We know that the edges of $u$ must be decorated by the $\varepsilon$'s and the $\delta$'s.  We replace these edge decorations in $G_i$ by the $\Phi(\varepsilon)$'s and the $\Phi(\delta)$'s accordingly.  This decoration replacement process is performed on all the vertices and edges in $G_i$.  Denote the result by $\Phi(G_i)$.  It is easy to see that $\Phi(G_i)$ is, in fact, a $\sfQ^{(n-1)+}$-decorated graph whose vertex decorations are the $\Phi(\xi(u))$ for $u \in v(G_i)$.

Observe that by the induction hypothesis again, we have
\[
\Phi(\gamma) := \left(\Phi(G_1), \ldots , \Phi(G_l)\right) \in \sfQ^{n+}\binom{\Phi(\ubeta)}{\Phi(\ualpha)}.
\]
Thus, \eqref{Phigamma} is satisfied.  The condition \eqref{Phioperations} is also satisfied.  In fact, the horizontal composition $\otimes$ in $\sfP^{n+}$ (resp. $\sfQ^{n+}$) is defined as splicing together two sequences of $\sfP^{(n-1)+}$-decorated (resp. $\sfQ^{(n-1)+}$-decorated) graphs.  In particular, we have the equality
\[
\Phi(\gamma \otimes \gamma') = \Phi(\gamma) \otimes \Phi(\gamma')
\]
because decoration replacement commutes with splicing sequences of graphs.  The other two equalities in \eqref{Phioperations} follow by the same reasoning.

Finally, if $h_i \colon \gamma \to \kappa$ is a face map out of $\gamma$ in $\bP(\sfP)$, then, using \eqref{Phigamma}, $\Phi(h_i)$ is the corresponding face map out of $\Phi(\gamma)$ in $\bP(\sfQ)$.  This finishes the induction and proves the Theorem.
\end{proof}

Consider a $\sfQ$-propertopic set $X \in \propsetQ$.  Using the functor $\Phi$ in Theorem ~\ref{varphipropertope}, we obtain the pullback $\sfP$-propertopic set $\varphi^*(X) \in \propset$, which is defined as the composite
\[
\bP(\sfP) \xrightarrow{\Phi} \bP(\sfQ) \xrightarrow{X} \set.
\]
We thus have a pullback functor
\begin{equation}
\label{pullbackfunctor}
\varphi^* \colon \propsetQ \to \propset.
\end{equation}

\begin{corollary}
\label{cor1:pullback}
The pullback functor $\varphi^*$ has both a left adjoint and a right adjoint.  In particular, $\varphi^*$ is an exact functor.
\end{corollary}

\begin{proof}
Since $\bP(\sfP)$ is a small category and $\set$ is complete and cocomplete, the pullback functor $\varphi^*$ has both a left Kan extension and a right Kan extension. These Kan extensions are the left and the right adjoints of $\varphi^*$ \cite[p.239]{maclane2}.  A functor that has both a left adjoint and a right adjoint is automatically exact.
\end{proof}

Recall from Definition ~\ref{def:weakPalgebra} that a weak-$n$ $\sfP$-algebra is a $\sfP$-propertopic set in which certain horns and boundaries have (unique) fillings.  The category $\algn(\sfP)$ of weak-$n$ $\sfP$-algebras is a full subcategory of the category $\propset$ of $\sfP$-propertopic sets.

\begin{corollary}
\label{cor2:pullback}
Let $\varphi \colon \sfP \to \sfQ$ be a map of unital $\frakC$-colored PROPs.  Then the pullback functor $\varphi^*$ \eqref{pullbackfunctor} restricts to a functor
\[
\varphi^* \colon \algn(\sfQ) \to \algn(\sfP)
\]
for any $n$ in the range $0 \leq n \leq \infty$.
\end{corollary}

\begin{proof}
Let $X \in \algn(\sfQ)$ be a weak-$n$ $\sfQ$-algebra.  If $\gamma \in \sfP^{(k-1)+}\binom{\ubeta}{\ualpha}$, then
\[
\Phi(\gamma) \in \sfQ^{(k-1)+}\binom{\Phi(\ubeta)}{\Phi(\ualpha)}
\]
by \eqref{Phigamma}.  Moreover, we have
\[
\varphi^*(X)(\gamma) = X(\Phi(\gamma))
\]
by the definition of the pullback functor $\varphi^*$.  In particular, any $k$-dimensional $\gamma$-horn
\[
(y_1, \ldots , y_r) \xrightarrow{?} ?
\]
in $\varphi^*(X) \in \propset$ can also be regarded as a $k$-dimensional $\Phi(\gamma)$-horn in $X$, and vice versa.  The same remark applies to boundaries instead of horns.  Then it follows from the existence of (unique) fillings of horns and boundaries in $X$ that $\varphi^*(X)$ is a weak-$n$ $\sfP$-algebra.
\end{proof}

For a weak-$n$ $\sfQ$-algebra $X \in \algn(\sfQ)$, we call $\varphi^*(X) \in \algn(\sfP)$ the \textbf{underlying weak-$n$ $\sfP$-algebra} of $X$.

\section{Higher dimensional algebras for applications}
\label{sec;somehighercat}

The purpose of this section is to point out several weak-$n$ $\sfP$-algebras (Definition ~\ref{def:weakPalgebra}) that should be relevant in various applications of our theory of higher dimensional algebras.  We do not do much more than giving the basic definitions.  Deeper understanding of some of the concepts defined below requires much further work.

In \S\ref{subsec:hdcat} we consider \emph{higher category theory}.  By choosing the PROP $\sfP$ appropriately, we define weak $n$-categories, bicommutative bimonoidal weak $n$-categories, $2$-fold monoidal weak $n$-categories, and weak $n$ versions of polycategories.  In particular, if $\sfP$ is a unital $1$-colored PROP, then every weak-$n$ $\sfP$-algebra has an underlying weak $n$-category via a pullback functor.  Likewise, every bicommutative bimonoidal weak $n$-category has an underlying (trivial) weak-$n$ $\sfP$-algebra via a pullback functor (Corollary ~\ref{underlyingweakncat}).

In \S\ref{subsec:hdtft} we consider \emph{higher topological field theories}.  We take $\sfP$ to be the Segal PROP $\Se$ considered in Example ~\ref{ex:CFT}.  By first applying a suitable homology functor, we define weak $n$ versions of Cohomological Field Theories-$I$ and Topological Quantum Field Theories.

In \S\ref{subsec:nstack} we consider \emph{higher algebraic geometry} by defining weak $n$ versions of stacks.

Throughout the rest of this section, let $n$ be in the range $0 \leq n \leq \infty$, unless otherwise specified.

\subsection{Higher category theory}
\label{subsec:hdcat}

Here we consider a few concepts regarding higher category theory.

In Example ~\ref{initialprop} we considered the initial unital $1$-colored PROP $\sfI$, whose category of algebras is isomorphic to $\set$.  Thus, it makes sense to make the following definition.

\begin{definition}
\label{def:weakncat}
A \textbf{weak $n$-category} is defined as a weak-$n$ $\sfI$-algebra, where $\sfI$ is the initial unital $1$-colored PROP in $\set$.  A \textbf{morphism} of weak $n$-categories is defined as a morphism of weak-$n$ $\sfI$-algebras.
\end{definition}

In some sense, weak $n$-categories are the simplest kinds of weak-$n$ $\sfP$-algebras because $\sfI$ is the initial object among the unital $1$-colored PROPs.

In Example ~\ref{terminalprop} we considered the unital $1$-colored PROP $\sfT$, which is the terminal object among all the $1$-colored PROPs.  The $\sfT$-algebras are the bicommutative bimonoids.

\begin{definition}
\label{def:bicomweakncat}
A \textbf{bicommutative bimonoidal weak $n$-category} is defined as a weak-$n$ $\sfT$-algebra, where $\sfT$ is the terminal $1$-colored PROP in $\set$.  A \textbf{morphism} of bicommutative bimonoidal weak $n$-categories is defined as a morphism of weak-$n$ $\sfT$-algebras.
\end{definition}

If $\sfP$ is an arbitrary unital $1$-colored PROP in $\set$, then there are unique maps
\[
\sfI \xrightarrow{\iota} \sfP \xrightarrow{\tau} \sfT
\]
of unital $1$-colored PROPs.  Using the pullback functor $\varphi^*$ in Corollary ~\ref{cor2:pullback} (with $\varphi = \iota$ or $\tau$), we obtain the following consequences.

\begin{corollary}
\label{underlyingweakncat}
Let $\sfP$ be an arbitrary unital $1$-colored PROP in $\set$.  Then:
\begin{enumerate}
\item
Every weak-$n$ $\sfP$-algebra has an underlying weak $n$-category.
\item
Every bicommutative bimonoidal weak $n$-category has an underlying weak-$n$ $\sfP$-algebra.
\end{enumerate}
\end{corollary}

In Example ~\ref{T+}, we considered the unital $\frakC$-colored PROP $\sfT_\frakC^+$, whose algebras are exactly the $\frakC$-colored PROPs.  Recall that $\sfT_\frakC$ is the terminal object among all the $\frakC$-colored PROPs.

\begin{definition}
A \textbf{weak-$n$ $\frakC$-colored PROP} is defined as a weak-$n$ $\sfT_\frakC^+$-algebra.  A \textbf{morphism} of weak-$n$ $\frakC$-colored PROPs is defined as a morphism of weak-$n$ $\sfT_\frakC^+$-algebras.
\end{definition}

Recall that to a polycategory $C$, one can associate an $Ob(C)$-colored PROP that determines the polycategory $C$ (Example ~\ref{ex:polycategory}).  Thus, one can think of weak-$n$ $\frakC$-colored PROPs, for different sets $\frakC$, as (containing the) $n$-time categorified polycategories, or \textbf{weak $n$-polycategories}.

In Example ~\ref{I+}, we observed that $\sfI^+$-algebras are exactly the (bi-equivariant graded) monoidal monoids.  We also discussed that monoidal monoids are  (bi-equivariant graded) de-categorified versions of $2$-fold monoidal categories \cite{bfsv}.

\begin{definition}
\label{def:monoidalmonoidal}
A \textbf{monoidal monoidal weak $n$-category}, or \textbf{$2$-fold monoidal weak $n$-category}, is defined as a weak-$n$ $\sfI^+$-algebra.  A \textbf{morphism} of monoidal monoidal weak $n$-categories is defined as a morphism of weak-$n$ $\sfI^+$-algebras.
\end{definition}

\subsection{Higher topological field theories}
\label{subsec:hdtft}

The next two definitions have to do with \emph{higher topological field theories}.

In Example ~\ref{ex:CFT} we discussed the Segal PROP $\Se$, which is a $1$-colored topological PROP.  We also noted that there is an obvious colored version of $\Se$, in which the boundary holes are allowed to have different circumferences.  Here, as in Example ~\ref{ex:CFT}, to simplify the discussion we only consider the $1$-colored version of $\Se$.

Recall from \eqref{H*Se} that if $H_*$ is the singular homology functor with coefficients in $\bk$, then $H_*(\Se)$ is the graded $\bk$-linear PROP for Cohomological Field Theories-$I$.

\begin{definition}
A \textbf{weak-$n$ Cohomological Field Theory-$I$} is defined as a weak-$n$ $H_*(\Se)$-algebra.  A \textbf{morphism} of weak-$n$ Cohomological Field Theories-$I$ is defined as a morphism of weak-$n$ $H_*(\Se)$-algebras.
\end{definition}

If we only take the $0$th homology, then $H_0(\Se)$ \eqref{H0Se} is the $\bk$-linear PROP for Topological Quantum Field Theories.

\begin{definition}
A \textbf{weak-$n$ Topological Quantum Field Theory}, abbreviated to \textbf{weak-$n$ TQFT}, is defined as a weak-$n$ $H_0(\Se)$-algebra.  A \textbf{morphism} of weak-$n$ TQFTs is defined as a morphism of weak-$n$ $H_0(\Se)$-algebras.
\end{definition}

One can consider weak-$n$ TQFT as one way to realize a higher dimensional version of TQFT as discussed in \cite{bd0}.

\subsection{Higher algebraic geometry}
\label{subsec:nstack}

In \cite{gro} Grothendieck suggested a higher dimensional version of stacks, or $n$-stacks.  The case $n = 2$ was considered by Breen \cite{breen}.  More generally, using Tamsamani's definition of weak $n$-category \cite{tam} (for $n < \infty$), Simpson \cite{simpson} discussed a notion of $n$-stacks as a parametrized family of Tamsamani's weak $n$-categories.  Here we suggest our own naive concept of $n$-stacks as a parametrized family of weak-$n$ $\sfP$-algebras.

If $X$ is a category, then a \emph{stack on $X$} is a sheaf of groupoids on $X$ satisfying some descent conditions.  So a stack on $X$ is a well-behaved functor
\[
F \colon X^{op} \to \mathbf{Gpd},
\]
where $\mathbf{Gpd}$ denotes the category of groupoids.  One way to fit stacks into our theory of higher dimensional algebras is as follows.

A groupoid is a category in which all the morphisms are invertible.  So a categorified generalization of it is a weak-$n$ $\sfP$-algebra, where $\sfP$ is any unital $\frakC$-colored PROP.  Thus, we should replace the category $\mathbf{Gpd}$ of groupoids by the category $\algn(\sfP)$ of weak-$n$ $\sfP$-algebras.  We take the notion of \emph{well-behaved} to mean \emph{fibrant} (Definition ~\ref{def:fibrantpropset}).

\begin{definition}
Let $X$ be a category and $\sfP$ be a unital $\frakC$-colored PROP.  A \textbf{weak-$n$ $\sfP$-stack on $X$} is defined as a functor
\[
F \colon X^{op} \to \algn(\sfP)
\]
such that $F(x)$ is a fibrant $\sfP$-propertopic set for each object $x$ in $X$.  A \textbf{morphism} of weak-$n$ $\sfP$-stacks on $X$ is a natural transformation of such functors.  Denote the category of weak-$n$ $\sfP$-stacks on $X$ by $\stack(\sfP,n,X)$.
\end{definition}

So a weak-$n$ $\sfP$-stack on $X$ is an $X^{op}$-diagram of fibrant weak-$n$ $\sfP$-algebras.  Recall from Corollary ~\ref{fibrantpropset} that a $\sfP$-propertopic set $Y$ is \emph{fibrant} if and only if every horn in $Y$ has a filling.  Another way to say it is that $Y$ is fibrant if and only if $Y$ is a weak-$\omega$ $\sfP$-algebra (Proposition ~\ref{weak0algebra}).  In particular, when $n = \infty$, a \textbf{weak-$\omega$ $\sfP$-stack on $X$} is exactly a functor
\[
F \colon X^{op} \to \alg^\infty(\sfP).
\]
In other words, we have
\[
\stack(\sfP,\infty,X) = \left(\alg^\infty(\sfP)\right)^{X^{op}}.
\]

Let $\varphi \colon \sfP \to \sfQ$ be a map of unital $\frakC$-colored PROPs.  Then there is a pullback functor
\[
\varphi^* \colon \algn(\sfQ) \to \algn(\sfP)
\]
for each $n$ in the range $0 \leq n \leq \infty$ (Corollary ~\ref{cor2:pullback}).  It follows that there is a pullback functor
\[
\varphi^* \colon \stack(\sfQ,n,X) \to \stack(\sfP,n,X)
\]
at the level of stacks.

A substantial piece of work on higher algebraic geometry is Lurie's book \cite{lurie}, which is based on one version of weak $\omega$-categories, called \emph{quasicategories}.  It would be nice to generalize Lurie's work to weak-$\omega$ $\sfP$-algebras for an arbitrary unital colored PROP $\sfP$.


\end{document}